\documentclass{amsart}

\usepackage{amsmath,amscd,amssymb}
\usepackage{mathtools}
\usepackage{hyperref}
\usepackage{tikz-cd}

\usepackage{mathbbol}
\DeclareSymbolFontAlphabet{\mathbb}{AMSb}
\DeclareSymbolFontAlphabet{\mathbbl}{bbold}

\usepackage{amsthm}
\numberwithin{equation}{section}

\title{Relative Serre functor for comodule algebras}
\author[K.~Shimizu]{Kenichi Shimizu}
\email{kshimizu@shibaura-it.ac.jp}
\address{Department of Mathematical Sciences \\
  Shibaura Institute of Technology \\
  307 Fukasaku, Minuma-ku, Saitama-shi, Saitama 337-8570, Japan.}
\date{}

\newtheorem{counter}{}[section]
\theoremstyle{definition}
\newtheorem{definition}         [counter]{Definition}

\theoremstyle{plain}
\newtheorem{lemma}              [counter]{Lemma}
\newtheorem{claim}              [counter]{Claim}

\newtheorem{proposition}        [counter]{Proposition}
\newtheorem{theorem}            [counter]{Theorem}
\newtheorem{corollary}          [counter]{Corollary}

\newtheorem*{theorem*}          {Theorem}
\theoremstyle{remark}
\newtheorem{remark}             [counter]{Remark}

\newcommand{\id}{\mathrm{id}}
\newcommand{\eval}{\mathrm{ev}}
\newcommand{\coev}{\mathrm{coev}}
\newcommand{\op}{\mathrm{op}}
\newcommand{\rev}{\mathrm{rev}}
\newcommand{\unitobj}{\mathbbl{1}}
\newcommand{\Hom}{\mathrm{Hom}}

\newcommand{\radj}{\mathrm{ra}}
\newcommand{\rradj}{\mathrm{rra}}

\newcommand{\iHom}{\underline{\mathrm{Hom}}}
\newcommand{\iEnd}{\underline{\mathrm{End}}}

\newcommand{\ieval}{\underline{\mathrm{ev}}}
\newcommand{\icoev}{\underline{\mathrm{coev}}}
\newcommand{\icomp}{\underline{\mathrm{com\smash{\mathrm{p}}}}}
\newcommand{\catactl}{\mathbin{\triangleright}}
\newcommand{\catactr}{\mathbin{\triangleleft}}
\newcommand{\DD}{\mathbb{D}} 
\newcommand{\Yone}{\mathbb{Y}}
\newcommand{\Mod}{\mathfrak{M}}
\newcommand{\Ser}{\mathbb{S}}
\newcommand{\Nak}{\mathbb{N}}
\newcommand{\EvalAtOne}{\mathbb{E}}
\newcommand{\itrace}{\underline{\mathrm{tr}}}
\newcommand{\profto}{\relbar\joinrel\mapstochar\joinrel\rightarrow}
\newcommand{\bfk}{\Bbbk} 

\newcommand{\FUN}{\mathrm{Fun}}
\newcommand{\PROF}{\mathrm{Prof}}
\newcommand{\REX}{\mathrm{Rex}}
\newcommand{\LEX}{\mathrm{Lex}}
\newcommand{\Sets}{\mathbf{Set}}
\newcommand{\piv}{\mathrm{piv}}
\newcommand{\Kdelta}{\boldsymbol{\delta}}
\newcommand{\inner}{\mathrm{inn}}
\begin{document}

\begin{abstract}
  Let $\mathcal{C}$ be a finite tensor category, and let $\mathcal{M}$ be an exact left $\mathcal{C}$-module category. The relative Serre functor of $\mathcal{M}$ is an endofunctor $\Ser$ on $\mathcal{M}$ together with a natural isomorphism $\underline{\mathrm{Hom}}(M, N)^* \cong \underline{\mathrm{Hom}}(N, \Ser(M))$ for $M, N \in \mathcal{M}$, where $\underline{\mathrm{Hom}}$ is the internal Hom functor of $\mathcal{M}$. In this paper, we discuss the case where $\mathcal{C}$ and $\mathcal{M}$ are the category of finite-dimensional left modules over a finite-dimensional Hopf algebra $H$ and a finite-dimensional left $H$-comodule algebra $L$, respectively. We give an explicit description of the relative Serre functor of $\mathcal{M}$ and its twisted module structure in terms of the Frobenius structure of $L$. We also study pivotal structures on $\mathcal{M}$ and give some concrete examples.
\end{abstract}
 
\maketitle

\setcounter{tocdepth}{2}
\tableofcontents

\section{Introduction}
\label{sec:introduction}

Finite tensor categories \cite{MR2119143} are a widely studied class of tensor categories including fusion categories and representation categories of finite-dimensional Hopf algebras.
One of the important subjects in this research area is module categories, that is, categories on which finite tensor categories act.
In this paper, we are interested in relative Serre functors for module categories, which was originally introduced by Schaumann under the name of the Serre equivalence in \cite{MR3435098}.
We shall recall its definition:
Let $\mathcal{C}$ be a finite tensor category, and let $\mathcal{M}$ be an exact left $\mathcal{C}$-module category with action $\catactl: \mathcal{C} \times \mathcal{M} \to \mathcal{M}$ (see \cite{MR3242743} or Subsection \ref{subsec:rel-Serre} for the precise definition).
A relative Serre functor of $\mathcal{M}$ is an endofunctor $\Ser$ on $\mathcal{M}$ together with a natural isomorphism
\begin{equation}
  \label{eq:Intro-rel-Serre-1}
  \underline{\mathrm{Hom}}(M, N)^* \cong \underline{\mathrm{Hom}}(N, \Ser(M))
  \quad (M, N \in \mathcal{M}),
\end{equation}
where $\underline{\mathrm{Hom}}: \mathcal{M}^{\op} \times \mathcal{M} \to \mathcal{C}$ is the internal Hom functor of $\mathcal{M}$ and $(-)^*$ is the duality functor (see Section~\ref{sec:rel-Serre} for detail).

We recall that a pivotal structure of $\mathcal{C}$ is an isomorphism $\mathfrak{p} : \id_{\mathcal{C}} \to (-)^{**}$ of tensor functors \cite{MR3242743}.
A pivotal structure of a tensor category is often required in the study of tensor categories and their applications.
For example, a pivotal structure is a prerequisite for the semisimplification construction \cite{MR4486913} of tensor categories.
A ribbon category, which plays a fundamental role in some constructions of topological invariants of knots and 3-manifolds, is necesarily a rigid monoidal category equipped with a braiding and a pivotal structure.

Since the relative Serre functor of $\mathcal{M} = \mathcal{C}$ is the double dual functor $(-)^{**}$, it may be natural to define a pivotal structure of a module category as a `trivialization' of the relative Serre functor. A precise mathematical formulation has been given by Schaumann \cite{MR3019263} as follows: The functor $\Ser$ has a canonical structure
\begin{equation}
  \label{eq:Intro-rel-Serre-2}
   X^{**} \catactl \Ser(M) \to \Ser(X \catactl M)
  \quad (X \in \mathcal{C}, M \in \mathcal{M})
\end{equation}
of a `twisted' $\mathcal{C}$-module functor. If $\mathcal{C}$ has a pivotal structure $\mathfrak{p}$, then we can make $\Ser$ an ordinary $\mathcal{C}$-module functor by composing \eqref{eq:Intro-rel-Serre-2} with $\mathfrak{p}_X \catactl \id_{\Ser(M)}$. A {\em pivotal structure} of $\mathcal{M}$ is an isomorphism $\id_{\mathcal{M}} \to \Ser$ of $\mathcal{C}$-module functors.

As in the case of tensor categories, the existence of a pivotal structure on a module category implies some useful consequences.
Suppose that $\mathcal{M}$ admits a pivotal structure.
Then one can define a well-behaved `module trace' for endomorphisms in $\mathcal{M}$ \cite{MR3943747,MR3435098,MR3019263}.
Furthermore, the dual tensor category of $\mathcal{C}$ with respect to $\mathcal{M}$ also has a pivotal structure \cite[Corollary 5.9]{MR3435098}.
In this paper, we will also show that $\iHom(M, M)$ is a symmetric Frobenius algebra in $\mathcal{C}$ for every $M \in \mathcal{M}$ (see \cite{MR3019263} and Subsection~\ref{subsec:symmetric-Frobenius-algebras}). Hence an exact left $\mathcal{C}$-module category admitting a pivotal structure is equivalent to the category of modules over a symmetric Frobenius algebra in $\mathcal{C}$ (Corollary~\ref{cor:sym-Frobenius-alg}).

In view of these results, we expect that a pivotal structure of an exact module category could be a fundamental tool in the study of finite tensor categories and their module categories. While the relative Serre functor can be expressed by the Nakayama functor and a categorical analogue of the modular function \cite[Theorem 4.26]{MR4042867}, explicit descriptions of the natural isomorphisms~\eqref{eq:Intro-rel-Serre-1} and~\eqref{eq:Intro-rel-Serre-2} are not known in most concrete cases. Hence, in general, it is difficult to determine whether a given module category admits a pivotal structure.

Given an algebra $A$, we denote by ${}_A \Mod$ the category of finite-dimensional left $A$-modules. The main purpose of this paper is to investigate the relative Serre functor in the case where $\mathcal{C} = {}_H \Mod$ and $\mathcal{M} = {}_L \Mod$ for some finite-dimensional Hopf algebra $H$ and some finite-dimensional left $H$-comodule algebra $L$ such that ${}_L \Mod$ is an exact module category over ${}_H \Mod$. We give an explicit description of a relative Serre functor of ${}_L \Mod$ and the natural isomorphisms \eqref{eq:Intro-rel-Serre-1} and \eqref{eq:Intro-rel-Serre-2} in terms of the Frobenius structure of $L$. We also discuss when ${}_L \Mod$ admits a pivotal structure and give some concrete examples of pivotal module categories over ${}_H \Mod$.

\subsection{Organization of this paper}

This paper is organized as follows:
In Section \ref{sec:prelim}, we collect some basic terminology and results related to monoidal categories, their modules categories and internal Hom functors. We also introduce the notion of module profunctors (Definition~\ref{def:C-prof}). This is a useful tool for dealing with adjoints of module functors. 

In Section \ref{sec:rel-Serre}, we recall some results on relative Serre functors from \cite{MR4042867,MR3435098} and add new observations.
Let $\mathcal{C}$ be a finite tensor category, and let $\mathcal{M}$ be an exact left $\mathcal{C}$-module category.
Given two finite left $\mathcal{C}$-module categories $\mathcal{A}$ and $\mathcal{B}$, we denote by $\REX_{\mathcal{C}}(\mathcal{A}, \mathcal{B})$ the category of linear right exact $\bfk$-left $\mathcal{C}$-module functors from $\mathcal{A}$ to $\mathcal{B}$.
It has been noted in the proof of \cite[Theorem 4.25]{MR4042867} that a relative Serre functor of $\mathcal{M}$ is given by the composition
\begin{equation}
  \label{eq:intro-standard-rel-Serre}
  \mathcal{M}
  \xrightarrow{\quad \Yone \quad}
  \REX_{\mathcal{C}}(\mathcal{M}, \mathcal{C})^{\op}
  \xrightarrow{\quad (-)^{\radj} \quad}
  \REX_{\mathcal{C}}(\mathcal{C}, \mathcal{M})
  \xrightarrow{\quad \EvalAtOne \quad}
  \mathcal{M},
\end{equation}
where $\Yone$ is the internal Yoneda functor $\Yone(M) = \iHom(M, -)$, $(-)^{\radj}$ is the functor given by taking right adjoints and $\EvalAtOne(F) = F(\unitobj)$. We call \eqref{eq:intro-standard-rel-Serre} the `standard' relative Serre functor in this paper. We observe that the natural isomorphisms \eqref{eq:Intro-rel-Serre-1} and \eqref{eq:Intro-rel-Serre-2} are written in somewhat explicit form if $\Ser$ is the standard relative Serre functor (Theorem~\ref{thm:mod-str-standard-rel-Serre-1}).

The notion of a pivotal structure of an exact module category (Definition~\ref{def:pivotal-module-cat}) originates from \cite{MR3435098}.
After reviewing known results on a pivotal structure of a module category, we give the following new result: Suppose that $\mathcal{C}$ is pivotal and $\mathcal{M}$ admits a pivotal structure. Then, for every $M \in \mathcal{M}$, the algebra $\iHom(M, M)$ in $\mathcal{C}$ is symmetric Frobenius in the sense of \cite{MR2500035} (Theorem~\ref{thm:sym-Frobenius-alg}).
As a consequence, every pivotal left $\mathcal{C}$-module category is equivalent to the category of modules over a symmetric Frobenius algebra in $\mathcal{C}$ (Corollary~\ref{cor:sym-Frobenius-alg}).

Now let $H$ be a finite-dimensional Hopf algebra. It is known that every finite left module category over $\mathcal{C} := {}_H \Mod$ is equivalent to $\mathcal{M} := {}_L \Mod$ for some finite-dimensional left $H$-comodule algebra $L$ \cite{MR2331768}.
We fix such an algebra $L$ and assume that ${}_L \Mod$ is an exact module category over ${}_H \Mod$.
In Section~\ref{sec:rel-Serre-comod-alg}, we investigate relative Serre functors and natural isomorphisms \eqref{eq:Intro-rel-Serre-1} and \eqref{eq:Intro-rel-Serre-2} in the case where $\mathcal{C} = {}_H \Mod$ and $\mathcal{M} = {}_L \Mod$.
It turns out that $L$ is a Frobenius algebra in this case (Lemma~\ref{lem:exact-comod-alg-Frobenius}).
We discuss what module-theoretic counterparts of categories and functors appearing in \eqref{eq:intro-standard-rel-Serre} are. As a result, when $\mathcal{M}$ is exact, we see that there is a relative Serre functor $\Ser_2: \mathcal{M} \to \mathcal{M}$ such that
\begin{equation*}
  \Ser_2(M) = \Hom_H((H \catactl M)^*, \bfk)
  \quad (M \in \mathcal{M})
\end{equation*}
as a vector space and the natural isomorphisms \eqref{eq:Intro-rel-Serre-1} and \eqref{eq:Intro-rel-Serre-2} are written explicitly for this realization of the relative Serre functor (see Subsection~\ref{subsec:comod-alg-rel-Serre}). The integral theory for Hopf algebras gives an isomorphism $M \cong \Ser_2(M)$ of vector spaces. 
By transporting the structure morphisms via this isomorphism, we obtain a relative Serre functor $\Ser$ of $\mathcal{M}$ with the following properties: As a vector space, $\Ser(M) = M$. The action of $L$ on $\Ser(M)$ is given by twisting the original action of $L$ on $M$ by the automorphism \eqref{eq:twisted-Nakayama} of $L$ written by the Nakayama automorphism of $L$ and the modular function on $H$. The natural isomorphisms \eqref{eq:Intro-rel-Serre-1} and \eqref{eq:Intro-rel-Serre-2} are written explicitly by using the Frobenius system of $L$ (Theorem~\ref{thm:main-theorem}).

Our result is too complicated to be described in Introduction and may not be useful in practical applications. However, we find that our formula of \eqref{eq:Intro-rel-Serre-2} reduces to a simple form if $L$ admits a Frobenius form satisfying a certain equation like the defining formula of cointegrals on Hopf algebras. In this case, for $X \in \mathcal{C}$ and $M \in {}_{L} \Mod$, the isomorphism \eqref{eq:Intro-rel-Serre-2} is given by 
\begin{equation*}
  \Ser(X \catactl M) \to X^{**} \catactl \Ser(M),
  \quad x \otimes m \mapsto \Phi_X(g_L^{-1} g_H^{} x) \otimes m,
\end{equation*}
where $\Phi_X : X \to X^{**}$ is the canonical isomorphism of vector spaces, $g_H \in H$ is the distinguished grouplike element of $H$ (see \eqref{eq:distinguished-grouplike} for our definition), and $g_L \in H$ is a certain grouplike element defined in a similar way as $g_H$ but by using the Frobenius form of $L$ (Theorem~\ref{thm:main-theorem-2}).
We also give some remarks and examples in the case where $L$ is a Hopf subalgebra of $H$.

In Section~\ref{sec:examples}, for some left comodule algebras $L$ given by Mombelli \cite{MR2678630}, we determine whether the relative Serre functor of ${}_L \Mod$ is isomorphic to the identity functor, whether the Nakayama functor of ${}_L \Mod$ is isomorphic to the identity functor, and whether ${}_L \Mod$ has a pivotal structure. Fortunately, the simpler form of the main result, Theorem~\ref{thm:main-theorem-2}, can be applied to these comodule algebras. The results in this section are summarized as Tables \ref{tab:Taft-comod-alg-rel-Serre} and~\ref{tab:book-comod-alg-rel-Serre}.

\subsection*{Acknowledgment}

The author thanks V.~Koppen, G.~Schaumann and T.~Shibata for comments and discussion on earlier version of this paper.
The author also thanks the anonymous referee for careful reading of the manuscript.
This work is supported by JSPS KAKENHI Grant Numbers JP16K17568 and JP20K03520.

\subsection*{Competing interests}

The author declares no conflict of interest.

\section{Preliminaries}
\label{sec:prelim}

In this section, we recall basic terminologies and results on monoidal categories, their modules categories and internal Hom functors.
Throughout this paper, adjoints of module functors are essential.
In Subsection~\ref{subsec:adjunction}, we discuss the uniqueness property of an adjoint functor and introduce the natural isomorphism \eqref{eq:right-adj-uniqueness}, which will be used to define some module functors in the sequel.
We also introduce the notion of module profunctors (Definition~\ref{def:C-prof}) and provide fundamental lemmas on module profunctors, which are useful for dealing with adjoints of module functors.

\subsection{Adjunction}
\label{subsec:adjunction}

Let $\mathcal{A}$ and $\mathcal{B}$ be categories.
An {\em adjunction} \cite[IV.1]{MR1712872} from $\mathcal{A}$ to $\mathcal{B}$ is a quadruple $(F, G, \eta, \varepsilon)$ consisting of two functors $F: \mathcal{A} \to \mathcal{B}$ and $G: \mathcal{B} \to \mathcal{A}$ and two natural transformations $\eta : \id_{\mathcal{A}} \to G F$ and $\varepsilon : F G \to \id_{\mathcal{B}}$, called the {\em unit} and the {\em counit} of the adjunction, respectively, satisfying the zig-zag identities. If $(F, G, \eta, \varepsilon)$ is an adjunction, then the natural transformation
\begin{equation}
  \label{eq:adjunction-iso}
  \Phi_{X,Y}: \Hom_{\mathcal{B}}(F(X), Y) \to \Hom_{\mathcal{A}}(X, G(Y)),
  \quad f \mapsto G(f) \circ \eta_X
\end{equation}
is bijective with the inverse given by $\Phi_{X,Y}^{-1}(g) = \varepsilon_{Y} \circ F(g)$. The natural isomorphism $\Phi$ is called the {\em adjunction isomorphism} for $(F, G, \eta, \varepsilon)$.

Let $F : \mathcal{A} \to \mathcal{B}$ and $G : \mathcal{B} \to \mathcal{A}$ be functors. If $(F, G, \eta, \varepsilon)$ is an adjunction for some $\eta$ and $\varepsilon$, then we write $F \dashv G$ and call the functors $F$ and $G$ a {\em left adjoint} of $G$ and a {\em right adjoint} of $F$, respectively. We note that a right adjoint of a functor is unique up to isomorphism if it exists. Indeed, if both $G$ and $G'$ are right adjoint functors of $F$, then there are natural isomorphism
\begin{equation*}
  \Hom_{\mathcal{A}}(X, G(Y)) \cong \Hom_{\mathcal{B}}(F(X), Y) \cong \Hom_{\mathcal{A}}(X, G'(Y))
  \quad (X \in \mathcal{A}, Y \in \mathcal{B})
\end{equation*}
and hence $G \cong G'$ by the Yoneda lemma. We refer to this fact as the {\em uniqueness} of a right adjoint.

In this paper, we often denote a right adjoint of a functor $F$ by $F^{\radj}$ (provided that it exists). If functors $F$ and $G$ are composable and both of them admit a right adjoint, then we have a canonical isomorphism
\begin{equation}
  \label{eq:right-adj-uniqueness}
  \gamma_{F,G}: F^{\radj} \circ G^{\radj} \to (G \circ F)^{\radj}
\end{equation}
by the uniqueness of a right adjoint. The isomorphism \eqref{eq:right-adj-uniqueness} is `associative' in the following sense: If functors $F$, $G$ and $H$ are composable and all of them admit a right adjoint, then the following diagram is commutative:
\begin{equation}
  \label{eq:right-adj-uniqueness-3}
  \begin{tikzcd}[column sep = 64pt]
    F^{\radj} \circ G^{\radj} \circ H^{\radj}
    \arrow[r, "{F^{\radj} \circ \gamma_{G,H}}"]
    \arrow[d, "{\gamma_{F,G} \circ H^{\radj}}"']
    & F^{\radj} \circ (H \circ G)^{\radj}
    \arrow[d, "{\gamma_{F, H G}}"] \\
    (G \circ F)^{\radj} \circ H^{\radj}
    \arrow[r, "{\gamma_{G F, H}}"]
    & (H \circ G \circ F)^{\radj}
  \end{tikzcd}
\end{equation}

\subsection{Monoidal categories}

A {\em monoidal category} \cite[VII.1]{MR1712872} is a category $\mathcal{C}$ endowed with a functor $\otimes: \mathcal{C} \times \mathcal{C} \to \mathcal{C}$ (called the {\em tensor product}), an object $\unitobj \in \mathcal{C}$ (called the {\em unit object}) and natural isomorphisms
\begin{equation}
  \label{eq:mon-cat-constraints}
  (X \otimes Y) \otimes Z \cong X \otimes (Y \otimes Z)
  \quad \text{and} \quad
  \unitobj \otimes X \cong X \cong X \otimes \unitobj
\end{equation}
for $X, Y, Z \in \mathcal{C}$ satisfying the pentagon axiom and the triangle axiom. A monoidal category $\mathcal{C}$ is said to be {\em strict} if the natural isomorphisms \eqref{eq:mon-cat-constraints} are identities. In view of Mac Lane's strictness theorem \cite[VII.2]{MR1712872}, we assume that all monoidal categories are strict.

\subsection{Duality in a monoidal category}

Let $L$ and $R$ be objects of a monoidal category $\mathcal{C}$, and let $\varepsilon: L \otimes R \to \unitobj$ and $\eta: \unitobj \to R \otimes L$ be morphisms in $\mathcal{C}$. Following \cite{MR3242743}, we say that the triple $(L, \varepsilon, \eta)$ is a {\em left dual object} of $R$ and the triple $(R, \varepsilon, \eta)$ is a {\em right dual object} of $L$ if the equations
\begin{equation*}
  (\varepsilon \otimes \id_L) \circ (\id_L \otimes \eta) = \id_L
  \quad \text{and} \quad
  (\id_R \otimes \varepsilon) \circ (\eta \otimes \id_R) = \id_R
\end{equation*}
hold. If this is the case, then the morphisms $\varepsilon$ and $\eta$ are called the {\em evaluation} and the {\em coevaluation}, respectively.

A monoidal category $\mathcal{C}$ is said to be {\em rigid} if every object of $\mathcal{C}$ has a left dual object and a right dual object. If $\mathcal{C}$ is a rigid monoidal category, then we usually denote a left dual object and a right dual object of $X \in \mathcal{C}$ by $(X^*, \eval_X, \coev_X)$ and $({}^* \! X, \eval'_X, \coev'_X)$, respectively. It is known that the maps $X \mapsto X^*$ and $X \mapsto {}^*X$ extend to monoidal equivalences between $\mathcal{C}^{\rev}$ and $\mathcal{C}^{\op}$, where $\mathcal{C}^{\rev}$ is the category $\mathcal{C}$ equipped with the reversed tensor product $X \otimes^{\rev} Y = Y \otimes X$. By replacing $\mathcal{C}$ with an equivalent one and choosing duals in a suitable way, we may assume that the functors $(-)^*$ and ${}^*(-)$ are strict monoidal functors and mutually inverse to each other.

\subsection{Module categories}

Let $\mathcal{C}$ be a monoidal category. A {\em left $\mathcal{C}$-module category} \cite{MR3242743} is a category $\mathcal{M}$ equipped with a functor $\catactl: \mathcal{C} \times \mathcal{M} \to \mathcal{M}$ (called the {\em action} of $\mathcal{C}$) and natural isomorphisms
\begin{equation}
  \label{eq:mod-cat-assoc}
  (X \otimes Y) \catactl M \cong X \catactl (Y \catactl M)
  \quad \text{and}
  \quad \unitobj \catactl M \cong M
  \quad (X, Y \in \mathcal{C}, M \in \mathcal{M})
\end{equation}
satisfying certain axioms similar to those for monoidal categories.
A right $\mathcal{C}$-module category and a $\mathcal{C}$-bimodule category are defined in an analogous way. There is an analogue of Mac Lane's strictness theorem for $\mathcal{C}$-module categories \cite[Remark 7.2.4]{MR3242743}. Thus, for simplicity, we usually assume that the natural isomorphisms \eqref{eq:mod-cat-assoc} of a left $\mathcal{C}$-module category $\mathcal{M}$ are identity morphisms and write
\begin{equation*}
  (X \otimes Y) \catactl M = X \catactl Y \catactl M = X \catactl (Y \catactl M)
\end{equation*}
for objects $X, Y \in \mathcal{C}$ and $M \in \mathcal{M}$. A similar convention is adopted for right module categories and bimodule categories.

We note that the exactness of the action is usually required when, for example, both $\mathcal{C}$ and $\mathcal{M}$ are abelian categories ({\it cf}. \cite{MR3934626,MR3242743}).
The exactness of the action is not imposed in this section for the sake of generality.

Let $\mathcal{M}$ and $\mathcal{N}$ be left $\mathcal{C}$-module categories. A {\em lax left $\mathcal{C}$-module functor} from $\mathcal{M}$ to $\mathcal{N}$ is a pair $(F, s)$ consisting of a functor $F: \mathcal{M} \to \mathcal{N}$ and a natural transformation
\begin{equation*}
  s_{X,M}: X \catactl F(M) \to F(X \catactl M)
  \quad (X \in \mathcal{C}, M \in \mathcal{M})
\end{equation*}
such that the equations
\begin{equation}
  \label{eq:lax-C-module-functor}
  s_{\unitobj, M} = \id_{F(M)}
  \quad \text{and} \quad
  s_{X \otimes Y, M} = s_{X, Y \catactl M} \circ (\id_X \catactl s_{Y, M})
\end{equation}
hold for all objects $X, Y \in \mathcal{C}$ and $M \in \mathcal{M}$. An {\em oplax left $\mathcal{C}$-module functor} from $\mathcal{M}$ to $\mathcal{N}$ is a pair $(F, s)$ consisting of a functor $F: \mathcal{M} \to \mathcal{N}$ and a natural transformation $s: F(X \catactl M) \to X \catactl F(M)$ ($X \in \mathcal{C}$, $M \in \mathcal{M}$) satisfying equations analogous to \eqref{eq:lax-C-module-functor}. We omit the definitions of morphisms of (op)lax module functors; see \cite{MR3934626} for details.

An (op)lax left $\mathcal{C}$-module functor $(F, s)$ is said to be {\em strong} if the natural transformation $s$ is invertible. If $\mathcal{C}$ is rigid, then every (op)lax left $\mathcal{C}$-module functor is strong \cite[Lemma 2.10]{MR3934626} and hence an (op)lax $\mathcal{C}$-module functor may simply be called a $\mathcal{C}$-module functor. Nevertheless, the adjective `(op)lax' will sometimes be used in the case where $\mathcal{C}$ is rigid to clarify the natural direction of the structure morphism.

\subsection{Module profunctors}

In this paper, we deal with some functors defined as an adjoint of a module functor. For this purpose, it is convenient to introduce the notion of a {\em module profunctor}. Recall that a profunctor $T: \mathcal{M} \profto \mathcal{N}$ is just a functor $T: \mathcal{N}^{\op} \times \mathcal{M} \to \Sets$. By the Yoneda lemma, we may identify an ordinary functor $F: \mathcal{M} \to \mathcal{N}$ with a profunctor
\begin{equation}
  \label{eq:functor-to-prof}
  H_F: \mathcal{M} \profto \mathcal{N},
  \quad (N, M) \mapsto \Hom_{\mathcal{N}}(N, F(M)).
\end{equation}
Inspired by the notion of a Tambara module introduced by Pastro and Street \cite{MR2425558}, we now introduce the following definition:

\begin{definition}
  \label{def:C-prof}
  Let $\mathcal{C}$ be a monoidal category, and let $\mathcal{M}$ and $\mathcal{N}$ be left $\mathcal{C}$-module categories. A (left) {\em $\mathcal{C}$-module profunctor} from $\mathcal{M}$ to $\mathcal{N}$ is a pair $(T, \theta)$ consisting of a profunctor $T: \mathcal{M} \profto \mathcal{N}$ and a (di-)natural transformation
  \begin{equation*}
    \theta_{X,M,N}: T(N,M) \to T(X \catactl N, X \catactl M)
    \quad (X \in \mathcal{C}, N \in \mathcal{N}, M \in \mathcal{M})
  \end{equation*}
  (strictly speaking, a family of morphisms that is natural in the variables $N$ and $M$ and dinatural \cite[IX.4]{MR1712872} in the variable $X$) such that the equations
  \begin{gather}
    \label{eq:C-prof-str-axiom-1}
    \theta_{\unitobj, N, M} = \id_{T(N, M)}, \\
    \label{eq:C-prof-str-axiom-2}
    \theta_{X \otimes Y, N, M} = \theta_{X, Y \catactl N, Y \catactl M} \circ \theta_{Y, N, M}
  \end{gather}
  hold for all $X, Y \in \mathcal{C}$, $N \in \mathcal{N}$ and $M \in \mathcal{M}$. Given $\mathcal{C}$-module profunctors $(S, \sigma)$ and $(T, \theta)$ from $\mathcal{M}$ to $\mathcal{N}$, a {\em morphism} from $(S, \sigma)$ to $(T, \theta)$ is a natural transformation $\phi: S \to T$ such that the equation
  $\phi_{X \catactl N, X \catactl M} \circ \sigma_{X,N,M} = \theta_{X,N,M} \circ \phi_{N,M}$
  hold for all $X \in \mathcal{C}$, $N \in \mathcal{N}$ and $M \in \mathcal{M}$.
\end{definition}

We give basic properties of module profunctors. We do not consider the `composition' of module profunctors in this paper, but we shall remark that one can compose a $\mathcal{C}$-module profunctor and (op)lax $\mathcal{C}$-module functors in the following way: Let $\mathcal{M}_i$ and $\mathcal{N}_i$ be left $\mathcal{C}$-module categories ($i = 1, 2$). If $(F, r): \mathcal{N}_1 \to \mathcal{N}_2$ is an oplax $\mathcal{C}$-module functor, $(G, s): \mathcal{M}_1 \to \mathcal{M}_2$ is a lax $\mathcal{C}$-module functor and $(T, \theta): \mathcal{M}_2 \profto \mathcal{N}_2$ is a $\mathcal{C}$-module profunctor, then the functor
\begin{equation*}
  S: \mathcal{N}_1^{\op} \times \mathcal{M}_1 \to \Sets,
  \quad (N_1, M_1) \mapsto T(F(N_1), G(M_1))
\end{equation*}
is a $\mathcal{C}$-module profunctor $\mathcal{M}_1 \profto \mathcal{N}_1$ by the structure map
\begin{equation*}
  T(r_{X,N}, s_{X,M}) \circ \theta_{X,N,M}:
  S(N, M)
  \to S(X \catactl N, X \catactl M)
\end{equation*}
for $X \in \mathcal{C}$, $N \in \mathcal{N}_1$ and $M \in \mathcal{M}_1$.

Now let $\mathcal{M}$ and $\mathcal{N}$ be left $\mathcal{C}$-module categories. The `identity' profunctor $\Hom_{\mathcal{N}}: \mathcal{N} \profto \mathcal{N}$ is a $\mathcal{C}$-module profunctor by the structure map given by
\begin{equation*}
  \Hom_{\mathcal{N}}(N, N') \to \Hom_{\mathcal{N}}(X \catactl N, X \catactl N'),
  \quad f \mapsto \id_X \catactl f
\end{equation*}
for $X \in \mathcal{C}$ and $N, N' \in \mathcal{N}$. Thus, if $F: \mathcal{M} \to \mathcal{N}$ is a lax left $\mathcal{C}$-module functor, then the profunctor $H_F: \mathcal{M} \profto \mathcal{N}$ defined by~\eqref{eq:functor-to-prof} is a $\mathcal{C}$-module profunctor as the composition of $\Hom_{\mathcal{N}}$ and $F$. This construction gives rise to a functor
\begin{equation}
  \label{eq:mod-prof-construction-1}
  \FUN_{\mathcal{C}}(\mathcal{M}, \mathcal{N}) \to \PROF_{\mathcal{C}}(\mathcal{M}, \mathcal{N}),
  \quad F \mapsto H_F,
\end{equation}
where $\FUN_{\mathcal{C}}(\mathcal{M}, \mathcal{N})$ and $\PROF_{\mathcal{C}}(\mathcal{M}, \mathcal{N})$ are the category of lax left $\mathcal{C}$-module functors and the category of $\mathcal{C}$-module profunctors from $\mathcal{M}$ to $\mathcal{N}$, respectively. By a Yoneda-type argument, it is routine to prove:

\begin{lemma}
  The functor~\eqref{eq:mod-prof-construction-1} is fully faithful.
\end{lemma}

The following lemma is important:

\begin{lemma}
  \label{lem:representable-module-profunctors}
  Let $\mathcal{M}$ and $\mathcal{N}$ be as above, and let $F: \mathcal{M} \to \mathcal{N}$ be a functor.
  \begin{enumerate}
  \item There is a one-to-one correspondence between lax left $\mathcal{C}$-module functor structures on $F$ and $\mathcal{C}$-module profunctor structures on $H_F$.
  \item There is a one-to-one correspondence between oplax left $\mathcal{C}$-module functor structures on $F$ and $\mathcal{C}$-module profunctor structures on the profunctor
    \begin{equation*}
      \tilde{H}_F: \mathcal{N} \profto \mathcal{M},
      \quad (M, N) \mapsto \Hom_{\mathcal{N}}(F(M), N).
    \end{equation*}
  \end{enumerate}
\end{lemma}
\begin{proof}
  We prove Part (1). Suppose that we are given a natural transformation $\theta$ making $T := H_F$ a $\mathcal{C}$-module profunctor. For $X \in \mathcal{C}$ and $M \in \mathcal{M}$, we define $s_{X,M}$ to be the image of $\id_{F(M)}$ under the map
  \begin{equation*}
    \theta_{X, F(M), M}: T(F(M), M) \to T(X \catactl F(M), X \catactl M).
  \end{equation*}
  By the naturality of $\theta_{X,N,M}$ in the variable $N \in \mathcal{N}$, we have
  \begin{equation}
    \label{eq:C-prof-lax-C-func:pf:1}
    \theta_{X,N,M}(f)
    = s_{X,M} \circ (\id_X \catactl f)
  \end{equation}
  for all $X \in \mathcal{C}$, $M \in \mathcal{M}$, $N \in \mathcal{N}$ and $f \in T(N, M)$. Using this formula and \eqref{eq:C-prof-str-axiom-1}, we can verify the first equation of \eqref{eq:lax-C-module-functor}. We also find that~\eqref{eq:C-prof-str-axiom-2} implies the second equation of \eqref{eq:lax-C-module-functor}. The naturality of $\theta$ implies the naturality of $s = \{ s_{X,M} \}$. Hence $(F, s)$ is a lax $\mathcal{C}$-module functor.

  Conversely, if $(F, s)$ is a lax $\mathcal{C}$-module functor, then we define $\theta$ by~\eqref{eq:C-prof-lax-C-func:pf:1}. It is routine to show that $(T, \theta)$ is indeed a $\mathcal{C}$-module profunctor. These constructions are mutually inverse. Now the proof of Part (1) is done. Part (2) can be proved in a similar way.
\end{proof}

The above lemma yields the following result:

\begin{lemma}[{{\it cf}. \cite[Lemma 2.11]{MR3934626}}]
  \label{lem:module-functor-adjunction}
  Let $F: \mathcal{N} \to \mathcal{M}$ be a functor, and suppose that it has a right adjoint $G$. If $(F, r)$ is an oplax $\mathcal{C}$-module functor, then $G$ has a unique structure of a lax $\mathcal{C}$-module functor such that the diagram
  \begin{equation}
    \label{eq:C-prof-adj-2}
    \begin{tikzcd}[column sep = 96pt]
      \Hom_{\mathcal{M}}(F(N), M)
      \arrow[r, "{f \, \mapsto \, (\id_{X} \catactl f) \circ r_{X,N}}"]
      \arrow[d, "{\Phi_{N,M}}"']
      & \Hom_{\mathcal{M}}(F(X \catactl N), X \catactl M)
      \arrow[d, "{\Phi_{X \catactl N, X \catactl M}}"] \\
      \Hom_{\mathcal{N}}(N, G(M))
      \arrow[r, "{g \, \mapsto \, s_{X,M} \circ (\id_{X} \catactl g)}"]
      & \Hom_{\mathcal{N}}(X \catactl N, G(X \catactl M))
    \end{tikzcd}
  \end{equation}
  commutes for all objects $M \in \mathcal{M}$, $N \in \mathcal{N}$ and $X \in \mathcal{C}$, where $\Phi$ is the adjunction isomorphism \eqref{eq:adjunction-iso} and $s$ is the structure morphism of $G$.
  Similarly, if $(G, s)$ is a lax $\mathcal{C}$-module functor, then $F$ has a unique structure $r$ of an oplax $\mathcal{C}$-module functor such that the diagram \eqref{eq:C-prof-adj-2} commutes.
\end{lemma}
\begin{proof}
  Suppose that $(F, r)$ is an oplax $\mathcal{C}$-module functor.
  Then, by Lemma~\ref{lem:representable-module-profunctors}, $G$ has a unique structure of a lax $\mathcal{C}$-module structure such that the adjunction isomorphism $\Phi$ is an isomorphism of $\mathcal{C}$-module profunctors.
  The proof of the first half part is done by noting that the commutativity of \eqref{eq:C-prof-adj-2} is equivalent to that $\Phi$ is a morphism of $\mathcal{C}$-module profunctors.
  The latter half part is proved in a similar way.
\end{proof}

The following construction will be used frequently:

\begin{lemma}
  \label{lem:C-mod-func-adj}
  If a strong $\mathcal{C}$-module functor $F : \mathcal{N} \to \mathcal{M}$ has a right adjoint $G$, then $G$ has a unique structure of a lax $\mathcal{C}$-module functor such that the unit and the counit of $F \dashv G$ are morphisms of lax left $\mathcal{C}$-module functors.
\end{lemma}
\begin{proof}
  Let $q_{X, N} : X \catactl F(N) \to F(X \catactl N)$ be the structure morphism of $F$, and let $\eta$ and $\varepsilon$ be the unit and the counit of the adjunction $F \dashv G$, respectively.
  Suppose that $G$ is a lax left $\mathcal{C}$-module functor with structure morphism $s$. We note that the unit $\eta$ is a morphism of lax $\mathcal{C}$-module functor if and only if the equation
  \begin{equation}
    \label{eq:C-mod-func-adj-proof-1}
    \eta_{X \catactl N}
    = G(q_{X, M}) \circ s_{X, F(N)} \circ (\id_X \catactl \eta_N)    
  \end{equation}
  holds for all $X \in \mathcal{C}$ and $N \in \mathcal{N}$.
  We also note that the counit $\varepsilon$ is a morphism of lax $\mathcal{C}$-module functor if and only if the equation
  \begin{equation}
    \label{eq:C-mod-func-adj-proof-2}
    \varepsilon_{X \catactl M} \circ F(s_{X,M}) \circ q_{X,G(M)}
    = \id_X \catactl \varepsilon_M
  \end{equation}
  holds for all $X \in \mathcal{C}$ and $M \in \mathcal{M}$.
  Thus, if \eqref{eq:C-mod-func-adj-proof-2} holds, then we have
  \begin{equation*}
    s_{X,M} = \Phi_{X \catactl G(M), X \catactl M}((\id_X \catactl \varepsilon_M) \circ q_{X,G(M)}^{-1})
  \end{equation*}
  for $X \in \mathcal{C}$ and $M \in \mathcal{M}$, where $\Phi$ is the adjunction isomorphism \eqref{eq:adjunction-iso}.
  This proves the uniqueness part of this lemma.

  Now we prove that $G$ has a structure of a lax $\mathcal{C}$-module functor as stated.
  By applying Lemma \ref{lem:module-functor-adjunction} to the oplax $\mathcal{C}$-module functor $(F, r)$ with $r = q^{-1}$, we find a natural transformation $s$ making $G$ a lax $\mathcal{C}$-module functor such that the diagram \eqref{eq:C-prof-adj-2} commutes.
  Assuming $M = F(N)$ and chasing the identity morphism $\id_{F(N)}$ around the diagram \eqref{eq:C-prof-adj-2}, we obtain
  \begin{equation*}
    G(r_{X, N}) \circ \eta_{X \catactl N}
    = s_{X, F(N)} \circ (\id_X \catactl \eta_N),
  \end{equation*}
  which implies~\eqref{eq:C-mod-func-adj-proof-1}.
  Similarly, assuming $N = G(M)$ and chasing $\id_{G(M)}$ around the diagram \eqref{eq:C-prof-adj-2}, we obtain
  \begin{equation*}
    \varepsilon_{X \catactl M} \circ F(s_{X, M}) = (\id_X \catactl \varepsilon_M) \circ r_{X,G(M)},
  \end{equation*}
  which implies~\eqref{eq:C-mod-func-adj-proof-2}. The proof is done.
\end{proof}

\subsection{Internal Hom functors}
\label{subsec:internal-hom}

Let $\mathcal{C}$ be a monoidal category. We say that a left $\mathcal{C}$-module category $\mathcal{M}$ is {\em closed} if, for every object $M \in \mathcal{M}$, the functor
\begin{equation}
  \label{eq:iHom-def-0}
  T_M: \mathcal{C} \to \mathcal{M},
  \quad X \mapsto X \catactl M
\end{equation}
has a right adjoint ({\it cf}. the definition of a closed monoidal category).

Suppose that $\mathcal{M}$ is a closed left $\mathcal{C}$-module category. Given an object $M \in \mathcal{M}$, we denote a right adjoint of the functor $T_M$ by $\iHom_{\mathcal{M}}(M, -)$ or $Y_M$. By definition, there is a natural isomorphism
\begin{equation}
  \label{eq:iHom-def-1}
  \Hom_{\mathcal{C}}(X, \iHom_{\mathcal{M}}(M, N))
  \cong \Hom_{\mathcal{M}}(X \catactl M, N)
\end{equation}
for $X \in \mathcal{C}$ and $N \in \mathcal{M}$. By the parameter theorem for adjunctions \cite[IV.7]{MR1712872}, we extend the assignment $(M, N) \mapsto \iHom_{\mathcal{M}}(M, N)$ to a functor
\begin{equation}
  \label{eq:iHom-def-2}
  \iHom_{\mathcal{M}}: \mathcal{M}^{\op} \times \mathcal{M} \to \mathcal{C}
\end{equation}
in such a way that the isomorphism~\eqref{eq:iHom-def-1} is also natural in $M \in \mathcal{M}$. We call the functor \eqref{eq:iHom-def-2} the {\em internal Hom functor} of $\mathcal{M}$. If the module category $\mathcal{M}$ is clear from the context, we write $\iHom_{\mathcal{M}}$ simply as $\iHom$.

We recall basic properties of the internal Hom functor. First, the internal Hom functor makes $\mathcal{M}$ a $\mathcal{C}$-enriched category \cite{MR1976459}. To define the composition and the identity for $\mathcal{M}$ as a $\mathcal{C}$-enriched category, we denote by
\begin{equation*}
  \icoev_{X,M}: X \to \iHom(M, X \catactl M)
  \quad \text{and} \quad
  \ieval_{M,N}: \iHom(M, N) \otimes M \to N
\end{equation*}
the unit and the counit of the adjunction $T_M \dashv Y_M$, respectively. Then, for three objects $M_1, M_2, M_3 \in \mathcal{M}$, the composition morphism
\begin{equation*}
  \icomp_{M_1, M_2, M_3}: \iHom(M_2, M_3) \otimes \iHom(M_1, M_2) \to \iHom(M_1, M_3)
\end{equation*}
is defined to be the morphism corresponding to
\begin{equation*}
  \ieval_{M_2, M_3} (\id_{\iHom(M_2, M_3)} \catactl \ieval_{M_1, M_2})
  : \iHom(M_2, M_3) \catactl \iHom(M_1, M_2) \catactl M_1 \to M_3
\end{equation*}
via the isomorphism~\eqref{eq:iHom-def-1} with $X = \iHom(M_2, M_3) \otimes \iHom(M_1, M_2)$, $M = M_1$ and $N = M_3$. The identity $\unitobj \to \iHom(M, M)$ is given by the unit $\icoev_{\unitobj, M}$.

Fix an object $M \in \mathcal{M}$. Since the functor $T_M: \mathcal{C} \to \mathcal{M}$ is a strong left $\mathcal{C}$-module functor in an obvious way, its right adjoint $Y_M = \iHom(M, -)$ is a lax left $\mathcal{C}$-module functor from $\mathcal{M}$ to $\mathcal{C}$ by Lemma~\ref{lem:C-mod-func-adj}. We denote by
\begin{equation*}
  \mathfrak{a}_{X,M,N}: X \otimes \iHom(M, N) \to \iHom(M, X \catactl N)
  \quad (X \in \mathcal{C}, M, N \in \mathcal{M})
\end{equation*}
its structure morphism. Explicitly, $\mathfrak{a}_{X,M,N}$ is given by the composition
\begin{subequations}
  \newcommand{\xarr}[1]{\xrightarrow{\makebox[8em]{$\scriptstyle #1$}}}
  \begin{align}
    \label{eq:iHom-mod-left-1}
    X \otimes \iHom(M, N)
    \xarr{\icoev_{X \otimes \iHom(M, N), M}}\,
    & \iHom(M, X \catactl \iHom(M, N) \catactl M) \\
    \label{eq:iHom-mod-left-2}
    \xarr{\iHom(M, X \catactl \ieval_{M, N})}\,
    & \iHom(M, X \catactl N)
  \end{align}
\end{subequations}
for $X \in \mathcal{C}$ and $N \in \mathcal{M}$.

We now suppose that the monoidal category $\mathcal{C}$ is rigid. Given an object $M \in \mathcal{M}$, we often denote by $M^{\op}$ the object $M$ regarded as an object of $\mathcal{M}^{\op}$. The category $\mathcal{M}^{\op}$ is a right $\mathcal{C}$-module category by the action $\mathbin{\tilde{\catactr}}$ given by
$M^{\op} \mathbin{\tilde{\catactr}} X = ({}^* \! X \catactl M)^{\op}$ for $X \in \mathcal{C}$ and $M \in \mathcal{M}$, and hence $\mathcal{M}^{\op} \times \mathcal{M}$ is a $\mathcal{C}$-bimodule category. The category $\FUN_{\mathcal{C}}(\mathcal{M}, \mathcal{C})$ is a right $\mathcal{C}$-module category by the action $\catactr$ given by
\begin{equation}
  \label{eq:Mod-M-C-right-action}
  (F \catactr X)(M) = F(M) \otimes X
  \quad (X \in \mathcal{C}, F \in \FUN_{\mathcal{C}}(\mathcal{M}, \mathcal{C}), M \in \mathcal{M}).
\end{equation}

\begin{lemma}
  The internal Hom functor defines a right $\mathcal{C}$-module functor
  \begin{equation*}
    \mathcal{M}^{\op} \to \FUN_{\mathcal{C}}(\mathcal{M}, \mathcal{C}),
    \quad M \mapsto Y_M.
  \end{equation*}
\end{lemma}
\begin{proof}
  For $X \in \mathcal{C}$, we define $R_X: \mathcal{C} \to \mathcal{C}$ by $R_X(V) = V \otimes X$.
  We note that $R_X$ is right adjoint to $R_{{}^{*}\!X}$.
  By the uniqueness of a right adjoint, we have
  \begin{equation*}
    Y_M \catactr X
    = R_X \circ Y_{M}
    = R_{{}^* \! X}^{\radj} \circ T_{M}^{\radj}
    \cong (T_{M} \circ R_{{}^* \! X})^{\radj}
    = (T_{\, {}^* \! X \catactl M})^{\radj}
    = Y_{\, {}^* \! X \catactl M}
  \end{equation*}
  as left $\mathcal{C}$-module functors. By the commutative diagram~\eqref{eq:right-adj-uniqueness-3}, we directly check that the functor $M \mapsto Y_M$ in concern is a right $\mathcal{C}$-module functor with this structure morphism.
\end{proof}

In this paper, we denote by
\begin{equation*}
  \mathfrak{b}_{M,N,X}: \iHom(M, N) \otimes X \to \iHom({}^* \! X \catactl M, N)
  \quad (M, N \in \mathcal{M}, X \in \mathcal{C})
\end{equation*}
the component of the isomorphism $Y_M \catactr X \to Y_{{}^* \! X \catactl M}$ given in the proof of the above lemma. For $X \in \mathcal{C}$ and $M \in \mathcal{M}$, we define the morphism $\beta_{X,M}$ by
\begin{equation}
  \label{eq:iHom-mod-right-00}
  \beta_{X,M} = \Big(
  R_{X} \circ Y_M
  = R_{\, {}^* \! X}^{\radj} \circ T_M^{\radj}
  \xrightarrow[\cong]{\quad \eqref{eq:right-adj-uniqueness} \quad}
  (T_M \circ R_{\, {}^* \! X})^{\radj} = Y_{{}^* \! X \catactl M}
  \Big).
\end{equation}
Then $\mathfrak{b}_{M,N,X}$ is given by $\mathfrak{b}_{M,N,X} = \beta_{X,M}(N)$.
By expressing the isomorphism \eqref{eq:right-adj-uniqueness} in terms of the unit and the counit, we see that the natural isomorphism $\mathfrak{b}$ is given by the following composition:
\begin{subequations}
  \newcommand{\xarr}[1]{\xrightarrow{\makebox[10em]{$\scriptstyle #1$}}}
  \begin{align}
    \notag
    & \iHom(M, N) \otimes X \\
    \label{eq:iHom-mod-right-1}
    & \xarr{\icoev_{\iHom(M, N) \catactl X, {}^* \! X \catactl M}}
      \iHom({}^* \! X \catactl M, \iHom(M, N) \catactl X \catactl {}^* \! X  \catactl M) \\
    \label{eq:iHom-mod-right-2}
    & \xarr{\quad \iHom(\id, \id \catactl \eval_X' \catactl \id) \quad}
      \iHom({}^* \! X \catactl M, \iHom(M, N) \catactl M) \\
    \label{eq:iHom-mod-right-3}
    & \xarr{\quad \iHom(M, \ieval_{M, N}) \quad}
      \iHom({}^* \! X \catactl M, N).
  \end{align}
\end{subequations}
Since, by construction, $\mathfrak{b}_{M, -, X}: \iHom(M, -) \otimes X \to \iHom({}^* \! X \catactr M, -)$ is an isomorphism of left $\mathcal{C}$-module functors, we see that the diagram
\begin{equation*}
  \begin{tikzcd}[column sep = 64pt]
    X \otimes \iHom(M, N) \otimes Y
    \arrow[d, "{\mathfrak{a}_{X,M,N} \otimes Y}"']
    \arrow[r, "{X \otimes \mathfrak{b}_{M,N,Y}}"]
    & X \otimes \iHom({}^* Y \catactl M, N)
    \arrow[d, "{\mathfrak{a}_{X,{}^* \! Y \catactl M,N}}"] \\
    \iHom(M, X \catactl N) \otimes Y
    \arrow[r, "{\mathfrak{b}_{X \catactl M,N,Y}}"]
    & \iHom({}^* Y \catactl M, X \catactl N)
  \end{tikzcd}
\end{equation*}
commutes for all $X, Y \in \mathcal{C}$ and $M, N \in \mathcal{M}$. This means:

\begin{lemma}
  \label{lem:iHom-bimodule-functor}
  The internal Hom functor $\iHom: \mathcal{M}^{\op} \times \mathcal{M} \to \mathcal{C}$ is a $\mathcal{C}$-bimodule functor with structure morphisms $\mathfrak{a}$ and $\mathfrak{b}$ introduced in the above.
\end{lemma}

\subsection{Module profunctors and bimodule functors}
\label{subsec:mod-prof-bimod}

Let $\mathcal{C}$ be a rigid monoidal category, and let $\mathcal{M}$ and $\mathcal{N}$ be closed left $\mathcal{C}$-module categories. Given a left $\mathcal{C}$-module functor $F: \mathcal{M} \to \mathcal{N}$, we define
\begin{equation*}
  \underline{H}_F: \mathcal{N}^{\op} \times \mathcal{M} \to \mathcal{C},
  \quad (N, M) \mapsto \iHom_{\mathcal{N}}(N, F(M)).
\end{equation*}
This is a $\mathcal{C}$-bimodule functor by Lemma~\ref{lem:iHom-bimodule-functor} and thus we have a functor
\begin{equation}
  \label{eq:mod-prof-construction-2}
  \FUN_{\mathcal{C}}(\mathcal{M}, \mathcal{N})
  \to \FUN_{\mathcal{C}|\mathcal{C}}(\mathcal{N}^{\op} \times \mathcal{M}, \mathcal{C}),
  \quad F \mapsto \underline{H}_F,
\end{equation}
where $\FUN_{\mathcal{C}|\mathcal{C}}(\mathcal{X}, \mathcal{Y})$ for $\mathcal{C}$-bimodule categories $\mathcal{X}$ and $\mathcal{Y}$ means the category of $\mathcal{C}$-bimodule functors from $\mathcal{X}$ to $\mathcal{Y}$.

Given a $\mathcal{C}$-bimodule functor $T: \mathcal{N}^{\op} \times \mathcal{M} \to \mathcal{C}$ with structure morphism
\begin{equation*}
  s_{X,M,N,Y}: X \otimes T(N, M) \otimes Y \to T({}^* Y \catactl N, X \catactl M),
\end{equation*}
we define $T_0: \mathcal{M} \profto \mathcal{N}$ by $T_0(N, M) = \Hom_{\mathcal{C}}(\unitobj, T(N, M))$. This profunctor is in fact a $\mathcal{C}$-module profunctor by the structure map
\begin{align*}
  T_0(N, M) & \to T_0(X \catactl M, X \catactl N), \\
  f & \mapsto s_{X, M, N, X^*} \circ (\id_X \otimes f \otimes \id_{X^*}) \circ \coev_X
\end{align*}
for $X \in \mathcal{C}$, $M \in \mathcal{M}$ and $N \in \mathcal{M}$. The assignment $T \mapsto T_0$ extends to a functor
\begin{equation}
  \label{eq:mod-prof-construction-3}
  \FUN_{\mathcal{C}|\mathcal{C}}(\mathcal{N}^{\op} \times \mathcal{M}, \mathcal{C})
  \to \PROF_{\mathcal{C}}(\mathcal{M}, \mathcal{N}),
  \quad T \mapsto T_0.
\end{equation}
The composition of \eqref{eq:mod-prof-construction-2} and \eqref{eq:mod-prof-construction-3} sends $F \in \FUN_{\mathcal{C}}(\mathcal{M}, \mathcal{N})$ to
\begin{equation*}
  \mathcal{M} \profto \mathcal{N},
  \quad (N, M) \mapsto \Hom_{\mathcal{C}}(\unitobj, \iHom_{\mathcal{N}}(N, F(M))).
\end{equation*}
By the definition of the internal Hom functor, we see that this $\mathcal{C}$-module profunctor is canonically isomorphic to the $\mathcal{C}$-module profunctor $H_F$ given by~\eqref{eq:functor-to-prof}. Thus we have the following diagram of functors commuting up to isomorphisms:
\begin{equation*}
  \begin{tikzcd}
    \FUN_{\mathcal{C}}(\mathcal{M}, \mathcal{N})
    \arrow[rr, "{\eqref{eq:mod-prof-construction-1}}"]
    \arrow[rd, "{\eqref{eq:mod-prof-construction-2}}"']
    & & \PROF_{\mathcal{C}}(\mathcal{M}, \mathcal{N}) \\
    & \FUN_{\mathcal{C}|\mathcal{C}}(\mathcal{N}^{\op} \times \mathcal{M}, \mathcal{C})
    \arrow[ru, "{\eqref{eq:mod-prof-construction-3}}"']
  \end{tikzcd}
\end{equation*}

\begin{lemma}
  \label{lem:mod-prof-ff}
  Every arrow in the above diagram is a fully faithful functor.
\end{lemma}
\begin{proof}
  We have proved that \eqref{eq:mod-prof-construction-1} is fully faithful. Thus \eqref{eq:mod-prof-construction-2} is faithful and \eqref{eq:mod-prof-construction-3} is full. To prove this lemma, it suffices to show either the fullness of \eqref{eq:mod-prof-construction-2} or the faithfulness of \eqref{eq:mod-prof-construction-3}.

  We show that the functor \eqref{eq:mod-prof-construction-3} is faithful. We write the source of \eqref{eq:mod-prof-construction-3} as $\mathcal{F}$ for simplicity. Let $S$ and $T$ be objects of $\mathcal{F}$, and let $f$ and $g$ be morphisms from $S$ to $T$ in $\mathcal{F}$. By the definition of a morphism of bimodule functors, there is a commutative diagram
  \begin{equation*}
    \footnotesize
    \begin{tikzcd}
      \Hom_{\mathcal{C}}(\unitobj, S(N, {}^* \! X \catactl M))
      \arrow[r, "{\cong}"]
      \arrow[d, "{\Hom_{\mathcal{C}}(\unitobj, f_{N, {}^* \! X \catactl M})}"]
      & \Hom_{\mathcal{C}}(\unitobj, {}^* \! X \otimes S(N, M))
      \arrow[r, "{\cong}"]
      \arrow[d, "{\Hom_{\mathcal{C}}(\unitobj, {}^* \! X \catactl f_{N, M})}"]
      & \Hom_{\mathcal{C}}(X, S(N, M))
      \arrow[d, "{\Hom_{\mathcal{C}}(X, f_{N, M})}"] \\
      \Hom_{\mathcal{C}}(\unitobj, T(N, {}^* \! X \catactl M))
      \arrow[r, "{\cong}"]
      & \Hom_{\mathcal{C}}(\unitobj, {}^* \! X \otimes T(N, M))
      \arrow[r, "{\cong}"]
      & \Hom_{\mathcal{C}}(X, T(N, M))
    \end{tikzcd}
  \end{equation*}
  for $X \in \mathcal{C}$, $N \in \mathcal{N}$ and $M \in \mathcal{M}$. There is a similar diagram for $g$. If $f$ and $g$ are mapped to the same morphism by the functor \eqref{eq:mod-prof-construction-3}, that is, the equation
  \begin{equation*}
    \Hom_{\mathcal{C}}(\unitobj, f_{N, M})
    = \Hom_{\mathcal{C}}(\unitobj, g_{N, M})
  \end{equation*}
  holds for all objects $M \in \mathcal{M}$ and $N \in \mathcal{N}$, then we have $f = g$ by the above commutative diagram for $f$, an analogous diagram for $g$ and the Yoneda lemma. The proof is done.
\end{proof}

We close this section by giving some useful remarks.

\begin{remark}
  \label{rem:iHom-mod-fun-induced-map}
  A functor $F: \mathcal{M} \to \mathcal{N}$ yields a natural transformation
  \begin{equation}
    \label{eq:iHom-mod-fun-induced-map-1}
    F|_{M,M'}: \Hom_{\mathcal{M}}(M, M') \to \Hom_{\mathcal{N}}(F(M), F(M')),
    \quad g \mapsto F(g)
  \end{equation}
  for $M, M' \in \mathcal{M}$. If $F$ is a left $\mathcal{C}$-module functor, then \eqref{eq:iHom-mod-fun-induced-map-1} is in fact a morphism of $\mathcal{C}$-module profunctors. Thus, by Lemma~\ref{lem:mod-prof-ff}, there is a unique morphism
  \begin{equation}
    \label{eq:iHom-mod-fun-induced-map-2}
    \underline{F}|_{M,M'}:
    \iHom_{\mathcal{M}}(M, M') \to \iHom_{\mathcal{N}}(F(M), F(M'))
    \quad (M, M' \in \mathcal{M})
  \end{equation}
  of $\mathcal{C}$-bimodule functors such that the diagram
  \begin{equation*}
    \begin{tikzcd}[column sep = 32pt]
      \Hom_{\mathcal{M}}(M, M')
      \arrow[r, "{\eqref{eq:iHom-def-1}}"]
      \arrow[d, "{F|_{M,M'}}"']
      & \Hom_{\mathcal{C}}(\unitobj, \iHom_{\mathcal{M}}(M, M'))
      \arrow[d, "{\Hom_{\mathcal{C}}(\unitobj, \underline{F}|_{M,M'})}"] \\
      \Hom_{\mathcal{N}}(F(M), F(M'))
      \ar[r, "{\eqref{eq:iHom-def-1}}"]
      & \Hom_{\mathcal{C}}(\unitobj, \iHom_{\mathcal{N}}(F(M), F(M')))
    \end{tikzcd}
  \end{equation*}
  commutes for all $M, M' \in \mathcal{M}$.
  One can check that the morphism \eqref{eq:iHom-mod-fun-induced-map-2} coincides with the composition
  \begin{equation*}
    \newcommand{\xarr}[1]{\xrightarrow{\makebox[9em]{$\scriptstyle #1$}}}
    \begin{aligned}
      \iHom_{\mathcal{M}}(M, M')
      & \xarr{\icoev_{\iHom(M, M'), F(M)}} \iHom_{\mathcal{N}}(F(M), \iHom(M, M') \catactl F(M')) \\
      & \xarr{\iHom_{\mathcal{N}}(F(M), s)} \iHom_{\mathcal{N}}(F(M), F(\iHom(M, M') \catactl M')) \\
      & \xarr{\iHom_{\mathcal{N}}(F(M), \ieval_{M,M'})} \iHom_{\mathcal{N}}(F(M), F(M')),
    \end{aligned}
  \end{equation*}
  where $s$ is the structure morphism of $F$.
\end{remark}

\begin{remark}
  \label{rem:iHom-mod-fun-adj}
  Let $F: \mathcal{M} \to \mathcal{N}$ be a left $\mathcal{C}$-module functor admitting a right adjoint, and let $\eta$ and $\varepsilon$ be the unit and the counit of the adjunction $F \dashv F^{\radj}$. We define
  \begin{equation}
    \label{eq:iHom-mod-fun-adj-2}
    \phi^F_{N,M}:
    \iHom_{\mathcal{N}}(F(M), N) \to \iHom_{\mathcal{M}}(M, F^{\radj}(N))
    \quad (N \in \mathcal{N}, M \in \mathcal{M})
  \end{equation}
  to be the following composition:
  \begin{equation*}
    \newcommand{\xarr}[1]{\xrightarrow{\makebox[9em]{$\scriptstyle #1$}}}
    \begin{aligned}
      \iHom_{\mathcal{N}}(F(M), N)
      & \xarr{\underline{F^{\radj}}|_{F(M),N}}
      \iHom_{\mathcal{M}}(F^{\radj} F(M), F^{\radj}(N)) \\
      & \xarr{\iHom_{\mathcal{M}}(\eta, F^{\radj}(N))}
      \iHom_{\mathcal{M}}(M, F^{\radj}(N)).
    \end{aligned}
  \end{equation*}
  One can check that the morphism \eqref{eq:iHom-mod-fun-adj-2} is an isomorphism of $\mathcal{C}$-bimodule functors with the inverse given by
  \begin{equation*}
    \newcommand{\xarr}[1]{\xrightarrow{\makebox[9em]{$\scriptstyle #1$}}}
    \begin{aligned}
      \iHom_{\mathcal{M}}(M, F^{\radj}(N))
      & \xarr{\underline{F}|_{F(M),N}}
      \iHom_{\mathcal{N}}(F(M), F F^{\radj}(N)) \\
      & \xarr{\iHom_{\mathcal{N}}(F(M), \varepsilon_N)}
      \iHom_{\mathcal{N}}(F(M), N).
    \end{aligned}
  \end{equation*}
  Let $G: \mathcal{N} \to \mathcal{P}$ be also a left $\mathcal{C}$-module functor admitting a right adjoint. Then we have the following commutative diagram:
  \begin{equation*}
    \begin{tikzcd}[column sep = 96pt]
      \iHom_{\mathcal{P}}(G F(M), P)
      \arrow[d, "{\phi^{G F}_{M,P}}"']
      \arrow[r, "{\phi^G_{F(M),P}}"]
      & \iHom_{\mathcal{N}}(F(M), G^{\radj}(P))
      \arrow[d, "{\phi^F_{F(M),G^{\radj}(P)}}"] \\
      \iHom_{\mathcal{M}}(M, (G F)^{\radj}(P))
      & \iHom_{\mathcal{M}}(M, F^{\radj} G^{\radj}(P)),
      \arrow[l, "{\iHom(M, \gamma_{F,G}(P))}"']
    \end{tikzcd}
  \end{equation*}
  where $\gamma$ is the isomorphism \eqref{eq:right-adj-uniqueness}.
\end{remark}

\section{Relative Serre functors for exact module categories}
\label{sec:rel-Serre}

In this section, we give a brief review of \cite{MR4042867,MR3435098} on relative Serre functors for exact module categories over finite tensor categories.
We discuss how the isomorphisms \eqref{eq:Intro-rel-Serre-1} and \eqref{eq:Intro-rel-Serre-2} coming with the relative Serre functor are given.
We also introduce a pivotal structure of an exact module category following Schaumann \cite{MR3435098}. Some applications of this structure will be given in Subsections \ref{subsec:piv-mod-categ} and \ref{subsec:symmetric-Frobenius-algebras}.

\subsection{Finite tensor categories and their modules}
\label{subsec:fin-ten-cat}

We first recall basic terminology. From now on, we work over a field $\bfk$ of arbitrary characteristic. By an algebra, we always mean an associative unital algebra over the field $\bfk$. Given algebras $A$ and $B$, we denote by ${}_A \Mod$, $\Mod_B$ and ${}_A \Mod_B$ the categories of finite-dimensional left $A$-modules, right $B$-modules and $A$-$B$-bimodules, respectively.

A {\em finite abelian category} \cite[Section 1.8]{MR3242743} is a $\bfk$-linear category that is equivalent to ${}_A \Mod$ for some finite-dimensional algebra $A$. We will frequently use the following well-known fact: A $\bfk$-linear functor between finite abelian categories has a left (right) adjoint if and only if it is left (right) exact.

A {\em finite multi-tensor category} \cite{MR2119143} is a finite abelian category $\mathcal{C}$ equipped with a structure of a rigid monoidal category such that the tensor product $\otimes: \mathcal{C} \times \mathcal{C} \to \mathcal{C}$ is $\bfk$-bilinear. A {\em finite tensor category} is a finite multi-tensor category with simple unit object. We note that the tensor product functor of a finite multi-tensor category is exact in each variable because of the rigidity.

Let $\mathcal{C}$ be a finite tensor category. A {\em finite left $\mathcal{C}$-module category} is a left $\mathcal{C}$-module category $\mathcal{M}$ such that $\mathcal{M}$ is a finite abelian category and the action of $\mathcal{C}$ on $\mathcal{M}$ is $\bfk$-bilinear and right exact in each variable. This definition ensures that a finite left $\mathcal{C}$-module category admits an internal Hom functor. It should be remarked that the action of $\mathcal{C}$ on a finite left $\mathcal{C}$-module category is {\em exact} despite that we only assume the right exactness; see \cite{MR3934626}.

Given two finite left $\mathcal{C}$-module categories $\mathcal{M}$ and $\mathcal{N}$, we denote by $\LEX_{\mathcal{C}}(\mathcal{M}, \mathcal{N})$ and $\REX_{\mathcal{C}}(\mathcal{M}, \mathcal{N})$ the category of $\bfk$-linear left and right exact left $\mathcal{C}$-module functors from $\mathcal{M}$ to $\mathcal{N}$, respectively. By Lemma~\ref{lem:C-mod-func-adj}, we have an equivalence
\begin{equation}
  \label{eq:adj-Rex-Lex}
  (-)^{\radj}: \REX_{\mathcal{C}}(\mathcal{M}, \mathcal{N})^{\op} \to \LEX_{\mathcal{C}}(\mathcal{N}, \mathcal{M}),
  \quad F \mapsto F^{\radj}
\end{equation}
of $\bfk$-linear categories.

Suppose, moreover, that $\mathcal{N}$ is a finite $\mathcal{C}$-bimodule category. Then the source and the target of the functor~\eqref{eq:adj-Rex-Lex} are $\mathcal{C}$-module categories in the following way:
\begin{itemize}
\item The category $\mathcal{R} := \REX_{\mathcal{C}}(\mathcal{M}, \mathcal{N})$ is a right $\mathcal{C}$-module category in the same way as \eqref{eq:Mod-M-C-right-action}, that is, by the action given by
  \begin{equation}
    \label{eq:REX-M-N-right-action}
    (F \catactr X)(M) = F(M) \catactr X
    \quad (X \in \mathcal{C}, F \in \mathcal{R}, M \in \mathcal{M}).
  \end{equation}
  Hence $\mathcal{R}^{\op}$ is a left $\mathcal{C}$-module category by the action $\tilde{\catactl}$ defined by
  \begin{equation}
    \label{eq:REX-M-N-op-left-action}
    X \mathbin{\tilde{\catactl}} F^{\op} = (F \catactr {}^*\! X)^{\op}
    \quad (F \in \mathcal{R}, X \in \mathcal{C}).
  \end{equation}
\item The category $\mathcal{L} := \LEX_{\mathcal{C}}(\mathcal{N}, \mathcal{M})$ is a left $\mathcal{C}$-module category by
  \begin{equation}
    \label{eq:LEX-M-N-left-action}
    (X \catactl F)(N) = F(N \catactr X)
    \quad (X \in \mathcal{C}, F \in \mathcal{L}, N \in \mathcal{N}).
  \end{equation}
\end{itemize}
The functor \eqref{eq:adj-Rex-Lex} is in fact an equivalence of left $\mathcal{C}$-module categories. This fact can be verified easily, but we give the following characterization of the module structure of \eqref{eq:adj-Rex-Lex} for later discussion: Given an object $X \in \mathcal{C}$, we define $R_X: \mathcal{N} \to \mathcal{N}$ by $R_X(N) = N \catactr X$ ($N \in \mathcal{N}$). The functor $R_X$ has an obvious left $\mathcal{C}$-module structure and is right adjoint to $R_{\, {}^* \! X}$. Thus, for $F \in \mathcal{R}$ and $X \in \mathcal{C}$, we have an isomorphism of left $\mathcal{C}$-module functors
\begin{equation*}
  X \catactl F^{\radj}
  = F^{\radj} \circ R_X
  = F^{\radj} \circ (R_{\, {}^* \! X})^{\radj}
  \xrightarrow[\cong]{\ \eqref{eq:right-adj-uniqueness}\ }
  (R_{\, {}^* \! X} \circ F)^{\radj}
  = (X \mathbin{\tilde{\catactl}} F)^{\radj}
\end{equation*}
by the uniqueness of a right adjoint. The functor \eqref{eq:adj-Rex-Lex} is made into a left $\mathcal{C}$-module functor by this isomorphism.

\subsection{Relative Serre functors}
\label{subsec:rel-Serre}

Let $\mathcal{C}$ be a finite tensor category, and let $\mathcal{M}$ be a finite left $\mathcal{C}$-module category. Following \cite{MR4042867,MR3435098}, we introduce:

\begin{definition}
  \label{def:rel-Serre}
  A {\em relative Serre functor} of $\mathcal{M}$ is a pair $\Ser = (\Ser, \phi)$ consisting of a functor $\Ser: \mathcal{M} \to \mathcal{M}$ and a natural isomorphism
  \begin{equation}
    \label{eq:rel-Serre-def}
    \phi_{M,N}: \iHom(N, \Ser(M)) \to \iHom(M, N)^* \quad (M, N \in \mathcal{C}).
  \end{equation}
\end{definition}

We first discuss when a relative Serre functor of $\mathcal{M}$ exists. An object $M \in \mathcal{M}$ is said to be $\mathcal{C}$-projective \cite{MR3934626} if the functor $Y_M := \iHom(M, -)$ is exact. We note the following characterization of $\mathcal{C}$-projectivity:

\begin{lemma}
  An object $M \in \mathcal{M}$ is $\mathcal{C}$-projective if and only if $P \catactl M$ is projective for all projective object $P \in \mathcal{M}$.
\end{lemma}
\begin{proof}
  For $M \in \mathcal{M}$ and $P \in \mathcal{C}$, there is an isomorphism
  \begin{equation}
    \label{eq:lem-C-projective}
    \Hom_{\mathcal{M}}(P \catactl M, -)
    \cong \Hom_{\mathcal{C}}(P, \iHom(M, -))
    = \Hom_{\mathcal{C}}(P, -) \circ Y_M
  \end{equation}
  of functors. Thus, if $Y_M$ is exact and $P \in \mathcal{C}$ is projective, then $\Hom_{\mathcal{M}}(P \catactl M, -)$ is exact as a composition of exact functors, and hence $P \catactl M$ is projective. Suppose, conversely, that $P \catactl M$ is projective for all projective object $P \in \mathcal{C}$. We consider the case where $P$ is a projective generator. Then $U := \Hom_{\mathcal{C} }(P, -)$ reflects exact sequences. Since $U \circ Y_M$ is exact by~\eqref{eq:lem-C-projective}, so is $Y_M$. The proof is done.
\end{proof}

Recall that an {\em exact left $\mathcal{C}$-module category} \cite[Definition 7.5.1]{MR2119143} is a finite left $\mathcal{C}$-module category $\mathcal{M}$ such that $P \catactl M \in \mathcal{M}$ is projective for all objects $M \in \mathcal{M}$ and all projective objects $P \in \mathcal{C}$.
According to \cite[Propositions 7.6.9]{MR2119143}, every $\bfk$-linear left $\mathcal{C}$-module functor from an exact left $\mathcal{C}$-module category to a finite left $\mathcal{C}$-module category is exact.

By the above lemma, a finite left $\mathcal{C}$-module category $\mathcal{M}$ is exact if and only if every object of $\mathcal{M}$ is $\mathcal{C}$-projective \cite[Propositions 7.9.7]{MR2119143}. We use this observation to prove:

\begin{lemma}
  A finite left $\mathcal{C}$-module category $\mathcal{M}$ has a relative Serre functor if and only if it is exact. If $\mathcal{M}$ is exact, then the functor
  \begin{equation}
    \label{eq:rel-Ser-standard}
    \Ser: \mathcal{M} \to \mathcal{M}, \quad M \mapsto Y_M^{\radj}(\unitobj)
  \end{equation}
  is a relative Serre functor of $\mathcal{M}$.  
\end{lemma}

The first-half part of this lemma is \cite[Proposition 4.23]{MR4042867}. The formula \eqref{eq:rel-Ser-standard} of a relative Serre functor is found in the proof of \cite[Theorem 4.25]{MR4042867}. For later use, we include a detailed proof of this lemma.

\begin{proof}
  Suppose that $\mathcal{M}$ has a relative Serre functor $(\Ser, \phi)$. Then we have natural isomorphisms
  \begin{subequations}
    \begin{align}
      \label{eq:rel-Serre-def-1}
      \Hom_{\mathcal{C}}(\iHom(M, N), X)
      & \cong \Hom_{\mathcal{C}}(X^*, \iHom(M, N)^*) \\
      \label{eq:rel-Serre-def-2}
      & \cong \Hom_{\mathcal{C}}(X^*, \iHom(N, \Ser(M))) \\
      \label{eq:rel-Serre-def-3}
      & \cong \Hom_{\mathcal{M}}(X^* \catactl N, \Ser(M)) \\
      \label{eq:rel-Serre-def-4}
      & \cong \Hom_{\mathcal{M}}(N, X \catactl \Ser(M))
    \end{align}
  \end{subequations}
  for $X \in \mathcal{C}$ and $M, N \in \mathcal{M}$. This implies that the functor $Y_M$ is right exact for all $M \in \mathcal{M}$. Since $Y_M$ is left exact (as it is defined as a right adjoint), we conclude that $Y_M$ is exact. Hence $\mathcal{M}$ is exact.

  Suppose, conversely, that $\mathcal{M}$ is exact. Then $Y_M: \mathcal{M} \to \mathcal{C}$ is exact for all $M \in \mathcal{M}$ and hence we can define a functor $\Ser: \mathcal{M} \to \mathcal{M}$ by~\eqref{eq:rel-Ser-standard}. If we do so, then there are natural isomorphisms
  \begin{subequations}
    \begin{align}
      \label{eq:rel-Ser-std-iso-1}
      \Hom_{\mathcal{C}}(X, \iHom(N, \Ser(M)))
      & \cong \Hom_{\mathcal{M}}(X \catactl N, Y_M^{\radj}(\unitobj)) \\
      \label{eq:rel-Ser-std-iso-2}
      & \cong \Hom_{\mathcal{M}}(\iHom(M, X \catactl N), \unitobj) \\
      \label{eq:rel-Ser-std-iso-3}
      & \cong \Hom_{\mathcal{M}}(X \otimes \iHom(M, N), \unitobj) \\
      \label{eq:rel-Ser-std-iso-4}
      & \cong \Hom_{\mathcal{M}}(X, \iHom(M, N)^*)
    \end{align}
  \end{subequations}
  for $X \in \mathcal{C}$ and $M, N \in \mathcal{M}$. Thus $\iHom(N, \Ser(M)) \cong \iHom(M, N)^*$ by the Yoneda lemma. Hence $\Ser$ is a relative Serre functor of $\mathcal{M}$. The proof is done.
\end{proof}

In view of the above lemma, we suppose that $\mathcal{M}$ is exact and fix a relative Serre functor $(\Ser, \phi)$ of $\mathcal{M}$. By specializing \eqref{eq:rel-Serre-def-1}--\eqref{eq:rel-Serre-def-4} in the proof of the above lemma to $X = \unitobj$, we have a natural isomorphism
\begin{equation}
  \label{eq:rel-Ser-as-representing-obj}
  \Hom_{\mathcal{C}}(\iHom(M, N), \unitobj) \cong \Hom_{\mathcal{M}}(N, \Ser(M))
  \quad (M, N \in \mathcal{M}).
\end{equation}
This means that the object $\Ser(M)$ represents the functor
\begin{equation*}
  \mathcal{M}^{\op} \to \Sets,
  \quad N \mapsto \Hom_{\mathcal{C}}(\iHom(M, N), \unitobj).
\end{equation*}
Thus, by the Yoneda lemma, a relative Serre functor of $\mathcal{M}$ is unique up to isomorphisms. We will give a stronger statement in Lemma \ref{lem:rel-Ser-uniqueness} below.

In closing this subsection, we note the following important fact:

\begin{lemma}[{\cite[Subsection 4.4]{MR4042867}}]
  A relative Serre functor of an exact left $\mathcal{C}$-module category is a category equivalence.
\end{lemma}

\subsection{Twisted module structure of a relative Serre functor}

Let $\mathcal{C}$ be a finite tensor category. We first introduce the following notation: Given a tensor auto\-equivalence $F$ on $\mathcal{C}$ and a left $\mathcal{C}$-module category $\mathcal{M}$ with action $\catactl$, we denote by ${}_{F} \mathcal{M}$ the category $\mathcal{M}$ equipped with a new action $\catactl_F$ given by $X \catactl_F M = F(X) \catactl M$ for $X \in \mathcal{C}$ and $M \in \mathcal{M}$. A left $\mathcal{C}$-module functor from a left $\mathcal{C}$-module category $\mathcal{N}$ to ${}_{F}\mathcal{M}$ is often called an ($F$-)twisted $\mathcal{C}$-module functor from $\mathcal{N}$ to $\mathcal{M}$.

Now let $\mathcal{M}$ be an exact left $\mathcal{C}$-module category. By Lemma~\ref{lem:iHom-bimodule-functor}, the internal Hom functor of $\mathcal{M}$ has a structure of a $\mathcal{C}$-bimodule functor from $\mathcal{M}^{\op} \times \mathcal{M}$ to $\mathcal{C}$, which we denote by
\begin{equation*}
  \mathfrak{c}_{X, M, N, Y}:
  X \otimes \iHom(M, N) \otimes Y
  \to \iHom({}^{*}Y \catactl M, X \catactl N)
\end{equation*}
in this section. This induces a natural isomorphism
\begin{equation}
  \label{eq:iHom-dual-bimod}
  \iHom(Y \catactl M, {}^* \! X \catactl N)^*
  \xrightarrow{\quad (\mathfrak{c}_{{}^{*\!}X, M, N, Y^*})^* \quad}
  Y^{**} \otimes \iHom(M, N)^* \otimes X
\end{equation}
for $M, N \in \mathcal{M}$ and $X, Y \in \mathcal{C}$. Let $\DD = (-)^{**}$ denote the double left dual functor on $\mathcal{C}$. The dual of the internal Hom functor is in fact a $\mathcal{C}$-bimodule functor
\begin{equation*}
  \mathcal{M}^{\op} \times \mathcal{M} \to {}_{\DD}\mathcal{C},
  \quad (M, N) \mapsto \iHom(N, M)^*
\end{equation*}
by the structure morphism~\eqref{eq:iHom-dual-bimod}. We may also view it as a $\mathcal{C}$-bimodule functor
\begin{equation*}
  \mathcal{M}^{\op} \times {}_{\overline{\DD}} \mathcal{M} \to \mathcal{C},
  \quad (M, N) \mapsto \iHom(N, M)^*,
\end{equation*}
where $\overline{\DD} = {}^{**}(-)$ is the double right dual functor. By the argument of Subsection~\ref{subsec:mod-prof-bimod}, a relative Serre functor $\Ser$ of $\mathcal{M}$ has a unique structure of a left $\mathcal{C}$-module functor ${}_{\overline{\DD}}\mathcal{M} \to \mathcal{M}$ (or, equivalently, $\mathcal{M} \to {}_{\DD}\mathcal{M}$) such that the natural isomorphism~\eqref{eq:rel-Serre-def} is an isomorphism of $\mathcal{C}$-bimodule functors ({\it cf}. \cite[Proposition 4.30]{MR3435098} and \cite[Lemma 4.22]{MR4042867}).

The twisted $\mathcal{C}$-module structure of a relative Serre functor is one of central subjects of this paper. The following lemma shows that a relative Serre functor is unique up to isomorphism of {\em twisted $\mathcal{C}$-module functors} on $\mathcal{M}$.

\begin{lemma}[{\it cf}. {\cite[Proposition 4.30]{MR3435098}}]
  \label{lem:rel-Ser-uniqueness}
  Let $(\Ser, \phi)$ and $(\Ser', \phi')$ be relative Serre functors of $\mathcal{M}$. Then there is a unique natural transformation $\theta: \Ser' \to \Ser$ such that the equation
  \begin{equation}
    \label{eq:rel-Ser-uniqueness}
    \phi_{M,N}' = \phi_{M, N} \circ \iHom(N, \theta_M)
  \end{equation}
  holds for all objects $M, N \in \mathcal{M}$. Such a natural transformation $\theta: \Ser' \to \Ser$ is in fact an isomorphism of twisted $\mathcal{C}$-module functors.
\end{lemma}
\begin{proof}
  Equation \eqref{eq:rel-Ser-uniqueness} is equivalent to $\iHom(N, \theta_M) = \psi_{M,N}$, where
  \begin{equation*}
    \psi_{M,N} := \phi_{M,N}^{-1} \circ \phi'_{M,N}: \iHom(N, \Ser'(M)) \to \iHom(N, \Ser(M))
    \quad (M, N \in \mathcal{M}).
  \end{equation*}
  The existence and the uniqueness of a natural transformation $\theta: \Ser \to \Ser'$ satisfying \eqref{eq:rel-Ser-uniqueness} follow from \eqref{eq:iHom-def-1} and the Yoneda lemma.
  Since $\psi$ is in fact an isomorphism of $\mathcal{C}$-bimodule functors from $\mathcal{M}^{\op} \times {}_{\overline{\DD}} \mathcal{M}$ to $\mathcal{C}$, the natural transformation $\theta$ is in fact an isomorphism of twisted $\mathcal{C}$-module functors by Lemma~\ref{lem:mod-prof-ff}.
\end{proof}

\subsection{The `standard' realization of a relative Serre functor}
\label{subsec:std-realization}

Let $\mathcal{C}$ be a finite tensor category, and let $\mathcal{M}$ be an exact left $\mathcal{C}$-module category.
One of our aims is to give an explicit description of the structure morphism of a relative Serre functor $\Ser$ of $\mathcal{M}$. However, the functor $\Ser$ is determined only up to isomorphism and a description of its structure morphism depends on how we realize $\Ser$. For this reason, we choose the `standard' relative Serre functor as follows: From now on, we assume that a right adjoint $Y_M$ of the functor $T_M$ given by \eqref{eq:iHom-def-0} and a right adjoint of $Y_M$ are fixed for each object $M \in \mathcal{M}$.

\begin{definition}
  \label{def:rel-Serre-std}
  The {\em standard relative Serre functor} of $\mathcal{M}$ is a relative Serre functor $(\Ser, \phi)$ of $\mathcal{M}$ given as follows: As a functor, $\Ser$ is given by~\eqref{eq:rel-Ser-standard}. The natural isomorphism $\phi$ is induced by the natural isomorphism
  \begin{equation*}
    \Hom_{\mathcal{C}}(X, \iHom(N, \Ser(M)))
    \xrightarrow{\quad \text{\eqref{eq:rel-Ser-std-iso-1}--\eqref{eq:rel-Ser-std-iso-4}} \quad}
    \Hom_{\mathcal{C}}(X, \iHom(M, N)^*)
  \end{equation*}
  for $X \in \mathcal{C}$ and $M, N \in \mathcal{M}$.
\end{definition}

We give a description of the twisted module structure of the standard relative Serre functor of $\mathcal{M}$. For this purpose, we introduce the `trace' on an object of $\mathcal{M}$ and prove Lemma~\ref{lem:rel-Ser-pairing} below.

\begin{definition}
  \label{def:rel-Serre-trace}
  Let $\Ser = (\Ser, \phi)$ be a relative Serre functor $(\Ser, \phi)$ of $\mathcal{M}$. Then the {\em trace} with respect to $\Ser$ is the family $\itrace = \{ \itrace_M : \iHom(M, \Ser(M)) \to \unitobj \}_{M \in \mathcal{M}}$ of morphisms in $\mathcal{C}$ defined by
  \begin{equation*}
    \itrace_{M} := \Big( \iHom(M, \Ser(M))
    \xrightarrow{\quad \phi_{M, M} \quad}
    \iHom(M, M)^*
    \xrightarrow{\quad \icoev_{\unitobj,M}^* \quad}
    \unitobj^* = \unitobj \Big)
  \end{equation*}
  for $M \in \mathcal{M}$.
\end{definition}

Let $\Ser$ and $\Ser'$ be relative Serre functors of $\mathcal{M}$, and let $\itrace$ and $\itrace'$ denote traces with respect to $\Ser$ and $\Ser'$, respectively. Then we have
\begin{equation}
  \label{eq:rel-Ser-trace-1}
  \itrace'_M = \itrace_M \circ \iHom(M, \theta_M)
  \quad (M \in \mathcal{M}),
\end{equation}
where $\theta: \Ser' \to \Ser$ is the isomorphism given by Lemma \ref{lem:rel-Ser-uniqueness}.

\begin{lemma}
  \label{lem:rel-Ser-pairing}
  Let $\Ser = (\Ser, \phi)$ be a relative Serre functor of $\mathcal{M}$, and let $\itrace$ be the trace with respect to $\Ser$. Then the following hold:
  \begin{itemize}
  \item [(a)] For all $M, N \in \mathcal{M}$, we have
    \begin{equation*}
      \itrace_M \circ \icomp_{M, N, \Ser(M)} 
      = \eval_{\iHom(M, N)} \circ (\phi_{M,N} \otimes \id_{\iHom(M, N)})
    \end{equation*}
  \item [(b)] If $(\Ser, \phi)$ is the standard relative Serre functor, then we have
    \begin{equation*}
      \itrace_M = \varepsilon_M(\unitobj):
      Y_M^{} Y_M^{\radj}(\unitobj) \to \unitobj
      \quad (M \in \mathcal{M}),
    \end{equation*}
    where $\varepsilon_M: Y_M^{} Y_M^{\radj} \to \id_{\mathcal{C}}$ is the counit.
  \end{itemize}
\end{lemma}
\begin{proof}
  In view of \eqref{eq:rel-Ser-trace-1}, we may assume that $(\Ser, \phi)$ is the standard relative Serre functor. Let $X \in \mathcal{C}$ and $M, N \in \mathcal{M}$ be arbitrary objects. By the definition of the isomorphism $\phi_{M,N}$, we have the following commutative diagram:
  \begin{equation*}
    \begin{tikzcd}[column sep = 48pt]
      \Hom_{\mathcal{C}}(X, \iHom(N, \Ser(M)))
      \arrow[dd, "{\Hom_{\mathcal{C}}(X, \phi_{M,N})}"']
      \arrow[r, "{\eqref{eq:rel-Ser-std-iso-1}}"]
      & \Hom_{\mathcal{M}}(X \catactl N, Y_M^{\radj}(\unitobj))
      \arrow[d, "{\eqref{eq:rel-Ser-std-iso-2}}"] \\
      & \Hom_{\mathcal{M}}(\iHom(M, X \catactl N), \unitobj)
      \arrow[d, "{\eqref{eq:rel-Ser-std-iso-3}}"] \\
      \Hom_{\mathcal{C}}(X, \iHom(M, N)^*)
      & \Hom_{\mathcal{M}}(X \otimes \iHom(M, N), \unitobj)
      \arrow[l, "{\eqref{eq:rel-Ser-std-iso-4}}"']
    \end{tikzcd}
  \end{equation*}
  We now consider the case where $X = \iHom(N, \Ser(M))$ and chase the identity morphism $\id_X$ around the above diagram:
  \begin{align*}
    \id_X
    \xmapsto{\makebox[3em]{\scriptsize \eqref{eq:rel-Ser-std-iso-1}}}
    & \, \ieval_{N, \Ser(M)}
      \xmapsto{\makebox[3em]{\scriptsize \eqref{eq:rel-Ser-std-iso-2}}}
      \, \varepsilon_M(\unitobj) \circ \iHom(M, \ieval_{M, \Ser(M)}) \\
    \xmapsto{\makebox[3em]{\scriptsize \eqref{eq:rel-Ser-std-iso-3}}}
    & \, \varepsilon_M(\unitobj) \circ \iHom(M, \ieval_{M, \Ser(M)})
      \circ \mathfrak{a}_{\iHom(N, \Ser(M)), M, N} \\
    = & \, \varepsilon_M(\unitobj) \circ \icomp_{M, N, \Ser(M)}.
  \end{align*}
  Thus we have the following equation:
  \begin{equation}
    \label{eq:rel-Ser-pairing-proof-1}
    \eval_{\iHom(M, N)} \circ (\phi_{M,N}^{} \otimes \id_{\iHom(M, N)})
    = \varepsilon_M(\unitobj) \circ \icomp_{M, N, \Ser(M)}.
  \end{equation}
  The formula of Part (b) is obtained by letting $M = N$ and precomposing
  \begin{equation*}
    \id_{\iHom(M, \Ser(M))} \otimes \icoev_{\unitobj, M}:
    \iHom(M, \Ser(M)) \to \iHom(M, \Ser(M)) \otimes \iHom(M, M)
  \end{equation*}
  to both sides of~\eqref{eq:rel-Ser-pairing-proof-1}. Part (a) follows from Part (b) and \eqref{eq:rel-Ser-pairing-proof-1}.
\end{proof}

We recall that $\REX_{\mathcal{C}}(\mathcal{M}, \mathcal{C})$ and $\LEX_{\mathcal{C}}(\mathcal{C}, \mathcal{M})$ are $\mathcal{C}$-module categories by the actions given by \eqref{eq:REX-M-N-right-action} and \eqref{eq:LEX-M-N-left-action} with $\mathcal{N} = \mathcal{C}$, respectively. We also remark:
\begin{itemize}
\item Since $\mathcal{M}$ is exact, we have a $\bfk$-linear functor $\Yone: \mathcal{M} \to \REX_{\mathcal{C}}(\mathcal{M}, \mathcal{C})^{\op}$ given by $\Yone(M) = \iHom(M, -)$. The left $\mathcal{C}$-module structure of $\iHom$ induces a natural isomorphism $\Yone(X \catactl M) \cong \Yone(M) \catactr X^{*}$ for $X \in \mathcal{C}$ and $M \in \mathcal{M}$.
  In view of our convention \eqref{eq:REX-M-N-op-left-action} on the left action of $\mathcal{C}$ on $\REX_{\mathcal{C}}(\mathcal{M}, \mathcal{C})^{\op}$, we may say that $\Yone$ is in fact a $\bfk$-linear left $\mathcal{C}$-module functor
  \begin{equation*}
    \Yone: \mathcal{M} \to {}_{\DD} (\REX_{\mathcal{C}}(\mathcal{M}, \mathcal{C})^{\op}).
  \end{equation*}
\item There is an equivalence of $\bfk$-linear categories
  \begin{equation}
    \label{eq:equivalence-eval-at-1}
    \EvalAtOne: \LEX_{\mathcal{C}}(\mathcal{C}, \mathcal{M}) \to \mathcal{M},
    \quad \EvalAtOne(F) = F(\unitobj).
  \end{equation}
  This functor is an equivalence of left $\mathcal{C}$-module categories, since
  \begin{equation*}
    \EvalAtOne(X \catactl F) = F(X) \cong X \catactl F(\unitobj)
    \quad (F \in \LEX_{\mathcal{C}}(\mathcal{C}, \mathcal{M}), X \in \mathcal{C}).
  \end{equation*}
\end{itemize}
We now consider the composition
\begin{equation}
  \label{eq:FSS-rel-Serre-formula}
  \mathcal{M}
  \xrightarrow{\quad \Yone \quad}
  {}_{\DD} (\REX_{\mathcal{C}}(\mathcal{M}, \mathcal{C})^{\op})
  \xrightarrow{\quad (-)^{\radj} \quad}
  {}_{\DD} (\LEX_{\mathcal{C}}(\mathcal{C}, \mathcal{M}))
  \xrightarrow{\quad \EvalAtOne \quad}
  {}_{\DD} \mathcal{M}
\end{equation}
of left $\mathcal{C}$-module functors. It is easy to see that the functor \eqref{eq:FSS-rel-Serre-formula} is identical to the standard relative Serre functor on the level of functors. Furthermore,

\begin{theorem}
  \label{thm:mod-str-standard-rel-Serre-1}
  The standard relative Serre functor of $\mathcal{M}$ is equal to \eqref{eq:FSS-rel-Serre-formula} as a left $\mathcal{C}$-module functor from $\mathcal{M}$ to ${}_{\DD}\mathcal{M}$.
\end{theorem}
\begin{proof}
  Let $(\Ser, \phi)$ be the standard relative Serre functor of $\mathcal{M}$, and let $\Ser'$ be the left $\mathcal{C}$-module functor given by \eqref{eq:FSS-rel-Serre-formula}. As we have remarked, $\Ser = \Ser'$ as functors. We would like to show $\Ser = \Ser'$ as twisted module functors. By Lemma~\ref{lem:mod-prof-ff} and the definition of $\Ser$, this is equivalent to that the natural isomorphism
  \begin{equation*}
    \phi_{M,N}: \iHom(M, N)^* \to \iHom(N, \Ser'(M))
  \end{equation*}
  is an isomorphism of $\mathcal{C}$-bimodule functors. Since the right $\mathcal{C}$-module structure of the target of $\phi$ does not depend on the module structure of $\Ser'$, it is enough to show that $\phi$ is an isomorphism of left $\mathcal{C}$-module functors, that is, the diagram
  \begin{equation}
    \label{eq:mod-str-rel-Serre-proof-1}
    \begin{tikzcd}[column sep = 64pt]
      X^{**} \otimes \iHom(N, \Ser'(M))
      \arrow[d, "{\mathfrak{a}_{X^{**}, N, \Ser'(M)}}"']
      \arrow[r, "{X^{**} \otimes \phi_{M,N}}"]
      & X^{**} \otimes \iHom(M, N)^* \\
      \iHom(N, X^{**} \catactl \Ser'(M))
      \arrow[d, "{\iHom(N, \xi_{X,M})}"']
      & (\iHom(M, N) \otimes X^*)^*
      \arrow[u, equal]
      \arrow[d, "{(\mathfrak{b}_{M, N, X^*}^{-1})^*}"] \\
      \iHom(N, \Ser'(X \catactl M))
      \arrow[r, "{\phi_{X \catactl M,N}}"]
      & \iHom(X \catactl M, N)^*
    \end{tikzcd}
  \end{equation}
  commutes for $X \in \mathcal{C}$ and $M, N \in \mathcal{M}$, where $\mathfrak{a}$ and $\mathfrak{b}$ denote the left and the right $\mathcal{C}$-module structure of $\iHom$ (see Subsection \ref{subsec:internal-hom}) and
  \begin{equation*}
    \xi_{X,M}: X^{**} \catactl \Ser'(M) \to \Ser'(X \catactl M)
    \quad (X \in \mathcal{C}, M \in \mathcal{M})
  \end{equation*}
  is the twisted left $\mathcal{C}$-module structure of $\Ser'$. To save space, we write the internal Hom functor $\iHom(M, N)$ as $[M, N]$. By the canonical isomorphism
  \begin{align*}
    & \Hom_{\mathcal{C}}(X^{**} \otimes [N, \Ser'(M)], X^{**} \otimes [M, N]^*) \\
    & \qquad \cong \Hom_{\mathcal{C}}(X^{**} \otimes [N, \Ser'(M)] \otimes [M, N] \otimes X^*, \unitobj)
  \end{align*}
  and Lemma~\ref{lem:rel-Ser-pairing}, the commutativity of~\eqref{eq:mod-str-rel-Serre-proof-1} is equivalent to that of
  \begin{equation}
    \label{eq:mod-str-rel-Serre-proof-2}
    {\footnotesize
      \begin{tikzcd}[column sep = 80pt]
        X^{**} \otimes [N, \Ser'(M)] \otimes [M, N] \otimes X^*
        \arrow[d, "{\mathfrak{a} \otimes \mathfrak{b}}"']
        \arrow[r, "{X^{**} \otimes \icomp \otimes X^*}"]
        & X^{**} \otimes [M, \Ser'(M)] \otimes X^*
        \arrow[d, "{X^{**} \otimes \itrace_M \otimes X^*}"] \\
        {} [N, X^{**} \catactl \Ser'(M)] \otimes [X \catactl M, N]
        \arrow[d, "{[N, \xi_{X,M}] \otimes [X \catactl M, N]}"']
        & X^{**} \otimes X^* \arrow[d, "{\eval_{X^*}}"] \\
        {} [N, \Ser'(X \catactl M)] \otimes [X \catactl M, N]
        \arrow[r, "{\itrace_{X \catactl M} \circ \icomp_{X \catactl M, N, \Ser'(X \catactl M)}}"]
        & \unitobj.
      \end{tikzcd}}
  \end{equation}

  To verify the commutativity of \eqref{eq:mod-str-rel-Serre-proof-2}, we first recall the definition of the module structures of $\Ser'$ and related functors. We fix objects $X \in \mathcal{C}$ and $M \in \mathcal{M}$ and define $\beta_{X,M}$ by \eqref{eq:iHom-mod-right-00}. We also define $\alpha_{X,M}$ by the following composition:
  \begin{equation*}
    \alpha_{X,M} = \Big(
    Y_M^{\radj} \circ R_{X^{**}} = Y_M^{\radj} \circ R_{X^*}^{\radj}
    \xrightarrow[\cong]{\ \eqref{eq:right-adj-uniqueness}\ }
    (R_{X^*} \circ Y_M)^{\radj}
    \xrightarrow{\ (\beta_{X^*,M}^{-1})^{\radj} \ }
    Y^{\radj}_{X \catactl M}
    \Big).
  \end{equation*}
  Let $\kappa_{X,Y}: X \otimes Y_M^{\radj}(Y) \to Y_M^{\radj}(X \otimes Y)$ be the left $\mathcal{C}$-module structure of $Y_M^{\radj}$. Then the isomorphism $\xi_{X,M}$ is given by the composition
  \begin{equation}
    \label{eq:mod-str-rel-Serre-proof-4}
    \newcommand{\xarr}[1]{\xrightarrow{\makebox[4em]{$\scriptstyle #1$}}}
    \begin{aligned}
      \xi_{X,M} = \Big(
      X^{**} \catactl Y_M^{\radj}(\unitobj)
      \xarr{\kappa_{X^{**}, \unitobj}}
      Y_M^{\radj}(X^{**})
      \xarr{\alpha_{X,M}(\unitobj)}
      Y_{X \catactl M}^{\radj}(\unitobj)  \Big).
    \end{aligned}
  \end{equation}

  Next, we prove the commutativity of the diagram \eqref{eq:mod-str-rel-Serre-proof-5} below, which gives a relation between $\itrace_{M}$ and $\itrace_{X \catactl M}$. Given $M \in \mathcal{M}$, we denote by $\varepsilon_M$ the counit of the adjunction $Y_M \dashv Y_M^{\radj}$ as in Lemma~\ref{lem:rel-Ser-pairing} (b). By the definition of $\alpha$ and $\beta$, the following diagram commutes:
  \begin{equation}
    \label{eq:mod-str-rel-Serre-proof-3}
    \begin{tikzcd}[column sep = 96pt]
      R_{X^*} \circ Y_M \circ Y_M^{\radj} \circ R_{X^{**}}
      \arrow[r, "{R_{X^*} \circ \varepsilon_M \circ R_{X^{**}}}"]
      \arrow[d, "{\beta_{X^*,M} \circ \alpha_{X,M}}"']
      & R_{X^*} \circ R_{X^{**}}
      \arrow[d, "{\eval_{X^*}}"] \\
      Y_{X \catactl M} \circ Y_{X \catactl M}^{\radj}
      \arrow[r, "{\varepsilon_{X \catactl M}}"]
      & \id_{\mathcal{C}}
    \end{tikzcd}
  \end{equation}
  Now we consider the following diagram:
  \begin{equation*}
    \begin{tikzcd}[column sep = 96pt]
      X^{**} \otimes \iHom(M, Y_M^{\radj}(\unitobj)) \otimes X^*
      \arrow[r, "{X^{**} \otimes \varepsilon_M(\unitobj) \otimes X^*}"]
      \arrow[d, "{\mathfrak{a} \otimes X^*}"']
      & X \otimes X^* \arrow[dd, equal] \\
      \iHom(M, X^{**} \catactl Y_M^{\radj}(\unitobj)) \otimes X^*
      \arrow[d, "{\iHom(M, \kappa_{X^{**}, \unitobj}) \otimes X^*}"'] \\
      \iHom(M, Y_M^{\radj}(X^{**})) \otimes X^*
      \arrow[r, "{\varepsilon_{M}(X^{**}) \otimes X^*}"]
      \arrow[d, "{\mathfrak{b}}"']
      & X \otimes X^* \arrow[dd, "{\eval_{X^*}}"] \\
      \iHom(X \catactl M, Y_{M}^{\radj}(X^{**}))
      \arrow[d, "{\iHom(X \catactl M, \alpha_{X,M}(\unitobj)}"'] \\
      \iHom(X \catactl M, Y_{X \catactl M}^{\radj}(X))
      \arrow[r, "{\varepsilon_{X \catactl M}(\unitobj)}"]
      & \unitobj
    \end{tikzcd}
  \end{equation*}
  By Lemma \ref{lem:C-mod-func-adj}, the counit $\varepsilon_M$ is a morphism of $\mathcal{C}$-module functors. Thus the top square commutes. By evaluating \eqref{eq:mod-str-rel-Serre-proof-3} at $\unitobj$, we see that the bottom square also commutes. By~\eqref{eq:mod-str-rel-Serre-proof-4} and the naturality of $\mathfrak{b}$, the composition along the first column of the above diagram is equal to the composition
  \begin{equation*}
    \newcommand{\xarr}[1]{\xrightarrow{\makebox[7em]{$\scriptstyle #1$}}}
    \begin{aligned}
      X^{**} \otimes \iHom(M, \Ser'(M)) \otimes X^*
      & \xarr{\mathfrak{c}_{X^{**}, M, \Ser'(M), X^*}}
      \iHom(X \catactl M, X^{**} \otimes \Ser'(M)) \\
      & \xarr{\iHom(X \catactl M, \xi_{X,M})}
      \iHom(X \catactl M, X^{**} \otimes \Ser'(M)),
    \end{aligned}
  \end{equation*}
  where $\mathfrak{c}$ is the $\mathcal{C}$-bimodule structure of $\iHom$. Hence we obtain the following commutative diagram:
  \begin{equation}
    \label{eq:mod-str-rel-Serre-proof-5}
    \begin{tikzcd}[column sep = 96pt]
      X^{**} \otimes \iHom(M, \Ser'(M)) \otimes X^*
      \arrow[r, "{X^{**} \otimes \itrace_{M} \otimes X^*}"]
      \arrow[d, "{\mathfrak{c}_{X^{**},M,\Ser'(M),X^*}}"']
      & X^{**} \otimes X^* \arrow[dd, "{\eval_{X^*}}"] \\
      \iHom(X \catactl M, X^{**} \catactl \Ser'(M))
      \arrow[d, "{\iHom(X \catactl M, \xi_{X,M})}"'] \\
      \iHom(X \catactl M, \Ser'(X \catactl M))
      \arrow[r, "{\itrace_{X \catactl M}}"]
      & \unitobj
    \end{tikzcd}
  \end{equation}

  Now we complete the proof of this theorem by showing that the diagram~\eqref{eq:mod-str-rel-Serre-proof-2} commutes. We refine the diagram~\eqref{eq:mod-str-rel-Serre-proof-2} as follows:
  \begin{equation*}
    \footnotesize
    \begin{tikzcd}[column sep = 36pt, row sep = 24pt]
      X^{**} \otimes [N, \Ser'(M)] \otimes [M, N] \otimes X^*
      \arrow[r, "{\id \otimes \icomp \otimes \id}" {yshift=3pt}]
      \arrow[d, "{\mathfrak{a} \otimes \mathfrak{b}}"']
      \arrow[rd, phantom, "{(\spadesuit)}"]
      & X^{**} \otimes [M, \Ser'(M)] \otimes X^*
      \arrow[d, "{\mathfrak{c}}"]
      \arrow[r, "{\id \otimes \itrace_M \otimes \id}" {yshift=3pt}]
      & X^{**} \otimes X^*
      \arrow[dd, "{\eval_{X^*}}"] \\
      {} [N, X^{**} \catactl \Ser'(M)] \otimes [X \catactl M, N]
      \arrow[r, "{\icomp}"']
      \arrow[d, "{[N, \xi_{X,M}] \otimes \id}"']
      & {} [X \catactl M, X^{**} \catactl \Ser'(M)]
      \arrow[d, "{[X \catactl M, \xi_{X,M}] \otimes \id}"]
      & \arrow[lu, phantom, "{(\heartsuit)}"] \\
      {} [N, \Ser'(X \catactl M)] \otimes [X \catactl M, N]
      \arrow[r, "{\icomp}"']
      & {} [X \catactl M, \Ser'(X \catactl M)]
      \arrow[r, "{\itrace_{X \catactl M}}"']
      \ar[lu, phantom, "{(\clubsuit)}"]
      & \unitobj
    \end{tikzcd}
  \end{equation*}
  The cell $(\clubsuit)$ in the diagram is commutative because of the naturality of $\icomp$. The commutativity of $(\spadesuit)$ can be verified directly (some formulas given in Appendix of \cite{MR4077241} are helpful). The cell $(\heartsuit)$ is just the commutative diagram~\eqref{eq:mod-str-rel-Serre-proof-5}. The proof is done.
\end{proof}

\subsection{Relative Serre functor and adjoints}

Let $\mathcal{C}$ be a finite tensor category, and let $\mathcal{M}$ and $\mathcal{N}$ be exact left $\mathcal{C}$-module categories.
By the exactness of $\mathcal{N}$, the equivalence~\eqref{eq:adj-Rex-Lex} turns into an equivalence
\begin{equation*}
  (-)^{\radj}: \REX_{\mathcal{C}}(\mathcal{M}, \mathcal{N})^{\op} \to \REX_{\mathcal{C}}(\mathcal{N}, \mathcal{M}),
  \quad F \mapsto F^{\radj}
\end{equation*}
of $\bfk$-linear categories. We define a $\bfk$-linear autoequivalence
\begin{equation*}
  (-)^{\rradj}: \REX_{\mathcal{C}}(\mathcal{M}, \mathcal{N}) \to \REX_{\mathcal{C}}(\mathcal{M}, \mathcal{N}),
  \quad F \mapsto F^{\rradj} := (F^{\radj})^{\radj}
\end{equation*}
by taking a double right adjoint. This functor is given by a `conjugation' with relative Serre functors. More precisely, we have:

\begin{theorem}
  \label{thm:rel-Serre-double-right-adj}
  There is a natural isomorphism
  \begin{equation}
    \label{eq:rel-Serre-double-right-adj-iso}
    \theta_F: \Ser_{\mathcal{N}} \circ F \to F^{\rradj} \circ \Ser_{\mathcal{M}}
    \quad (F \in \REX_{\mathcal{C}}(\mathcal{M}, \mathcal{N}))
  \end{equation}
  of twisted left $\mathcal{C}$-module functors such that the diagram
  \begin{equation*}
    \begin{tikzcd}[column sep = 48pt]
      \Ser_{\mathcal{N}} \circ F \circ G
      \arrow[d, "{\theta_{F G}}"']
      \arrow[r, "{\theta_F \circ G}"]
      & F^{\rradj} \circ \Ser_{\mathcal{M}} \circ G
      \arrow[r, "{F^{\rradj} \circ \theta_G}"]
      & F^{\rradj} \circ G^{\rradj} \circ \Ser_{\mathcal{L}}
      \arrow[d, "{\gamma_{F^{\radj}, G^{\radj}} \circ \Ser_{\mathcal{L}}}"] \\
      (F \circ G)^{\rradj} \circ \Ser_{\mathcal{L}}
      \arrow[rr, "{\gamma_{F,G}^{\radj} \circ \Ser_{\mathcal{L}}}"]
      & & (F^{\radj} \circ G^{\radj})^{\radj} \circ \Ser_{\mathcal{L}}
    \end{tikzcd}
  \end{equation*}
  commutes for all $F \in \REX_{\mathcal{C}}(\mathcal{M}, \mathcal{N})$ and $G \in \REX_{\mathcal{C}}(\mathcal{L}, \mathcal{M})$, where $\gamma$ is the natural isomorphism~\eqref{eq:right-adj-uniqueness} and $\mathcal{L}$, $\mathcal{M}$ and $\mathcal{N}$ are exact left $\mathcal{C}$-module categories.
\end{theorem}

Except the part that $\theta_F$ is a morphism of twisted $\mathcal{C}$-module functors, this theorem follows from a property of the Nakayama functor \cite[Theorem 1.8]{MR4042867} and a relation between the Nakayama functor and the relative Serre functor \cite[Theorem 4.26]{MR4042867}. Below we give a proof based on Lemma~\ref{lem:mod-prof-ff}.

\begin{proof}
  There are natural isomorphisms
  \begin{gather*}
    \iHom_{\mathcal{N}}(N, \Ser_{\mathcal{N}} F(M))
    \xrightarrow{\ \eqref{eq:rel-Serre-def} \ }
    \iHom_{\mathcal{N}}(F(M), N)^*
    \xrightarrow{\ \eqref{eq:iHom-mod-fun-adj-2} \ }
    \iHom_{\mathcal{M}}(M, F^{\radj}(N))^* \\
    \xrightarrow{\quad \eqref{eq:rel-Serre-def} \quad}
    \iHom_{\mathcal{M}}(F^{\radj}(N), \Ser_{\mathcal{M}}(M))
    \xrightarrow{\quad \eqref{eq:iHom-mod-fun-adj-2} \quad}
    \iHom_{\mathcal{M}}(N, F^{\rradj} \Ser_{\mathcal{M}}(M))
  \end{gather*}
  for $M, N \in \mathcal{M}$ and $F \in \REX_{\mathcal{C}}(\mathcal{M}, \mathcal{N})$. Each of them is in fact an isomorphism of $\mathcal{C}$-bimodule functors from $\mathcal{N}^{\op} \times {}_{\overline{\DD}}\mathcal{M}$ to $\mathcal{C}$. Thus, by Lemma~\ref{lem:mod-prof-ff}, we obtain an isomorphism $\theta_F: \Ser_{\mathcal{N}} F \to F^{\rradj} \Ser_{\mathcal{M}}$ of twisted left $\mathcal{C}$-module functors. It is routine to check that $\theta_F$ is natural in $F$. The commutativity of the diagram follows from the commutative diagram given in Remark~\ref{rem:iHom-mod-fun-adj}.
\end{proof}

\subsection{Pivotal module categories}
\label{subsec:piv-mod-categ}

Let $\mathcal{C}$ be a pivotal finite tensor category, that is, a finite tensor category equipped with an isomorphism $\mathfrak{p}: \id_{\mathcal{C}} \to \DD$ of tensor functors (often referred to as a {\em pivotal structure} of $\mathcal{C}$). Now we introduce a pivotal structure of an exact module category as follows:

\begin{definition}[{\it cf}. Schaumann {\cite[Remark 5.3]{MR3435098}}]
  \label{def:pivotal-module-cat}
  Let $\mathcal{M}$ be an exact left $\mathcal{C}$-module category, and let $(\Ser, \phi)$ be a relative Serre functor of $\mathcal{M}$. A {\em pivotal structure} of $\mathcal{M}$ (respecting the pivotal structure $\mathfrak{p}$ of $\mathcal{C}$) is an isomorphism $\widetilde{\mathfrak{p}}: \id_{\mathcal{M}} \to \Ser$ of functors such that the diagram
  \begin{equation*}
    \begin{tikzcd}[column sep = 64pt]
      X \catactl M
      \arrow[rd, "{\mathfrak{p}_X \catactl \widetilde{\mathfrak{p}}_{M}}"']
      \arrow[rr, "{\widetilde{\mathfrak{p}}_{X \catactl M}}"]
      & & \Ser(X \catactl M) \\
      & X^{**} \catactl \Ser(M)
      \arrow[ru, "{\xi_{X,M}}"']
    \end{tikzcd}
  \end{equation*}
  commutes for all objects $X \in \mathcal{C}$ and $M \in \mathcal{M}$, where $\xi$ is the twisted $\mathcal{C}$-module structure of $\Ser$. A {\em pivotal left $\mathcal{C}$-module category} is an exact left $\mathcal{C}$-module category endowed with a pivotal structure.
\end{definition}

Stated differently, a pivotal structure of $\mathcal{M}$ is an isomorphism of $\mathcal{C}$-module functors from $\id_{\mathcal{M}}$ to a relative Serre functor $\Ser$ regarded as a (non-twisted) left $\mathcal{C}$-module functor by the structure morphism given by
\begin{equation}
  \label{eq:rel-Serre-non-twisted}
  X \catactl \Ser(M)
  \xrightarrow{\quad \mathfrak{p}_X \catactl \id_{\Ser(M)} \quad}
  X^{**} \catactl \Ser(M)
  \xrightarrow{\quad \xi_{X,M} \quad} \Ser(X \catactl M)
\end{equation}
for $X \in \mathcal{C}$ and $M \in \mathcal{M}$.

Schaumann \cite{MR3435098} introduced and studied {\em inner-product structures} on module categories. As noted in \cite[Remark 5.3]{MR3435098}, such a structure exists only if the module category in concern is exact. Furthermore, inner product structures on an exact left $\mathcal{C}$-module category $\mathcal{M}$ are in bijection with pivotal structures of $\mathcal{M}$.
Hence a result on inner-product module categories can be translated into a result on pivotal module categories.
For example, by rephrasing \cite[Proposition 5.5]{MR3435098}, we obtain:

\begin{lemma}
  \label{lem:pivot-exact-mod-cat-unique}
  Let $\mathcal{M}$ be an indecomposable exact left module category over a pivotal finite tensor category. Then a pivotal structure of $\mathcal{M}$ is, if it exists, unique up to scalar multiple.
\end{lemma}

Given a finite left $\mathcal{C}$-module category $\mathcal{M}$, the {\em dual} of $\mathcal{C}$ with respect to $\mathcal{M}$ is the $\bfk$-linear abelian monoidal category defined and denoted by
\begin{equation*}
  \mathcal{C}_{\mathcal{M}}^* := \REX_{\mathcal{C}}(\mathcal{M}, \mathcal{M})^{\rev},
\end{equation*}
where the tensor product of $\REX_{\mathcal{C}}(\mathcal{M}, \mathcal{M})$ is given by the composition of module functors. It is known that $\mathcal{C}_{\mathcal{M}}^*$ is a finite multi-tensor category provided that the module category $\mathcal{M}$ is exact \cite{MR2119143}.

In the rest of this section, we exhibit some applications of a pivotal structure of a module category. The following theorem concerns a pivotal structure of the dual tensor category.


\begin{theorem}
  \label{thm:dual-tensor-cat-pivotal}
  If $\mathcal{M}$ is a pivotal left $\mathcal{C}$-module category, then $\mathcal{C}_{\mathcal{M}}^*$ is a pivotal finite multi-tensor category.
\end{theorem}

We note that the converse of this theorem does not hold: There are examples of a pivotal finite tensor category $\mathcal{C}$ and an exact left $\mathcal{C}$-module category $\mathcal{M}$ such that $\mathcal{C}_{\mathcal{M}}^*$ admits a pivotal structure but $\mathcal{M}$ does not; see Remark~\ref{rem:non-pivotal-module-but}.

This theorem has been proved by Schaumann \cite[Corollary 5.9]{MR3435098} in terms of inner-product structures on module categories. For reader's convenience, we give a direct proof based on the above argument.

\begin{proof}[Proof of Theorem~\ref{thm:dual-tensor-cat-pivotal}]
  Let $\mathfrak{p}$ and $\widetilde{\mathfrak{p}}$ be the pivotal structure of $\mathcal{C}$ and $\mathcal{M}$, respectively, and regard $\Ser_{\mathcal{M}}$ as a (non-twisted) left $\mathcal{C}$-module functor by~\eqref{eq:rel-Serre-non-twisted}. Then we have a natural isomorphism
  \begin{equation*}
    \mathfrak{q}_F := \Big( F
    \xrightarrow{\quad \widetilde{\mathfrak{p}} \circ F \quad}
    \Ser_{\mathcal{M}} \circ F
    \xrightarrow{\quad \eqref{eq:rel-Serre-double-right-adj-iso} \quad}
    F^{\rradj} \circ \Ser_{\mathcal{M}}
    \xrightarrow{\quad F^{\rradj} \circ \widetilde{\mathfrak{p}}{}^{-1} \quad}
    F^{\rradj} \Big)
  \end{equation*}
  for $F \in \mathcal{C}_{\mathcal{M}}^*$. We remark that a {\em left} dual object of $F \in \mathcal{C}_{\mathcal{M}}^*$ is a {\em right} adjoint of $F$ (since the tensor product $\mathcal{C}_{\mathcal{M}}^*$ is given by $F \otimes G = G \circ F$). By the commutative diagram given in Theorem~\ref{thm:rel-Serre-double-right-adj}, the isomorphism $\mathfrak{q}_F: F \to F^{\rradj}$ is in fact a pivotal structure of $\mathcal{C}_{\mathcal{M}}^*$. The proof is done.
\end{proof}

\subsection{A construction of symmetric Frobenius algebras}
\label{subsec:symmetric-Frobenius-algebras}

Let $\mathcal{C}$ be a pivotal finite tensor category with pivotal structure $\mathfrak{p}$. An algebra in $\mathcal{C}$ is a synonym for a monoid in $\mathcal{C}$ \cite[VII.3]{MR1712872}.
A {\em Frobenius algebra} in $\mathcal{C}$ is a pair $(A, \lambda)$ consisting of an algebra $A$ in $\mathcal{C}$ and a morphism $\lambda: A \to \unitobj$ in $\mathcal{C}$ (called the {\em Frobenius form}) such that the morphism
\begin{equation}
  \label{eq:Frobenius-alg-psi}
  \psi := \Big( A \xrightarrow{\quad \id_A \otimes \coev_A \quad}
  A \otimes A \otimes A^*
  \xrightarrow{\quad \lambda \mu \otimes \id_A \quad}
  A^* \Big)
\end{equation}
is invertible, where $\mu: A \otimes A \to A$ is the multiplication of $A$. For a Frobenius algebra $A = (A, \lambda)$ in $\mathcal{C}$, the {\em Nakayama automorphism} of $A$ is defined by
\begin{equation*}
  \nu_A := \Big( A
  \xrightarrow{\quad \mathfrak{p}_{A} \quad} A^{**}
  \xrightarrow{\quad \psi^* \quad} A^{*}
  \xrightarrow{\quad \psi^{-1} \quad} A
  \Big).
\end{equation*}
The Frobenius algebra $A$ is said to be {\em symmetric} if $\nu_A = \id_A$. See \cite{MR2500035} for more on Frobenius algebras in a rigid monoidal category and their Nakayama automorphisms.
We note that the definition of an ordinary (symmetric) Frobenius algebra \cite[Chapter 6]{MR1653294} is recovered as the case where $\mathcal{C} = {}_{\bfk}\Mod$.

Now let $\mathcal{M}$ be an exact left $\mathcal{C}$-module category with relative Serre functor $(\Ser, \phi)$.
For every object $M \in \mathcal{M}$, the object $\iEnd(M) := \iHom(M, M)$ is an algebra in $\mathcal{C}$ with respect to the composition for the internal Hom functor. Before we state the main result of this subsection, we remark:

\begin{theorem}
  \label{thm:Frobenius-alg}
  Let $M \in \mathcal{M}$ be an object. If $p: M \to \Ser(M)$ is an isomorphism in $\mathcal{M}$, then $A := \iEnd(M)$ is a Frobenius algebra with Frobenius form
  \begin{equation*}
    \lambda := \Big( A \xrightarrow{\quad \iHom(M, p) \quad}
    \iHom(M, \Ser(M)) \xrightarrow{\quad \itrace_M \quad} \unitobj \Big).    
  \end{equation*}
  The Nakayama automorphism of $(A, \lambda)$ is given by
  \begin{equation*}
    \nu_A
    = \iHom(p, p^{-1}) \circ \phi_{M, \Ser(M)}^{-1} \circ \phi_{M,M}^{*} \circ \mathfrak{p}_A.
  \end{equation*}
\end{theorem}
\begin{proof}
  The multiplication of $A$ is given by $\mu := \icomp_{M,M,M}$. By Lemma~\ref{lem:rel-Ser-pairing} and the naturality of the composition of the internal Hom functor, we have
  \begin{align*}
    \lambda \circ \mu
    & = \itrace_M \circ \icomp_{M, M, \Ser(M)} \circ (\iHom(M, p) \otimes \id_{\iHom(M, M)}) \\
    & = \eval_{\iHom(M, M)}
      \circ ((\phi_{M,M}^{} \circ \iHom(M, p)) \otimes \id_{\iHom(M, M)}).
  \end{align*}
  Since $\phi_{M,M}$ and $p$ are isomorphisms, and since $\eval_{X}$ is non-degenerate for all $X \in \mathcal{C}$, we conclude that $\lambda \circ \mu$ is non-degenerate, that is, $(A, \lambda)$ is a Frobenius algebra. To complete the proof, we compute the Nakayama automorphism of $(A, \lambda)$. We define $\psi$ by \eqref{eq:Frobenius-alg-psi}. By the above computation, we have
  \begin{equation*}
    \psi = \phi_{M, M}^{} \circ \iHom(M, p).
  \end{equation*}
  Thus the Nakayama automorphism is computed as follows:
  \begin{align*}
    \nu_A
    & = \psi^{-1} \circ \psi^{*} \circ \mathfrak{p}_A \\
    & = \iHom(M, p^{-1}) \circ \phi_{M,M}^{-1}
      \circ \iHom(M, p)^* \circ \phi_{M,M}^{*} \circ \mathfrak{p}_A \\
    & = \iHom(M, p^{-1}) \circ \iHom(p, \Ser(M))
      \circ \phi_{M, \Ser(M)}^{-1} \circ \phi_{M,M}^{*} \circ \mathfrak{p}_A \\
    & = \iHom(p, p^{-1}) \circ \phi_{M, \Ser(M)}^{-1} \circ \phi_{M,M}^{*} \circ \mathfrak{p}_A.
  \end{align*}
  Here, the third equality follows from the naturality of $\phi$. The proof is done.
\end{proof}

Now we state the main result of this subsection:

\begin{theorem}
  \label{thm:sym-Frobenius-alg}
  Let $M$ be an object of $\mathcal{M}$. If $\mathcal{M}$ admits a pivotal structure $\widetilde{\mathfrak{p}}$, then the algebra $\iEnd(M)$ is a symmetric Frobenius algebra with Frobenius form
  \begin{equation}
    \label{eq:sym-Frobenius-form}
    t_M := \Big( \iEnd(M)
    \xrightarrow{\quad \iHom(M, \widetilde{\mathfrak{p}}_M) \quad}
    \iHom(M, \Ser(M)) \xrightarrow{\quad \itrace_M \quad} \unitobj \Big).
  \end{equation}
\end{theorem}

Before we give a proof of this theorem, we give a corollary and a remark.
The fact that $\iEnd(M)$ is a symmetric Frobenius algebra has been known in the semisimple case; see \cite[Theorem 6.6]{MR3019263}. Corollary \ref{cor:sym-Frobenius-alg} below can be thought of as a non-semisimple analogue of \cite[Theorem 6.10]{MR3019263}. Given an algebra $A$ in $\mathcal{C}$, we denote by $\mathcal{C}_A$ the category of right $A$-modules in $\mathcal{C}$. The category $\mathcal{C}_A$ has an obvious structure of a finite left $\mathcal{C}$-module category.

\begin{corollary}
  \label{cor:sym-Frobenius-alg}
  For every pivotal left $\mathcal{C}$-module category $\mathcal{M}$, there is a symmetric Frobenius algebra $A$ in $\mathcal{C}$ such that $\mathcal{M} \approx \mathcal{C}_A$ as left $\mathcal{C}$-module categories.
\end{corollary}
\begin{proof}
  Let $G$ be a projective generator of $\mathcal{M}$, and set $A = \iEnd(G)$. Then $\mathcal{M} \approx \mathcal{C}_A$ as left $\mathcal{C}$-module categories \cite{MR3934626,MR3242743}. Theorem \ref{thm:sym-Frobenius-alg} shows that $A$ is a symmetric Frobenius algebra. The proof is done.
\end{proof}

\begin{remark}
  By \cite[Theorem 5.20]{MR3435098}, $\mathcal{C}_A$ is a pivotal left $\mathcal{C}$-module category if $A$ is a {\em special} symmetric Frobenius algebra in the sense of \cite{MR2500035}.
  It would be interesting to know a necessary and sufficient condition for $\mathcal{C}_A$ being pivotal.
  Unlike the semisimple case considered in \cite{MR3019263}, the symmetric Frobenius algebra $\iEnd(M)$ of Theorem~\ref{thm:sym-Frobenius-alg} is not special in general; see Remark \ref{rem:non-symmetric-iEnd}. We also note that $\mathcal{C}_A$ for a symmetric Frobenius algebra $A$ in $\mathcal{C}$ is not necessarily an exact left $\mathcal{C}$-module category (consider the case where $\mathcal{C} = {}_{\bfk}\Mod$ and take $A$ to be any non-semisimple symmetric Frobenius algebra over $\bfk$).
\end{remark}

Now we go into the proof of Theorem~\ref{thm:sym-Frobenius-alg}. By Theorem~\ref{thm:Frobenius-alg}, $\iEnd(M)$ is a Frobenius algebra with Frobenius form $t_M$ given by~\eqref{eq:sym-Frobenius-form}. We shall show that the Nakayama automorphism of $\iEnd(M)$ is the identity morphism. For this purpose, we need Lemmas \ref{lem:mod-str-standard-rel-Serre-2} and~\ref{lem:mod-str-standard-rel-Serre-3} below.

For a while, we consider the general setting that $\mathcal{C}$ and $\mathcal{M}$ may not have pivotal structures. Let $(\Ser, \phi)$ be a relative Serre functor of $\mathcal{M}$ and regard $\Ser$ as an equivalence $\Ser: {}_{\overline{\DD}}\mathcal{M} \to \mathcal{M}$ of left $\mathcal{C}$-module categories. Since the internal Hom functor of ${}_{\overline{\DD}}\mathcal{M}$ is given by $\iHom(-, -)^{**}$, the equivalence $\Ser$ induces an isomorphism
\begin{equation*}
  \underline{\Ser}|_{M,N}: \iHom(M, N)^{**} \to \iHom(\Ser(M), \Ser(N))
  \quad (M, N \in \mathcal{M})
\end{equation*}
of twisted $\mathcal{C}$-bimodule functors (see Remark~\ref{rem:iHom-mod-fun-induced-map}).

\begin{lemma}
  \label{lem:mod-str-standard-rel-Serre-2}
  For all objects $M, N \in \mathcal{M}$, we have
  \begin{equation*}
    \underline{\Ser}|_{M,N} = \phi_{N,\Ser(M)}^{-1} \circ \phi_{M,N}^*.
  \end{equation*}
\end{lemma}
\begin{proof}
  Let $\xi_{X,M}: X^{**} \catactl \Ser(M) \to \Ser(X \catactl M)$ be the twisted $\mathcal{C}$-module structure of $\Ser$. To save space, we denote by $[ , ]$ the internal Hom functor. We consider the following diagram:
  \begin{equation*}
    \footnotesize
    \begin{tikzcd}[column sep = 24pt]
      [M, N]^{**} \otimes \unitobj
      \arrow[d, equal]
      \arrow[rr, "{\id \otimes \icoev_{\unitobj, \Ser(M)}}"]
      & \arrow[d, phantom, "{(\spadesuit)}"]
      & {} [M, N]^{**} \otimes [\Ser(M), \Ser(M)]
      \arrow[d, equal, d] \\
      {} [M, N]^{**}
      \arrow[dr, phantom, "{(\clubsuit)_{\strut}}"]
      \arrow[d, "{\underline{\Ser}|_{M, N}}"']
      \arrow[r, "{\icoev}"]
      & {} [\Ser(M), [M, N]^{**} \catactl \Ser(M)]
      \arrow[d, "{[\id, \xi_{[M, N], M}]}"]
      \arrow[r, "{\mathfrak{a}^{-1}}"]
      & {} [M, N]^{**} \otimes [\Ser(M), \Ser(M)]
      \arrow[d, "{\id \otimes \phi_{M, \Ser(M)}}"] \\
      {} [\Ser(M), \Ser(N)]
      \arrow[dr, phantom, "{(\heartsuit)}"]
      \arrow[d, "{\phi_{N, \Ser(M)}}"']
      & {} [\Ser(M), \Ser([M, N] \catactl M)]
      \arrow[l, "{[\id, \Ser(\ieval)]}"']
      \arrow[d, "{\phi_{[M,N] \catactl N, \Ser(M)}}"]
      \arrow[r, phantom, "{(\diamondsuit)}"]
      & {} [M, N]^{**} \otimes [M, \Ser(M)]^*
      \arrow[d, equal] \\
      {} [N, \Ser(M)]^*
      & {} [[M, N] \catactl M, \Ser(M)]^*
      \arrow[l, "{[\ieval, \id]^*}"]
      \arrow[r, "{(\mathfrak{b}^{-1})^*}"]
      & ([M, \Ser(M)] \otimes [M, N]^{*})^* \\
      {} [N, \Ser(M)]^* \arrow[u, equal]
      & ([N, \Ser(M)] \otimes [M, N] \otimes [M, N]^*)^*
      \arrow[l, "{(\id \otimes \coev)^*{}^{\strut}}"]
      \arrow[u, phantom, "{(\spadesuit)}"]
      & ([M, \Ser(M)] \otimes [M, N]^{*})^*
      \arrow[l, "{(\icomp \otimes \id)^*{}^{\strut}}"]
      \arrow[u, equal]
    \end{tikzcd}
  \end{equation*}
  One can directly check that two cells labeled $(\spadesuit)$ are commutative (see Appendix of \cite{MR4077241}). The cell $(\clubsuit)$ is commutative by the definition of $\underline{\Ser}$ (see Remark~\ref{rem:iHom-mod-fun-induced-map}). The cell $(\heartsuit)$ is commutative by the naturality of $\phi$. Finally, since $\phi$ is a morphism of $\mathcal{C}$-bimodule functors, the cell $(\diamondsuit)$ is commutative. Hence the above diagram commutes. Now the diagram yields the equation
  \begin{equation}
    \label{eq:lem:mod-str-standard-rel-Serre-2-proof-1}
    \begin{aligned}
      \phi_{N,\Ser(M)} \circ \underline{\Ser}|_{M, N}
      & = (\id_{[N, \Ser(M)]} \otimes \coev_{[M, N]})^*
      \circ (\icomp_{M,N,\Ser(M)} \otimes \id_{[M, N]^*})^* \\
      & \qquad \circ (\id_{[M, N]^{**}} \otimes (\phi_{M, \Ser(M)} \circ \icoev_{\unitobj, \Ser(M)})).
    \end{aligned}
  \end{equation}
  By Lemma~\ref{lem:rel-Ser-pairing}, we have
  \begin{gather*}
    {}^*(\phi_{M, \Ser(M)} \circ \icoev_{\unitobj, \Ser(M)})
    = \eval_{[M, \Ser(M)]} \circ
    ((\phi_{M, \Ser(M)} \circ \icoev_{\unitobj, \Ser(M)}) \otimes \id_{[M, \Ser(M)]}) \\
    = \itrace_M \circ \icomp_{M, \Ser(M), \Ser(M)}
    \circ (\icoev_{\unitobj, \Ser(M)} \otimes \id_{[M, \Ser(M)]}) = \itrace_M.
  \end{gather*}
  Thus, by \eqref{eq:lem:mod-str-standard-rel-Serre-2-proof-1} and Lemma~\ref{lem:rel-Ser-pairing}, we compute
  \begin{align*}
    & {}^*(\phi_{N,\Ser(M)} \circ \underline{\Ser}|_{M, N}) \\
    & = (\itrace_M \otimes \id_{[M,N]^*})
      \circ (\icomp_{M,N,\Ser(M)} \otimes \id_{[M, N]^*})
      \circ (\id_{[N, \Ser(M)]} \otimes \coev_{[M, N]}) \\
    & = ((\eval_{[M, N]} \circ (\phi_{M,N} \otimes \id_{\iHom(M, N)})) \otimes \id_{[M,N]^*})
      \circ (\id_{[N, \Ser(M)]} \otimes \coev_{[M, N]}) \\
    & = \phi_{M,N}.
  \end{align*}
  One easily deduce the desired formula from this result. The proof is done.
\end{proof}

Now we suppose that $\mathcal{C}$ has a pivotal structure $\mathfrak{p}$ and $\mathcal{M}$ has a pivotal structure $\widetilde{\mathfrak{p}}$ respecting $\mathfrak{p}$. Then we have:

\begin{lemma}
  \label{lem:mod-str-standard-rel-Serre-3}
  For all objects $M, N \in \mathcal{M}$, we have
  \begin{equation}
    \label{eq:mod-str-standard-rel-Serre-3-proof-1}
    \underline{\Ser}|_{M, N}
    = \iHom(\widetilde{\mathfrak{p}}{}_M^{-1}, \widetilde{\mathfrak{p}}_N^{})
    \circ \mathfrak{p}_{\iHom(M, N)}^{-1}.
  \end{equation}
\end{lemma}
\begin{proof}
  Let $M$ and $N$ be objects of $\mathcal{M}$. To save space, we denote by $[ , ]$ the internal Hom functor. We consider the following diagram:
  \begin{equation*}
    \begin{tikzcd}[column sep = 64pt]
      [\Ser(M), N] \catactl \Ser(M)
      \arrow[rr, "{\ieval_{\Ser(M), N}}"]
      & & N \arrow[d, equal] \\
      {} [M, N] \catactl \Ser(M)
      \arrow[u, "{[\widetilde{\mathfrak{p}}^{-1}_M, N] \catactl \id_{\Ser(M)}}"]
      \arrow[r, "{\id_{[M, N]} \catactl \widetilde{\mathfrak{p}}_M^{-1}}"]
      &{} [M, N] \catactl M
      \arrow[d, "{\widetilde{\mathfrak{p}}_{[M, N] \catactl M}}"]
      \arrow[r, "{\ieval_{M, N}}"]
      & N \arrow[d, "{\widetilde{\mathfrak{p}}_N}"] \\
      {} [M, N]^{**} \catactl \Ser(M)
      \arrow[u, "{\mathfrak{p}_{[M, N]}^{-1} \catactl \id_{\Ser(M)}}"]
      \arrow[r, "{\xi_{[M, N], M}}"]
      & \Ser([M, N] \catactl M)
      \arrow[r, "{\Ser(\ieval_{M,N})}"]
      & \Ser(N)
    \end{tikzcd}
  \end{equation*}
  By the definition of a pivotal structure of $\mathcal{M}$, the lower left square is commutative. The lower right square is commutative by the naturality of $\widetilde{\mathfrak{p}}$. The top square is commutative by the dinaturality of $\ieval_{M, N}$ in the variable $M$. Thus we have
  \begin{equation*}
    \begin{aligned}
      & \Ser(\ieval_{M, N}) \circ \xi_{[M, N], M} \\
      & = \widetilde{\mathfrak{p}}_N \circ \ieval_{\Ser(M), N}
      \circ ([\widetilde{\mathfrak{p}}_M^{-1}, N] \catactl \id_{\Ser(M)})
      \circ (\mathfrak{p}_{[M, N]}^{-1} \catactl \id_{\Ser(M)}) \\
      & = \ieval_{\Ser(M), \Ser(N)} \circ (([\widetilde{\mathfrak{p}}_M^{-1}, \widetilde{\mathfrak{p}}_N] \circ \mathfrak{p}_{[M, N]}^{-1}) \catactl \id_{\Ser(M)})
    \end{aligned}
  \end{equation*}
in $\Hom_{\mathcal{\mathcal{M}}}([M, N]^{**} \catactl \Ser(M), \Ser(N))$, where the first equality follows from the commutativity of the above diagram and the second one follows from the naturality of $\ieval_{\Ser(M), -}$.
  There is a canonical isomorphism
  \begin{equation*}
    \begin{aligned}
      \Hom_{\mathcal{M}}([M, N]^{**} \catactl \Ser(M), \Ser(N))
      & \to \Hom_{\mathcal{C}}([M, N]^{**}, [\Ser(M), \Ser(N)]), \\
      f & \mapsto [\id_{\Ser(N)}, f] \circ \icoev_{[M, N], \Ser(M)}
    \end{aligned}
  \end{equation*}
  by the definition of the internal Hom functor. This isomorphism sends the left and the right hand side of the equation
  \begin{equation*}
    \Ser(\ieval_{M, N}) \circ \xi_{[M, N], M}
    = \ieval_{\Ser(M), \Ser(N)} \circ (([\widetilde{\mathfrak{p}}_M^{-1}, \widetilde{\mathfrak{p}}_N] \circ \mathfrak{p}_{[M, N]}^{-1}) \catactl \id_{\Ser(M)})
  \end{equation*}
  to the left and the right hand side of \eqref{eq:mod-str-standard-rel-Serre-3-proof-1}, respectively (see Remark \ref{rem:iHom-mod-fun-induced-map} for the left hand side). The proof is done.
\end{proof}

\begin{proof}[Proof of Theorem~\ref{thm:sym-Frobenius-alg}]
  By Theorem~\ref{thm:Frobenius-alg} and the above two lemmas, the Nakayama automorphism of $(\iEnd(M), t_M)$ is the identity morphism. The proof is done.
\end{proof}

\section{Relative Serre functor for comodule algebras}
\label{sec:rel-Serre-comod-alg}

Let $H$ be a finite-dimensional Hopf algebra.
Then an exact module category over $\mathcal{C} = {}_H \Mod$ is equivalent to ${}_L\Mod$ for some exact left $H$-comodule algebra $L$ (see Definition \ref{def:exact-comod-alg}).
In this section, we describe category-theoretical notions appearing in the theory of relative Serre functors in a module-theoretic terms.
We use these results to give formulas of the relative Serre functor for ${}_L\Mod$ and related natural isomorphisms (Theorems \ref{thm:main-theorem} and \ref{thm:main-theorem-2}).
We also discuss when ${}_L\Mod$ possesses a pivotal structure and give several remarks about this.

\subsection{Reminder on Frobenius algebras}

Throughout this section, we work over a field $\bfk$.
Unless otherwise noted, the unadorned symbol $\otimes$ means the tensor product over the field $\bfk$.
Given a vector space $V$, we denote its dual space by $V^*$.

Before we go into the main topic, we recall basics on Frobenius algebras \cite[Chapter 6]{MR1653294}. Let $L$ be a finite-dimensional algebra. Then $L^*$ is an $L$-bimodule by
\begin{equation*}
  a \rightharpoonup f \leftharpoonup b = \langle f, b \, ? \, a \rangle
  \quad (f \in L^*, a, b \in L).
\end{equation*}
The Nakayama functor $\Nak_L: {}_L\Mod \to {}_L\Mod$ is defined by $\Nak_L(M) = L^* \otimes_L M$. This functor is an equivalence if and only if $L$ is quasi-Frobenius, that is, the class of projective objects in ${}_L\Mod$ coincides with the class of injective objects of ${}_L\Mod$.

Suppose that $L$ is a Frobenius algebra with Frobenius form $\lambda_L: L \to \bfk$.
Then, by definition, the linear map $\theta_L: L \to L^*$ defined by $\theta_L(a) = \lambda_L \leftharpoonup a$ is bijective. The Nakayama automorphism of $L$ (with respect to $\lambda_L$) is the algebra automorphism $\nu_L$ of $L$ characterized by
\begin{equation}
  \label{eq:def-Nakayama}
  \langle \lambda_L, b a \rangle
  = \langle \lambda_L, \nu_L(a) b \rangle
  \quad (a, b \in L).
\end{equation}
Given a left $L$-module $M$ and an algebra automorphism $f$ of $L$, we denote by ${}_{(f)}M$ the left $L$-module obtained from $M$ by twisting the action of $L$ by $f$.
The map $\theta_L$ is in fact an isomorphism ${}_{(\nu_L)}L \to L^*$ of $L$-bimodules. Hence the functor $\Nak_L$ is isomorphic to ${}_{(\nu_L)}(-)$. In particular, we have
\begin{equation}
  \label{eq:Frobenius-Nakayama-preserves-dim}
  \dim_{\bfk} \Nak_L(V) = \dim_{\bfk} V
\end{equation}
for all $V \in {}_L \Mod$. According to \cite[\S16D]{MR1653294}, we have:

\begin{lemma}
  \label{lem:Frobenius-Nakayama-dim}
  A quasi-Frobenius algebra $L$ admits a structure of a Frobenius algebra if and only if the equation \eqref{eq:Frobenius-Nakayama-preserves-dim} holds for all simple objects $V \in {}_L \Mod$.
\end{lemma}

Let $L$ be a finite-dimensional algebra.
A {\em Frobenius system} for $L$ is a triple $(\lambda_L,\allowbreak \{ a^i \}_{i = 1}^r,\allowbreak \{ b_i \}_{i = 1}^r)$ consisting of a linear map $\lambda_L: L \to \bfk$ and two bases $\{ a^i \}$ and $\{ b_i \}$ of $L$ such that $\langle \lambda_L, a^i b_j \rangle = \Kdelta_{i j}$ for all $i, j = 1, \cdots, r$, where $r$ is the dimension of $L$ over $\bfk$ and $\Kdelta$ means the Kronecker delta. We note that $L$ is a Frobenius algebra if and only if a Frobenius system for $L$ exists.
The following equations involving the Frobenius system are well-known; see, {\it e.g.}, \cite[Section 2]{MR1865524}.

\begin{lemma}
  \label{lem:Frob-alg-dual-basis}
  Let $L$ be a Frobenius algebra with Frobenius system as above. Then, for all elements $c \in L$, we have
  \begin{gather}
    \label{eq:Frob-str-L-1}
    \quad \langle \lambda_L, a^i \rangle b_i = 
    1_L = \langle \lambda_L, b_i \rangle a^i, \\
    \label{eq:Frob-str-L-2}
    a^i c \otimes b_i = a^i \otimes c b_i,
    \quad \nu_L(c) a^i \otimes b_i = a^i \otimes b_i c,
  \end{gather}
  where the Einstein convention is used to suppress sums over $i$.
\end{lemma}

By the above lemma, it is straightforward to verify:

\begin{lemma}
  \label{lem:Frob-alg-Hom}
  Notations are as in the above lemma. For $M \in {}_L \Mod$, the map
  \begin{equation}
    \Hom_L(M, L) \to M^*,
    \quad f \mapsto \lambda_L \circ f
  \end{equation}
  is a natural isomorphism of vector spaces with the inverse
  \begin{equation}
    M^* \to \Hom_L(M, L), \quad m^* \mapsto \langle m^*, a^i \mathord{?} \rangle b_i.
  \end{equation}
\end{lemma}

\subsection{Hopf algebras and comodule algebras}

Till the end of this section, we fix a finite-dimensional Hopf algebra $H$ and denote by $\mathcal{C}$ the finite tensor category ${}_H \Mod$ of finite-dimensional left $H$-modules.
The comultiplication, the counit and the antipode of $H$ will be denoted by $\Delta$, $\varepsilon$ and $S$, respectively. The Sweedler notation, such as
\begin{equation*}
  \Delta(h) = h_{(1)} \otimes h_{(2)},
  \quad
  \Delta(h_{(1)}) \otimes h_{(2)} = h_{(1)} \otimes h_{(2)} \otimes h_{(3)} = h_{(1)} \otimes \Delta(h_{(2)}),
\end{equation*}
will be used to express the comultiplication. The following notation for the action of the dual algebra $H^*$ on $H$ will also be used:
\begin{equation*}
  f \rightharpoonup h := h_{(1)} \langle f, h_{(2)} \rangle,
  \quad h \leftharpoonup f := \langle f, h_{(1)} \rangle h_{(2)}
  \quad (f \in H^*, h \in H).
\end{equation*}

A left $H$-comodule algebra is an algebra $L$ equipped with a left $H$-comodule structure such that the coaction $L \to H \otimes L$ is an algebra map. Let $L$ be a finite-dimensional left $H$-comodule algebra with coaction $\delta_L: L \to H \otimes L$. We also use a Sweedler-type notation, such as
\begin{gather*}
  \delta_L(a) = a_{(-1)} \otimes a_{(0)} \in H \otimes L, \\
  \Delta(a_{(-1)}) \otimes a_{(0)}
  = a_{(-2)} \otimes a_{(-1)} \otimes a_{(0)}
  = a_{(-1)} \otimes \delta_L(a_{(0)}),
\end{gather*}
to express the coaction of $H$. The category ${}_L \Mod$ is a finite left $\mathcal{C}$-module category by the action defined as follows: We set $X \catactl M = X \otimes M$ for $X \in \mathcal{C}$ and $M \in {}_L \Mod$ as a vector space. The action of $L$ on $X \catactl M$ is given by
\begin{equation*}
  a \cdot (x \otimes m) = a_{(-1)} x \otimes a_{(0)} m
  \quad (a \in L, x \in X, m \in M).
\end{equation*}

A left $H$-comodule algebra is said to be {\em $H$-simple from the right} (respectively, {\em $H$-simple}) if it has no non-trivial right (respectively, two-sided) ideal stable under the coaction of $H$ \cite[Definition 1.2]{MR2331768}.
According to \cite[Theorem 3.3]{MR2331768}, every indecomposable exact left $\mathcal{C}$-module category is equivalent, as a left $\mathcal{C}$-module category, to ${}_L \Mod$ for some finite-dimensional left $H$-comodule algebra $L$ such that $L$ is $H$-simple from the right and the space of $H$-coinvariants of $L$ is trivial.
Now we introduce the following terminology:

\begin{definition}
  \label{def:exact-comod-alg}
  An {\em exact left $H$-comodule algebra} is a finite-dimensional left $H$-comodule algebra $L$ such that ${}_L \Mod$ is an exact left $\mathcal{C}$-module category.
\end{definition}

According to Skryabin \cite[Theorem 4.2]{MR2286047}, a finite-dimensional $H$-simple left $H$-comodule algebra is Frobenius.
This result implies that a large class of exact left $H$-comodule algebras, including left coideal subalgebras of $H$ \cite[Propositions 1.6 and 1.20]{MR2331768}, is Frobenius. Slightly generalizing this consequence, we prove:

\begin{lemma}
  \label{lem:exact-comod-alg-Frobenius}
  An exact left $H$-comodule algebra is a Frobenius algebra.
\end{lemma}
\begin{proof}
  Let $L$ be an exact left $H$-comodule algebra.
  If $\mathcal{M} := {}_L \Mod$ is decomposed into a direct sum of $\mathcal{C}$-module full subcategories as $\mathcal{M} =\mathcal{M}_1 \oplus \dotsb \oplus \mathcal{M}_r$, then there are two-sided $H$-subcomodule ideals $L_i \subset L$ ($i = 1, \cdots, r$) such that $L = L_1 \oplus \dotsb \oplus L_r$ and ${}_{L_i}\Mod = \mathcal{M}_i$ for all $i = 1, \cdots, r$ \cite[Proposition 1.18]{MR2331768}.
  It is obvious that each $L_i$ is exact.
  Since a finite product of Frobenius algebras is again Frobenius, the theorem reduces to the case where $\mathcal{M}$ is indecomposable. Thus we may, and do, assume that $\mathcal{M}$ is indecomposable.

  We consider two functions $\mathsf{d}(M) = \dim_{\bfk} M$ and $\mathsf{d}'(M) = \dim_{\bfk} \Nak(M)$ defined on the class of objects of $\mathcal{M}$, where $\Nak := \Nak_L$ is the Nakayama functor. By Lemma~\ref{lem:Frobenius-Nakayama-dim}, it suffices to show that the equation $\mathsf{d} = \mathsf{d}'$ holds.
  There is a natural isomorphism $\Nak(X \catactl M) \cong {}^{**}X \catactl \Nak(M)$ for $X \in \mathcal{C}$ and $M \in {}_L \Mod$ \cite{MR4042867}. Thus we have
  \begin{equation*}
    \mathsf{d}(X \catactl M) = \dim_{\bfk}(X) \mathsf{d}(M)
    \quad \text{and} \quad
    \mathsf{d}'(X \catactl M) = \dim_{\bfk}(X) \mathsf{d}'(M)
  \end{equation*}
  for all $X \in \mathcal{C}$ and $M \in \mathcal{M}$. It follows from the theory of Frobenius-Perron dimensions for module categories \cite[Subsection 2.6]{MR3600085} that there exists $\lambda \in \mathbb{R}_{>0}$ such that $\mathsf{d}'(M) = \lambda \mathsf{d}(M)$ for all $M \in \mathcal{M}$.

  Let $\overline{\DD} = {}^{**}(-)$ denote the double right dual functor on $\mathcal{C}$.
  By the categorical Radford $S^4$-formula \cite{MR2097289}, there is a positive integer $n$ such that $\overline{\DD}{}^n$ is isomorphic to $\id_{\mathcal{C}}$ as a tensor functor. For such an integer $n$, the $n$-th power of $\Nak$ becomes an invertible object of the dual tensor category $\mathcal{C}_{\mathcal{M}}^*$. Thus $\Nak^{n m}$ is isomorphic to $\id_{\mathcal{C}}$ in $\mathcal{C}_{\mathcal{M}}^*$ for some positive integer $m$. This implies $\lambda^{n m} = 1$. Since $\lambda \in \mathbb{R}_{> 0}$, we have $\lambda = 1$, that is, $\mathsf{d} = \mathsf{d}'$. The proof is done.
\end{proof}

\subsection{Module functors as equivariant bimodules}
\label{subsec:equivariant-bimodule}

Categories of module functors are used in the general theory of relative Serre functors. Here we recall a description of the category of right exact module functors from \cite{MR2331768} and discuss how right adjoints of module functors are given.

Let $A$ and $B$ be finite-dimensional left $H$-comodule algebras. Since they are algebras in the monoidal category ${}^H \Mod$ of finite-dimensional left $H$-comodules, the category ${}^H_A \Mod_B^{}$ of $A$-$B$-bimodules in ${}^H \Mod$ is defined. By definition, an object of this category is a finite-dimensional $A$-$B$-bimodule $M$ equipped with a left $H$-comodule structure $\delta_M: M \to H \otimes M$ such that the equation
\begin{equation}
  \label{eq:equivariant-bimod-def}
  \delta_M(a m b) = a_{(-1)} m_{(-1)} b_{(-1)} \otimes a_{(0)} m_{(0)} b_{(0)}
\end{equation}
holds for all $a \in A$, $b \in B$ and $m \in M$.

A finite-dimensional $A$-$B$-bimodule $M$ gives rise to a $\bfk$-linear right exact functor $M \otimes_B (-)$ from ${}_B \Mod$ to ${}_A \Mod$. If $M \in {}_A^H \Mod_B^{}$, then the functor $M \otimes_B (-)$ is an (oplax) left $\mathcal{C}$-module functor with the structure morphism given by
\begin{equation}
  \label{eq:equivariant-EW-mod-str}
  \begin{aligned}
    M \otimes_B (X \catactl N)
    & \to X \catactl (M \otimes_B N), \\
    m \otimes_B (x \otimes n)
    & \mapsto m_{(-1)} x \otimes (m_{(0)} \otimes_B n)
  \end{aligned}
\end{equation}
for $x \in X \in {}_H \Mod$, $n \in N \in {}_B \Mod$ and $m \in M$. The functor
\begin{equation}
  \label{eq:equivariant-EW}
  {}^H_A \Mod_B^{} \to \REX_{\mathcal{C}}({}_B \Mod, {}_A \Mod),
  \quad M \mapsto M \otimes_B (-)
\end{equation}
is in fact an equivalence of categories \cite[Proposition 1.23]{MR2331768}.

\subsubsection{Right adjoint of a module functor by Hom}

Let $M$ be an object of ${}_A^{H} \Mod_B^{}$. Since the functor $\Hom_A(M, -)$ is right adjoint to $M \otimes_B (-)$, it is a (lax) left $\mathcal{C}$-module functor by Lemma~\ref{lem:C-mod-func-adj}. For later use, we give an explicit formula of the structure map of $\Hom_{A}(M, -)$ and its inverse:

\begin{lemma}
  \label{lem:Hom-mod-str}
  Let $M$ be as above. Then the functor $\Hom_{A}(M, -): {}_A \Mod \to {}_B \Mod$ has a structure of a left $\mathcal{C}$-module functor given by
  \begin{equation}
    \label{eq:Hom-mod-str}
    \begin{aligned}
      X \catactl \Hom_A(M, N) & \to \Hom_A(M, X \catactl N) \\
      x \otimes f & \mapsto [ m \mapsto m_{(-1)} x \otimes f(m_{(0)}) ]
    \end{aligned}
  \end{equation}
  for $x \in X \in \mathcal{C}$, $N \in {}_A \Mod$ and $f \in \Hom_A(M, N)$.
  The map \eqref{eq:Hom-mod-str} is an isomorphism with the inverse given by
  \begin{equation}
    \label{eq:Hom-mod-str-inverse}
    \begin{aligned}
      g & \mapsto x_i \otimes [ m \mapsto \langle x^i, S(m_{(-1)}) g(m_{(0)})_{X} \rangle g(m_{(0)})_{N} ]
    \end{aligned}
  \end{equation}
  for $g \in \Hom_{A}(M, X \catactl N)$, where $g(m) \in X \otimes N$ for $m \in M$ is written symbolically as $g(m) = g(m)_{X} \otimes g(m)_N$ despite that it may not be a single tensor in general, $\{ x_i \}$ is a basis of the vector space $X$, $\{ x^i \}$ is the dual basis of $X^*$ and the Einstein convention is used to suppress the sum over $i$.
\end{lemma}

An expression of the form $[ m \mapsto (...) ]$ in this lemma indicates the map sending an element $m$ to the element expressed by (...). For example, \eqref{eq:Hom-mod-str-inverse} is interpreted as follows: Given $g \in \Hom_{\bfk}(M, X \catactl N)$ and $x^* \in X^*$, we consider the map
\begin{equation*}
  \psi(g, x^*): M \to N,
  \quad m \mapsto \langle x^*, S(m_{(-1)}) g(m_{(0)})_{X} \rangle g(m_{(0)})_{N}.
\end{equation*}
One can check that $\psi(g, x^*) \in \Hom_A(M, N)$ whenever $g \in \Hom_{A}(M, X \catactl N)$. The expression \eqref{eq:Hom-mod-str-inverse} actually indicates the following map:
\begin{equation*}
  \Hom_A(M, X \catactl N) \to X \catactl \Hom_A(M, N),
  \quad g \mapsto x_i \otimes \psi(g, x^i).
\end{equation*}
Expressions of the form $[ m \mapsto (...) ]$ are used throughout this section. We leave the reader to check the well-definedness of formulas involving such expressions.

\begin{proof}[Proof of Lemma \ref{lem:Hom-mod-str}]
  The unit and the counit of the tensor-Hom adjunction for the bimodule $M$ are given respectively by
  \begin{gather*}
    \eta_V: V \to \Hom_A(M, M \otimes_B V),
    \quad v \mapsto [ m \mapsto m \otimes_B v ], \\
    \varepsilon_W: M \otimes_B \Hom_A(M, W) \to W,
    \quad m \otimes_B f \mapsto f(m)
  \end{gather*}
  for $m \in M$, $v \in V \in {}_B \Mod$, $W \in {}_A \Mod$ and $f \in \Hom_A(M, W)$. The left $\mathcal{C}$-module structure of $\Hom_A(M, -)$ is thus given by the composition
  \begin{equation*}
    \newcommand{\xarr}[1]{\xrightarrow{\makebox[7em]{$\scriptstyle #1$}}}
    \begin{aligned}
      X \catactl \Hom_A(M, N)
      & \xarr{\quad \eta_{X \catactl \Hom_A(M, N)}^{} \quad}
        \Hom_A(M, M \otimes_B (X \catactl \Hom_A(M, N))) \\
      & \xarr{\quad \eqref{eq:equivariant-EW-mod-str} \quad}
        \Hom_A(M, X \catactl (M \otimes_B \Hom_A(M, N))) \\
      & \xarr{\quad \Hom_A(M, X \catactl \varepsilon_N) \quad} \Hom_A(M, X \catactl N),
    \end{aligned}
  \end{equation*}
  which results \eqref{eq:Hom-mod-str}. By the proof of \cite[Lemma 2.10]{MR3934626}, the inverse of \eqref{eq:Hom-mod-str} is given by the composition
  \begin{equation*}
    \newcommand{\xarr}[1]{\xrightarrow{\makebox[10em]{$\scriptstyle #1$}}}
    \begin{aligned}
      \Hom_A(M, X \catactl N)
      \xarr{\coev_X \catactl \Hom_A(M, X \catactl N)}
      & \, X \catactl X^* \catactl \Hom_A(M, X \catactl N) \\
      \xarr{\eqref{eq:Hom-mod-str}}
      & \, X \catactl \Hom_A(M, X^* \catactl X \catactl N) \\
      \xarr{X \catactl \Hom_A(M, \eval_X \catactl N)}
      & \, X \catactl \Hom_A(M, N),
    \end{aligned}
  \end{equation*}
  which results \eqref{eq:Hom-mod-str-inverse}. The proof is done.
\end{proof}

\begin{remark}
  \label{rem:equivariant-EW-adjoint}
  By the above lemma and \cite[Proposition 1.23]{MR2331768}, we have a category equivalence
  \begin{equation}
    \label{eq:equivariant-EW-lex}
    ({}_A^{H} \Mod_B^{})^{\op} \to \LEX_{\mathcal{C}}({}_A \Mod, {}_B \Mod),
    \quad M \mapsto \Hom_A(M, -)
  \end{equation}
  such that the following diagram is commutative up to the isomorphism given by the uniqueness of a right adjoint:
  \begin{equation}
    \label{eq:equivariant-EW-adj}
    \begin{tikzcd}[column sep = 32pt]
      & ({}_A^{H} \Mod_B^{})^{\op}
      \arrow[ld, "{\eqref{eq:equivariant-EW}}"']
      \arrow[rd, "{\eqref{eq:equivariant-EW-lex}}"] \\
      \REX_{\mathcal{C}}({}_B \Mod, {}_A \Mod)^{\op}
      \arrow[rr, "{(-)^{\radj}}"]
      & & \LEX_{\mathcal{C}}({}_A \Mod, {}_B \Mod)
    \end{tikzcd}
  \end{equation}
  Now we consider the case where $A = H$. Then ${}_A^{H} \Mod {}_{B}^{} = {}_H^{H} \Mod_{B}^{}$ is a right $\mathcal{C}$-module category by the action defined as follows: For $M \in {}_H^{H} \Mod_{B}^{}$ and $X \in \mathcal{C}$, we set $M \catactr X = M \otimes X$ as a vector space and equip it with the left $H$-comodule structure and the $H$-$B$-bimodule structure defined by
  \begin{equation*}
    m \otimes x \mapsto m_{(-1)} \otimes m_{(0)} \otimes x
    \quad \text{and} \quad
    h \cdot (m \otimes x) \cdot b
    = h_{(1)} m b \otimes h_{(2)} x,
  \end{equation*}
  respectively, for $m \in M$, $x \in X$, $h \in H$ and $b \in B$.
  The categories $\REX_{\mathcal{C}}({}_B \Mod, \mathcal{C})$ and $\LEX_{\mathcal{C}}(\mathcal{C}, {}_B \Mod)$ are a right and a left $\mathcal{C}$-module category by \eqref{eq:REX-M-N-right-action} and  \eqref{eq:LEX-M-N-left-action}, respectively. The functor~\eqref{eq:equivariant-EW} is an equivalence ${}^H_H \Mod_B^{} \to \REX_{\mathcal{C}}({}_B \Mod, \mathcal{C})$ of right $\mathcal{C}$-module categories together with the natural isomorphism given by
  \begin{equation}
    \label{eq:equivariant-EW-bicomod-alg-1}
    (M \catactr X) \otimes_B N
    \to (M \otimes_B N) \otimes X,
    \quad (m \otimes x) \otimes_B n
    \mapsto (m \otimes_B n) \otimes x
  \end{equation}
  for $m \in M \in {}_H^{H} \Mod_B^{}$, $x \in X \in \mathcal{C}$ and $n \in N \in {}_B \Mod^{}$. The functor~\eqref{eq:equivariant-EW-lex} is an equivalence of left $\mathcal{C}$-module categories by the natural isomorphism
  \begin{equation}
    \label{eq:equivariant-EW-bicomod-alg-2}
    \Hom_H(M \catactr {}^* \! X, N)
    \cong \Hom_H(M, N \otimes X)
    \quad (M \in {}^{H}_H\Mod_B^{}, N, X \in \mathcal{C})
  \end{equation}
  given by the duality of $\mathcal{C}$. One can check that \eqref{eq:equivariant-EW-adj} with $A = H$ is in fact a diagram of left $\mathcal{C}$-module functors commuting up to the isomorphism given by the uniqueness of a right adjoint.
\end{remark}

\subsubsection{Right adjoint of a module functor by tensor}
\label{subsubsec:ra-by-tensor}

We recall the following basic result on finitely generated projective modules:
Given a left module $M$ over a ring $A$, we define $M^{\dagger} = \Hom_A(M, A)$.
There is a natural transformation
\begin{equation}
  \label{eq:dagger-eval}
  e_{M,N}: M^{\dagger} \otimes_A N \to \Hom_A(M, N),
  \quad f \otimes_A n \mapsto [ m \mapsto f(m) n ]
\end{equation}
for a left $A$-module $N$.
According to the dual basis lemma \cite[\S2B]{MR1653294}, $P$ is finitely generated projective if and only if there are finitely many elements $m_1, \dotsc, m_k \in M$ and $m^1, \dotsc m^k \in M^{\dagger}$ such that $m^i(x) \cdot m_i = x$ for all $x \in M$ (such a family $\{ m_i, m^i \}$ is called {\em a pair of dual bases} for $M$ over $R$).
Given a pair $\{ m_i, m^i \}$ of dual bases for $M$ over $R$, the inverse of \eqref{eq:dagger-eval} is given by
\begin{equation}
  \label{eq:dagger-eval-inv}
  e_{M,N}^{-1}: \Hom_A(M, N) \to M^{\dagger} \otimes_A N,
  \quad f \mapsto m^i \otimes_A f(m_i).
\end{equation}

Now let, as before, $A$ and $B$ be finite-dimensional left $H$-comodule algebras.
As we have observed in Lemma \ref{lem:Hom-mod-str}, the functor $\Hom_A(M, -)$ for $M \in {}_A^H\Mod_B$ is a left $\mathcal{C}$-module functor.
We suppose that $M \in {}_A^H \Mod_B^{}$ is projective as a left $A$-module.
Then, by the equivalence \eqref{eq:equivariant-EW}, the $B$-$A$-bimodule $M^{\dagger}$ can be made into an object of ${}_B^H \Mod_A^{}$ in such a way that the map \eqref{eq:dagger-eval} is an isomorphism of left $\mathcal{C}$-module functors.
The detail is as follows:

\begin{lemma}
  \label{lem:equivariant-bimod-dagger}
  Let $M$ be as above. Then the $B$-$A$-bimodule $M^{\dagger}$ is an object of ${}_B^{H} \Mod_A^{}$ by the left $H$-coaction $f \mapsto f_{(-1)} \otimes f_{(0)}$ determined by the equation
  \begin{equation}
    \label{eq:dagger-comod-str}
    f_{(-1)} \otimes f_{(0)}(m)
    = S(m_{(-1)}) f(m_{(0)})_{(-1)} \otimes f(m_{(0)})_{(0)}
  \end{equation}
  in $H \otimes A$ for $f \in M^{\dagger}$ and $m \in M$.
  Thus the functor $M^{\dagger} \otimes_{A}(-)$ has a structure of an (oplax) left $\mathcal{C}$-module functor given by
  \begin{equation}
    \label{eq:dagger-C-module-str}
    \begin{aligned}
      \xi_{X,N}: M^{\dagger} \otimes_A (X \catactl N) & \to X \catactl (M^{\dagger} \otimes_A N), \\
      f \otimes_A (x \otimes n) & \mapsto f_{(-1)} x \otimes (f_{(0)} \otimes_A n)
    \end{aligned}
  \end{equation}
  for $n \in N \in {}_B\Mod$, $x \in X \in {}_H \Mod$ and $f \in M^{\dagger}$.
  The map~\eqref{eq:dagger-eval} gives rise to an isomorphism of left $\mathcal{C}$-module functors.
\end{lemma}

\begin{proof}
  We first show that there is a unique linear map $M^{\dagger} \to H \otimes M^{\dagger}$ determined by \eqref{eq:dagger-comod-str}. By using $\{ m_i, m^i \}$ as above, we define $\delta: M^{\dagger} \to H \otimes M^{\dagger}$ by
  \begin{equation*}
    \delta(f) = S(m_{i(-1)}) f(m_{i(0)})_{(-1)} \otimes m^i \cdot f(m_{i(0)})_{(0)}
    \quad (f \in M^{\dagger}).
  \end{equation*}
  For $m \in M$, we define the linear map $\phi_m$ by
  \begin{equation*}
    \phi_m: H \otimes M^{\dagger} \to H \otimes A,
    \quad \phi_m(h \otimes f) = h \otimes f(m)
    \quad (h \in H, f \in M^{\dagger}).
  \end{equation*}
  We note that two elements $T$ and $T'$ of $H \otimes M^{\dagger}$ are equal if and only if the equation $\phi_m(T) = \phi_m(T')$ holds for all $m \in M$.
  For simplicity, we set $a^i = m^i(m)$.
  Then, for $f \in M^{\dagger}$ and $m \in M$, we compute
  \begin{align*}
    \phi_m \delta(f)
    & = S(m_{i(-1)}) f(m_{i(0)})_{(-1)} \otimes m^i(m) f(m_{i(0)})_{(0)} \\
    & = S(m_{i(-1)}) S(a^i_{(-2)}) a^i_{(-1)} f(m_{i(0)})_{(-1)} \otimes a^i_{(0)} f(m_{i(0)})_{(0)} \\
    & = S(a^i_{(-1)} m_{i(-1)}) (a^i_{(0)} f(m_{i(0)}))_{(-1)} \otimes (a^i_{(0)} f(m_{i(0)}))_{(0)} \\
    & = S(a^i_{(-1)} m_{i(-1)}) f(a^i_{(0)} m_{i(0)})_{(-1)} \otimes f(a^i_{(0)} m_{i(0)})_{(0)} \\
    & = S((a^i m_{i})_{(-1)}) f((a^i m_{i})_{(0)})_{(-1)} \otimes f((a^i m_{i})_{(0)})_{(0)} \\
    & = S(m_{(-1)}) f(m_{(0)})_{(-1)} \otimes f(m_{(0)})_{(0)},
  \end{align*}
  where the third and the fifth equality follow from \eqref{eq:equivariant-bimod-def}, the fourth from the $A$-linearity of $f$, and the sixth from $a^i m_i = m$. Thus the map $f \mapsto \delta(f)$ is a unique linear map satisfying \eqref{eq:dagger-comod-str}.

  For $f \in M^{\dagger}$, $a \in A$, $b \in B$ and $m \in M$, we compute
  \begin{align*}
    \phi_m \delta(f a)
    & = S(m_{(-1)}) (f (m_{(0)}) a)_{(-1)} \otimes (f(m_{(0)}) a)_{(0)} \\
    & = S(m_{(-1)}) f (m_{(0)})_{(-1)} a_{(-1)} \otimes f(m_{(0)})_{(0)} a_{(0)} \\
    & = \phi_m(f_{(-1)} a_{(-1)} \otimes f_{(0)} a_{(0)}), \\
    \phi_m \delta(b f)
    & = S(m_{(-1)}) f(m_{(0)} b)_{(-1)} \otimes f(m b)_{(0)} \\
    & = b_{(-2)} S(b_{(-1)}) S(m_{(-1)}) f(m_{(0)} b_{(0)})_{(-1)} \otimes f(m b_{(0)})_{(0)} \\
    & = b_{(-1)} S(m_{(0)} b_{(0)}) f(m_{(0)} b_{(0)})_{(-1)} \otimes f(m_{(0)} b_{(0)})_{(0)} \\
    & = \phi_m(b_{(-1)} f_{(-1)} \otimes b_{(0)} f_{(0)}).
  \end{align*}
  Thus $\delta$ is a morphism of $B$-$A$-bimodules. Although we do not yet know whether $\delta$ makes $M^{\dagger}$ a left $H$-comodule, we define the linear map $\xi_{X,N}$ for $X \in \mathcal{C}$ and $N \in {}_A \Mod$ by \eqref{eq:dagger-C-module-str}.
  One can check that $\xi_{X,N}$ is well-defined by the $A$-linearity of $\delta$.
  The $B$-linearity of $\delta$ ensures that $\xi_{X,N}$ is indeed a morphism of left $B$-modules. Now we consider the diagram
  \begin{equation*}
    \begin{tikzcd}[column sep = 48pt]
      M^{\dagger} \otimes_A (X \catactl N)
      \arrow[r, "{e_{M,X \catactl N}}"]
      \arrow[d, "{\xi_{X,N}}"']
      & \Hom_A(M, X \catactl N)
      \arrow[d, "{\zeta_{X,N}^{-1}}"] \\
      X \catactl (M^{\dagger} \otimes_A N)
      \arrow[r, "{\id_X \catactl e_{M,N}}"]
      & X \catactl \Hom_A(M, N)
    \end{tikzcd}
  \end{equation*}
  in ${}_L \Mod$, where $\zeta_{X, N}: X \catactl \Hom_A(M, N) \to \Hom_A(M, X \catactl N)$ is the left $\mathcal{C}$-module structure of $\Hom_A(M, -)$ given by \eqref{eq:Hom-mod-str}. This diagram is commutative. Indeed, for $f \in M^{\dagger}$, $x \in X$, $n \in N$ and $m \in M$, we have
  \begin{align*}
    e_{M,X \catactl N}(f \otimes_A (x \otimes n))(m)
    & = f(m)_{(-1)} x \otimes_{\bfk} f(m)_{(0)} n \\
    & = m_{(-2)} S(m_{(-1)}) f(m_{(0)})_{(-1)} x \otimes_{\bfk} f(m_{(0)})_{(0)} n \\
    & = m_{(-1)} f_{(-1)} x \otimes f_{(0)}(m_{(0)}) n \\
    & = ((\zeta_{X,N} (\id_X \catactl e_{M,N}) \xi_{X,N}) (f \otimes_A (x \otimes n)))(m).
  \end{align*}
  The commutativity of the diagram shows that $M^{\dagger} \otimes_A (-)$ is a left $\mathcal{C}$-module functor by \eqref{eq:dagger-C-module-str} and the map~\eqref{eq:dagger-eval} is an isomorphism of left $\mathcal{C}$-module functors. Hence $M^{\dagger}$ is an object of ${}^H_A \Mod_B^{}$ with the left $H$-comodule structure $\delta$. The proof is done.
\end{proof}

\begin{remark}
  We suppose that $A$ is an exact left $H$-comodule algebra. Then every $\mathcal{C}$-module functor from ${}_A \Mod$ is exact \cite{MR2119143}. Thus $\Hom_A(M, -)$ is exact for every $M \in {}^H_A \Mod_B^{}$ and therefore every object of ${}_A^{H} \Mod_B^{}$ is projective as a left $A$-module. The assignment $M \mapsto M^{\dagger}$ extends to a contravariant functor such that the following diagram is commutative up to isomorphisms:
  \begin{equation*}
    \begin{tikzcd}[column sep = 48pt]
      {}_A^{H} \Mod_B^{}
      \arrow[d, "{(-)^{\dagger}}"']
      \arrow[r, "{\eqref{eq:equivariant-EW}}"]
      & \REX_{\mathcal{C}}({}_B \Mod, {}_A \Mod)
      \arrow[d, "{(-)^{\radj}}"] \\
      {}_B^{H} \Mod_A^{}
      \arrow[r, "{\eqref{eq:equivariant-EW}}"]
      & \REX_{\mathcal{C}}({}_A \Mod, {}_B \Mod).
    \end{tikzcd}
  \end{equation*}
  Namely, the contravariant functor $(-)^{\dagger} : {}^H_A \Mod_B^{} \to {}^H_B \Mod_A^{}$ is a module-theoretic counterpart of the operation taking right adjoints.
\end{remark}

\subsection{Module-theoretic description of the internal Hom functor}

Let $L$ be a finite-dimensional left $H$-comodule algebra.
It is shown in \cite{MR2331768} that an internal Hom functor of the $\mathcal{C}$-module category ${}_L \Mod$ is given by the Yan-Zhu stabilizer.
Here we give another description of an internal Hom functor of ${}_L \Mod$ and its $\mathcal{C}$-bimodule structure.

If $M \in {}_L \Mod$, then $H \catactl M \in {}_L\Mod$ becomes an object of ${}_L^H \Mod_{H}^{}$ together with the right $H$-action and the right $H$-coaction given by
\begin{equation*}
  (h \otimes m) \cdot h' = h h' \otimes m
  \quad \text{and} \quad h \otimes m \mapsto h_{(1)} \otimes h_{(2)} \otimes m,
\end{equation*}
respectively, for $h, h' \in H$ and $m \in M$. It is easy to see that the map
\begin{equation*}
  X \catactl M \to (H \catactl M) \otimes_H X,
  \quad x \otimes m \mapsto (1 \otimes m) \otimes_H x
  \quad (m \in M, x \in X)
\end{equation*}
is a natural isomorphism of left $L$-modules for $M \in {}_L \Mod$ and  $X \in {}_H \Mod$. Thus, by the tensor-Hom adjunction, we have natural isomorphisms
\begin{equation*}
  \Hom_{L}(X \catactl M, N)
  \cong \Hom_{L}((H \catactl M) \otimes_H X, N)
  \cong \Hom_{H}(X, \Hom_{L}(H \catactl M, N))
\end{equation*}
for $X \in \mathcal{C}$ and $M, N \in {}_L \Mod$. This means that the functor
\begin{equation*}
  \iHom(M, N) = \Hom_{L}(H \catactl M, N)
  \quad (M, N \in {}_L\Mod)
\end{equation*}
is an internal Hom functor of ${}_L \Mod$. By the construction, the unit and the counit of the adjunction $(-) \catactl M \dashv \iHom(M, -)$ are given by
\begin{gather}
  \label{eq:Hopf-iHom-unit}
  \icoev_{X,M}: X \to \iHom(M, X \catactl M),
  \quad x \mapsto [ h \otimes m \mapsto h x \otimes m ], \\
  \label{eq:Hopf-iHom-counit}
  \ieval_{M,N}: \iHom(M, N) \catactl M \to N,
  \quad f \otimes m \mapsto f(1 \otimes m)
\end{gather}
for $x \in X \in \mathcal{C}$, $M, N \in {}_L \Mod$, $m \in M$, $h \in H$ and $f \in \iHom(M, N)$. Now we compute some morphisms related to the internal Hom functor:

\begin{lemma}
  For $M_1, M_2, M_3 \in {}_L \Mod$, the composition
  \begin{equation*}
    \icomp_{M_1, M_2, M_3}: \iHom(M_2, M_3) \otimes \iHom(M_1, M_2) \to \iHom(M_1, M_3)
  \end{equation*}
  for the internal Hom functor is given by
  \begin{equation*}
    f \otimes g \mapsto [ h \otimes m \mapsto f(h_{(1)} \otimes g(h_{(2)} \otimes m)) ]
  \end{equation*}
  for $f \in \iHom(M_2, M_3)$ and $g \in \iHom(M_1, M_2)$.
\end{lemma}
\begin{proof}
  The map $\icomp_{M_1, M_2, M_3}$ is defined to be the composition
  \begin{equation*}
    \ieval_{M_2, M_3} \circ (\id_{\iHom(M_2, M_3)} \catactl \ieval_{M_1, M_2})
    \circ \icoev_{\iHom(M_2, M_3) \otimes \iHom(M_1, M_2), M_1}.
  \end{equation*}
  By using \eqref{eq:Hopf-iHom-unit} and \eqref{eq:Hopf-iHom-counit}, one proves this lemma straightforwardly.
\end{proof}

\begin{lemma}
  The $\mathcal{C}$-bimodule structure of $\iHom$ is given by
  \begin{align}
    \label{eq:Hopf-iHom-mod-1}
    X \otimes \iHom(M, N)
    & \to \iHom(M, X \catactl N), \\
    \notag x \otimes \xi
    & \mapsto [ h \otimes m \mapsto h_{(1)} x \otimes \xi(h_{(2)} \otimes m) ], \\[3pt]
    \label{eq:Hopf-iHom-mod-2}
    \iHom(M, N) \otimes Y
    & \to \iHom({}^* Y \catactl M, N), \\
    \notag \xi \otimes y
    & \mapsto [ h \otimes {}^*\!y \otimes m
      \mapsto  \xi(h_{(1)} \otimes m) \, \langle {}^* \! y, h_{(2)} y \rangle ]
  \end{align}
  for $x \in X \in \mathcal{C}$, $y \in Y \in \mathcal{C}$, $M, N \in {}_L \Mod$ and $\xi \in \iHom(M, N)$.
\end{lemma}
\begin{proof}
  We recall from Subsection~\ref{subsec:internal-hom} that the left $\mathcal{C}$-module structure of $\iHom$ is given by \eqref{eq:iHom-mod-left-1}--\eqref{eq:iHom-mod-left-2}.
  For $x \in X$ and $\xi \in \iHom(M, N)$, we compute:
  \begin{align*}
    x \otimes \xi
    \xmapsto{\makebox[3em]{\scriptsize \eqref{eq:iHom-mod-left-1}}}
    & \, [ h \otimes m \mapsto h_{(1)} x \otimes h_{(2)} \xi \otimes m ] \\
    \xmapsto{\makebox[3em]{\scriptsize \eqref{eq:iHom-mod-left-2}}}
    & \, [ h \otimes m \mapsto h_{(1)} x \otimes \xi(h_{(2)} \otimes m) ],
  \end{align*}
  where $h \in H$ and $m \in M$. This agrees with \eqref{eq:Hopf-iHom-mod-1}.
  The right $\mathcal{C}$-module structure of $\iHom$ is given by \eqref{eq:iHom-mod-right-1}--\eqref{eq:iHom-mod-right-3}. For $\xi \in \iHom(M, N)$ and $y \in Y$, we compute:
  \begin{align*}
    \xi \otimes y
    \xmapsto{\makebox[3em]{\scriptsize \eqref{eq:iHom-mod-right-1}}}
    & \, [ h \otimes {}^*\!y \otimes m \mapsto h_{(1)} \xi \otimes h_{(2)} y \otimes {}^*\!y \otimes m ] \\
    \xmapsto{\makebox[3em]{\scriptsize \eqref{eq:iHom-mod-right-2}}}
    & \, [ h \otimes {}^*\!y \otimes m \mapsto h_{(1)} \xi \otimes \langle {}^*\!y, h_{(2)} y \rangle m ] \\
    \xmapsto{\makebox[3em]{\scriptsize \eqref{eq:iHom-mod-right-3}}}
    & \, [ h \otimes {}^*\!y \otimes m \mapsto \xi(h_{(1)} \otimes m) \, \langle {}^*\!y, h_{(2)} y \rangle ],
  \end{align*}
  where $h \in H$, ${}^*\!y \in {}^* Y$ and $m \in M$.
  This agrees with \eqref{eq:Hopf-iHom-mod-2}. The proof is done.
\end{proof}

\begin{remark}
  \label{rem:internal-Yoneda}
  By Lemma~\ref{lem:Hom-mod-str}, an object $P \in {}_L^{H} \Mod_H^{}$ gives rise to a left exact left $\mathcal{C}$-module functor $\Hom_L(P, -): {}_L \Mod \to \mathcal{C}$. The left $\mathcal{C}$-module structure of $\iHom$ given in the above lemma is just the case where $P = H \catactl M \in {}_L^{H} \Mod_H^{}$.
\end{remark}

\subsection{Module-theoretic description of the internal Yoneda functor}
\label{subsec:internal-Yoneda}

Let $L$ be an exact left $H$-comodule algebra.
The aim of this subsection is to introduce a module-theoretic counterpart of the internal Yoneda functor
\begin{equation*}
  \Yone: {}_L \Mod \to \REX_{\mathcal{C}}({}_L \Mod, \mathcal{C})^{\op},
  \quad M \mapsto \iHom(M, -).
\end{equation*}
Since $L$ is assumed to be exact and $H \in \mathcal{C}$ is projective, $H \catactl M \in {}_L \Mod$ is projective for all $M \in {}_L \Mod$. By Lemma \ref{lem:equivariant-bimod-dagger}, the following functor is defined:
\begin{equation*}
  \Yone_1: {}_L \Mod \to ({}_H^H \Mod_L^{})^{\op}
  \quad M \mapsto (H \catactl M)^{\dagger} = \iHom(M, L).
\end{equation*}
Furthermore, the natural transformation
\begin{equation}
  \label{eq:Yoneda-module-tensor}
  \Yone_1(M) \otimes_L (-)
  \xrightarrow[\quad \eqref{eq:dagger-eval} \quad]{e_{H \catactl M, -}}
  \Hom_{L}(H \catactl M, -)
  = \Yone(M)
\end{equation}
is an isomorphism of left $\mathcal{C}$-module endofunctors on ${}_L \Mod$.

\begin{lemma}
  \label{lem:Yoneda-module}
  The functor $\Yone_1$ is a $\DD$-twisted left $\mathcal{C}$-module functor with the structure morphism given by
  \begin{equation}
    \label{eq:Yoneda-module-1}
    \begin{aligned}
      s_{M,X}: \Yone_1(M) \catactr X^*
      & \to \Yone_1(X \catactl M), \\
      \xi \otimes x^*
      & \mapsto [ h \otimes x \otimes m \mapsto \xi(h_{(1)} \otimes m) \langle x^*, S(h_{(2)}) x \rangle ]
    \end{aligned}
  \end{equation}
  for $X \in \mathcal{C}$, $M \in {}_L \Mod$, $\xi \in \Yone_1(M)$ and $x^* \in X^*$.
  The composition
  \begin{equation}
    \label{eq:Yoneda-module-2}
    {}_L \Mod
    \xrightarrow{\quad \Yone_1 \quad}
    ({}_H^{H} \Mod_L^{})^{\op}
    \xrightarrow[\approx]{\quad \text{\rm \eqref{eq:equivariant-EW-mod-str} and Remark~\ref{rem:equivariant-EW-adjoint}} \quad}
    \REX_{\mathcal{C}}({}_L \Mod, \mathcal{C})^{\op}
  \end{equation}
  is isomorphic to $\Yone$ via \eqref{eq:Yoneda-module-tensor} as a $\DD$-twisted left $\mathcal{C}$-module functor.
\end{lemma}
\begin{proof}
  It is straightforward to verify that $\Yone_1$ is a $\DD$-twisted left $\mathcal{C}$-module functor by the structure map given by \eqref{eq:Yoneda-module-1}.
  We denote by
  \begin{align*}
    \mathfrak{b}_{M,N,X} & : \Yone(M)(N) \otimes X^* \to \Yone(X \catactl M)(N)
    \quad \text{and} \\
    \mathfrak{b}'_{M,N,X} & : (\Yone_1(M) \otimes_L N) \otimes X^* \to \Yone_1(X \catactl M)(N)
  \end{align*}
  the $\DD$-twisted left $\mathcal{C}$-module structure of $\Yone$ and that of \eqref{eq:Yoneda-module-2}, respectively.
  The former one is given by \eqref{eq:Hopf-iHom-mod-2} with $Y = X^*$.
  By straightforward computation, we see that the latter one is given by
  \begin{align*}
    \mathfrak{b}'_{M,N,X}((\xi \otimes_L n) \otimes x^*)
    & \mathop{=}^{\eqref{eq:equivariant-EW-bicomod-alg-1}}
      \mbox{} s_{M,X}(\xi \otimes x^*) \otimes_L n \\
    & \mathop{=}^{\eqref{eq:Yoneda-module-1}}
      \mbox{} [ h \otimes x \otimes m \mapsto \xi(h_{(1)} \otimes m) \langle x^*, S(h_{(2)}) x \rangle ] \otimes_L n
  \end{align*}
  for $\xi \in \Yone_1(M)$, $n \in N$ and $x^* \in X^*$. We compute:
  \begin{align*}
    & \mathfrak{b}_{M,N,X}(e_{H \catactl M,N} \otimes \id_{X^*})((\xi \otimes_L n) \otimes x^*) \\
    & \quad = \mathfrak{b}_{M,N,X}([h \otimes m \mapsto \xi(h \otimes m) n] \otimes x^*) \\
    & \quad = [ h \otimes m \otimes x \mapsto \xi(h_{(1)} \otimes m) n
      \langle h_{(2)} x^*, x \rangle ]
      \quad (\text{by \eqref{eq:Hopf-iHom-mod-2} with $Y = X^*$}) \\
    & \quad = [ h \otimes m \otimes x \mapsto \xi(h_{(1)} \otimes m) n
      \langle x^*, S(h_{(2)}) x \rangle ] \\
    & \quad = e_{H \catactl X \catactl M, N} \mathfrak{b}'_{M,N,X}((\xi \otimes_L n) \otimes x^*).
  \end{align*}
  This shows that \eqref{eq:Yoneda-module-tensor} is an isomorphism of $\DD$-twisted left $\mathcal{C}$-module functors. The proof is done.
\end{proof}

We fix a Frobenius system $(\lambda_L, \{ a^i \}, \{ b_i \})$ for $L$, which exists by Lemma \ref{lem:exact-comod-alg-Frobenius}.
For later use, we introduce another module-theoretic internal Yoneda functor $\Yone_2$, which depends on the choice of the Frobenius system for $L$.
For $M \in {}_L\Mod$, we set $\Yone_2(M) := (H \otimes M)^*$ as a vector space.
By Lemma~\ref{lem:Frob-alg-Hom}, the map
\begin{equation}
  \label{eq:Y2-M-psi-1}
  \psi_M: \Yone_1(M) \to \Yone_2(M),
  \quad f \mapsto \lambda_L \circ f
  \quad (f \in \Yone_1(M))
\end{equation}
is an isomorphism of vector spaces with the inverse given by
\begin{equation}
  \label{eq:Y2-M-psi-2}
  \psi^{-1}_M(\xi)(h \otimes m) = \langle \xi, a^i_{(-1)} h \otimes a^i_{(0)} m \rangle b_i
\end{equation}
for $h \in H$ and $m \in M$. We make $\Yone_2(M)$ an object of ${}^H_H \Mod_L$ by transporting the structure maps of $\Yone_1(M) \in {}^H_H \Mod_L$ via $\psi_M$.
It is obvious that the assignment $M \mapsto \Yone_2(M)$ gives rise to a functor
\begin{equation*}
  \Yone_2: {}_L \Mod \to ({}^H_H \Mod_L)^{\op},
\end{equation*}
which is isomorphic to $\Yone_1$ via \eqref{eq:Y2-M-psi-1}. We note:

\begin{lemma}
  \label{lem:Yoneda-module-2}
  Let $M \in {}_L\Mod$.
  The left action of $H$, the right action of $L$ and the left coaction $\xi \mapsto \xi_{(-1)} \otimes \xi_{(0)}$ of $H$ on $\Yone_2(M) \in {}^H_H \Mod_L$ are determined by
  \begin{gather}
    \label{eq:Y2-M-action}
    \langle h \cdot \xi \cdot a, k \otimes m \rangle
    = \langle \xi, \nu_L(a)_{(-1)} k h \otimes \nu_L(a)_{(0)} m \rangle, \\
    \label{eq:claim-Y2-M-coact-1}
    \xi_{(-1)} \langle \xi_{(0)}, k \otimes m \rangle
    = S(k_{(1)}) b_{i(-1)} \langle \lambda_L, b_{i(0)} \rangle \langle \xi, a^i_{(-1)} k_{(2)} \otimes a^i_{(0)} m \rangle
  \end{gather}
  for $\xi \in \Yone_2(M)$, $h, k \in H$, $a \in L$ and $m \in M$, where $\nu_L$ is the Nakayama automorphism of $L$ with respect to $\lambda_L$.
\end{lemma}
\begin{proof}
  It suffices to check that the equations
  \begin{gather*}
    \langle h \cdot \psi_M(f) \cdot a, k \otimes m \rangle
    = \langle \psi_M(h \cdot f \cdot a), k \otimes m \rangle, \\
    f_{(-1)} \langle \psi_M(f_{(0)}), k \otimes m \rangle
    = \psi_M(f)_{(-1)} \langle \psi_M(f)_{(0)}, k \otimes m \rangle
  \end{gather*}
  hold for all elements $h, k \in H$, $a \in L$, $m \in M$ and $f \in \Yone_1(M)$.
  The former one is verified directly as follows:
  \begin{gather*}
    \langle h \cdot \psi_M(f) \cdot a, k \otimes m \rangle
    \stackrel{\eqref{eq:Y2-M-action}}{=}
    \langle \lambda_L, f(\nu_{L}(a)_{(-1)} k h \otimes \nu_{L}(a)_{(0)} m) \rangle \\
    = \langle \lambda_L, \nu_L(a) \cdot f(k h \otimes m) \rangle
    \stackrel{\eqref{eq:def-Nakayama}}{=}
    \langle \lambda_L, f(k h \otimes m) a \rangle
    = \langle \psi_M(h \cdot f \cdot a), k \otimes m \rangle,
  \end{gather*}
  where the second equality follows from the $L$-linearity of $f$. To prove the latter one, we set $\xi = \psi_M(f)$ and $w = k \otimes m$ for simplicity. Then we have
  \begin{align*}
    f_{(-1)} \langle \psi_M(f_{(0)}), w \rangle = \,
    & f_{(-1)} \langle \lambda_L, f_{(0)}(w) \rangle \\
    {}^{\eqref{eq:dagger-comod-str}} = \,
    & S(w_{(-1)}) f(w_{(0)})_{(-1)} \langle \lambda_L, f(w_{(0)})_{(0)} \rangle \\
    = \,
    & S(k_{(1)}) f(k_{(2)} \otimes m)_{(-1)} \langle \lambda_L, f(k_{(2)} \otimes m)_{(0)} \rangle \\
    {}^{\eqref{eq:Y2-M-psi-2}} = \,
    & S(k_{(1)}) \langle \xi, a^i_{(-1)} k_{(2)} \otimes a^i_{(0)} m \rangle
      b_{i(-1)} \langle \lambda_L, b_{i(0)} \rangle \\
    {}^{\eqref{eq:claim-Y2-M-coact-1}} =\,
    & \xi_{(-1)} \langle \xi_{(0)}, k \otimes m \rangle. \qedhere
  \end{align*}
\end{proof}

Thus, $\Yone_2 : {}_L\Mod \to ({}^H_H\Mod_L)^{\op}$ is a $\DD$-twisted left $\mathcal{C}$-module functor by
\begin{equation*}
  t_{M,X} := \psi_{X \catactl M} \circ s_{M,X} \circ (\psi_M^{-1} \catactl \id_{X^*})
  \quad (M \in {}_L\Mod, X \in \mathcal{C}),
\end{equation*}
where $s$ is the structure morphism of $\Yone_1$ given by~\eqref{eq:Yoneda-module-1}.
It is obvious from the construction that the composition
\begin{equation*}
  {}_L \Mod \xrightarrow{\quad \Yone_2 \quad} ({}^H_H \Mod_L^{})^{\op}
  \xrightarrow{\quad \eqref{eq:equivariant-EW} \quad} \REX_{\mathcal{C}}({}_L\Mod, \mathcal{C})^{\op}
\end{equation*}
is isomorphic to $\Yone$ as a $\DD$-twisted left $\mathcal{C}$-module functor.
For later use, we provide the following explicit expression of $t_{M,X}$.

\begin{lemma}
  For $M \in {}_L \Mod$ and $X \in \mathcal{C}$, we have
  \begin{equation}
    \label{eq:Y2-M-C-mod-str}
    \langle t_{M,X}(\xi \otimes x^*), h \otimes x \otimes m \rangle
    = \langle \xi, h_{(1)} \otimes m \rangle \langle x^*, S(h_{(2)}) x \rangle
  \end{equation}
  for $f \in \Yone_2(M)$, $x^* \in X^*$, $h \in H$, $x \in X$ and $m \in M$.
\end{lemma}
\begin{proof}
  For $\xi$, $x^*$, $h$, $x$ and $m$ as above, we have
  \begin{align*}
    (s_{M,X}(\psi_M^{-1}(\xi) \otimes x^*))(h \otimes x \otimes m)
    & = \psi_M^{-1}(\xi)(h_{(1)} \otimes m) \langle x^*, S(h_{(2)}) x \rangle \\
    & = \langle \xi, a^i_{(-1)} h_{(1)} \otimes a^i_{(0)} m \rangle
      \langle x^*, S(h_{(2)}) x \rangle b_i
  \end{align*}
  by \eqref{eq:Yoneda-module-1} and \eqref{eq:Y2-M-psi-2}.
  Thus equation~\eqref{eq:Y2-M-C-mod-str} is verified as follows:
  \begin{align*}
    \langle t_{X,M}(\xi \otimes x^*), h \otimes x \otimes m \rangle
    & = \langle \xi, a^i_{(-1)} h_{(1)} \otimes a^i_{(0)} m \rangle
      \langle x^*, S(h_{(2)}) x \rangle \langle \lambda_L, b_i \rangle \\
    {}^{\eqref{eq:Frob-str-L-1}}
    & = \langle \xi, h_{(1)} \otimes m \rangle \langle x^*, S(h_{(2)}) x \rangle. \qedhere
  \end{align*}
\end{proof}

\subsection{Relative Serre functor for comodule algebras}
\label{subsec:comod-alg-rel-Serre}

Let $L$ be an exact left $H$-comodule algebra with coaction $\delta_L$.
Then $L$ is a Frobenius algebra by Lemma~\ref{lem:exact-comod-alg-Frobenius}.
We fix a Frobenius system $(\lambda_L, \{ a^i \}, \{ b_i \})$ for $L$ and define the algebra automorphism $\nu_L'$ of $L$ by
\begin{equation}
  \label{eq:twisted-Nakayama}
  \nu'_L(a) :=
  (\id_L \otimes \alpha_H) \delta_L \nu_L(a)
  = \langle \alpha_H, \nu_L(a)_{(-1)} \rangle \nu_L(a)_{(0)} \quad (a \in L),
\end{equation}
where $\nu_L$ is the Nakayama automorphism \eqref{eq:def-Nakayama} and $\alpha_H: H \to \bfk$ is the right modular function on $H$ (see \eqref{eq:right-modular} below for the precise definition). For $M \in {}_L \Mod$, we define $\Ser(M) \in {}_L \Mod$ to be the vector space $M$ equipped with the action $\bullet$ given by
\begin{equation}
  \label{eq:Ser-action}
  a \bullet m = \nu'_L(a) m
  \quad (a \in L, m \in M).
\end{equation}
By \cite[Theorems 4.4 and 4.26]{MR4042867}, the functor
\begin{equation}
  \label{eq:Ser-FSS}
  \Ser': {}_L\Mod \to {}_L \Mod,
  \quad \Ser'(M) = \Nak_L(D^{*} \catactl M) \quad (M \in {}_L \Mod)
\end{equation}
is a relative Serre functor for ${}_L\Mod$, where $D$ is the so-called distinguished invertible object of $\mathcal{C}$.
Since the Nakayama functor $\Nak_L$ is isomorphic to the functor twisting the action of $L$ by $\nu_L$, and since $D^{*}$ is isomorphic to the left $H$-module given by the algebra map $\alpha_H : H \to \bfk$, the functor $\Ser$ is isomorphic to \eqref{eq:Ser-FSS}. Thus $\Ser$ is also a relative Serre functor of ${}_L \Mod$. In this sense, a relative Serre functor of ${}_L \Mod$ has been completely determined on the level of functors. However, this description does not give any information about the natural isomorphisms
\begin{equation}
  \label{eq:rel-Ser-isomorphisms}
  \iHom(N, \Ser(M)) \cong \iHom(M, N)^*
  \quad \text{and} \quad
  \Ser(X \catactl M) \cong X^{**} \catactl \Ser(M)
\end{equation}
for $M, N \in {}_L \Mod$ and $X \in {}_H \Mod$, which are important when, for example, we discuss whether ${}_L \Mod$ admits a pivotal structure.

We aim to have module-theoretic descriptions of isomorphisms \eqref{eq:rel-Ser-isomorphisms}. For this purpose, we use the integral theory for Hopf algebras \cite{MR1243637,MR1265853}.
Let $\Lambda_H \in H$ be a non-zero right integral in $H$.
The {\em right modular function} on $H$ is defined to be the unique algebra map $\alpha_H: H \to \bfk$ satisfying
\begin{equation}
  \label{eq:right-modular}
  h \cdot \Lambda_H = \langle \alpha_H, h \rangle \Lambda_H \quad (h \in H).
\end{equation}
There is a unique right cointegral $\lambda_H$ on $H$ ($=$ a right integral in $H^*$) such that $\langle \lambda_H, \Lambda_H \rangle = 1$.
We define $g_H^{} \in H$ to be the right modular function $\alpha_{H^*} \in H^{**}$ regarded as an element of $H$ through the canonical isomorphism $H^{**} \cong H$ of vector spaces. Equivalently, $g_H$ is the unique element of $H$ such that
\begin{equation}
  \label{eq:distinguished-grouplike}
  h_{(1)} \langle \lambda_H, h_{(2)} \rangle = \langle \lambda_H, h \rangle g_H^{}
  \quad (h \in H).
\end{equation}
The map $\alpha_H$ is an algebra map, and the element $g_H$ is a grouplike element, that is, $\Delta(g_H) = g_H \otimes g_H$ and $\varepsilon(g_H) = 1$.
By using the Frobenius system, we define
\begin{equation}
  \label{eq:Serre-mod-str-element}
  \mathfrak{T} = \langle \lambda_L, b_{i(0)} \rangle S^3(b_{i(-1)}) g_H^{}
  \otimes \langle \alpha_H, a^i_{(-1)} \rangle a^i_{(0)} \in H \otimes L
\end{equation}
and write it symbolically as $\mathfrak{T} = \mathfrak{T}_H \otimes \mathfrak{T}_L$. Now we state the following main theorem in this section:

\begin{theorem}
  \label{thm:main-theorem}
  For $M \in {}_L \Mod$, we define $\Ser(M) \in {}_L \Mod$ to be the vector space $M$ equipped with the left $L$-action given by~\eqref{eq:Ser-action}. Then the endofunctor $\Ser$ on ${}_L \Mod$ is a relative Serre functor of ${}_L \Mod$ together with the following isomorphisms:
  \begin{enumerate}
  \item [(1)] The twisted $\mathcal{C}$-module structure $\Ser(X \catactl M) \to X^{**} \catactl \Ser(M)$ is given by
    \begin{equation}
      \label{eq:Ser-mod-structure}
      x \otimes m
      \mapsto \Phi_X(\mathfrak{T}_H x)
      \otimes \mathfrak{T}_L m
    \end{equation}
    for $x \in X \in \mathcal{C}$ and $m \in M \in {}_L \Mod$, where $\Phi_X: X \to X^{**}$ is the canonical isomorphism of finite-dimensional vector spaces.
  \item [(2)]
    For $M \in {}_L\Mod$, we define $\itrace_M : \iHom(M, \Ser(M)) \to \bfk$ by
    \begin{equation}
      \label{eq:iHom-trace}
      \itrace_M(\xi) = \langle \lambda_L, t^i(\Lambda_H \otimes \xi(t_i)) \rangle
      \quad (\xi \in \iHom(M, N)),
    \end{equation}
    where $\{ t_i, t^i \}$ is a pair of dual bases of $H \catactl M$ (see \S\ref{subsubsec:ra-by-tensor}), which is projective since $L$ is exact.
    The isomorphism
    \begin{equation*}
      \iHom(N, \Ser(M)) \cong \iHom(M, N)^*
      \quad (M, N \in {}_L \Mod)
    \end{equation*}
    of twisted $\mathcal{C}$-bimodule functors is induced by the pairing
    \begin{equation*}
      \itrace_M \circ \icomp_{M,N,\Ser(M)}: \iHom(N, \Ser(M)) \otimes \iHom(M, N) \to \bfk.
    \end{equation*}
  \end{enumerate}
\end{theorem}
  
The formula \eqref{eq:Ser-mod-structure} may not be convenient for practical use.
In Subsection~\ref{subsec:gl-coint}, we will give a simple expression of the twisted $\mathcal{C}$-module structure under the assumption that the Frobenius form $\lambda_L$ satisfies a certain equation like a cointegral on a Hopf algebra.

To prove Theorem~\ref{thm:main-theorem}, we consider the diagram
\begin{equation}
  \label{eq:main-theorem-proof-diagram-1}
  \begin{tikzcd}[column sep = 42pt, row sep = 32pt]
    {}_L \Mod \arrow[r, "{\Yone}"]
    \ar[d, leftarrow, "{\id}"]
    \ar[rd, Rightarrow, shorten = 2ex,
    "\text{\S\ref{subsec:internal-Yoneda}}"]
    & \REX_{\mathcal{C}}({}_L \Mod, \mathcal{C})^{\op}
    \ar[rd, Rightarrow, shorten = 2ex,
    "\text{Remark~\ref{rem:equivariant-EW-adjoint}}"]
    \arrow[r, "{(-)^{\radj}}"]
    & \LEX_{\mathcal{C}}(\mathcal{C}, {}_L \Mod)
    \ar[d, leftarrow, "{\id}"]
    \arrow[r, "{\EvalAtOne}", "{\eqref{eq:equivalence-eval-at-1}}"']
    & {}_L \Mod \ar[d, leftarrow, "\id"] \\
    {}_L \Mod \arrow[r, "{\Yone_2}"']
    & ({}^H_H \Mod_L^{})^{\op}
    \arrow[u, "{\eqref{eq:equivariant-EW}}"']
    \arrow[r, "{\eqref{eq:equivariant-EW-lex}}"']
    & \LEX_{\mathcal{C}}(\mathcal{C}, {}_L \Mod)
    \arrow[r, "{\EvalAtOne}", "{\eqref{eq:equivalence-eval-at-1}}"']
    & {}_L \Mod
  \end{tikzcd}
\end{equation}
of ($\DD$-twisted) $\mathcal{C}$-module functors, where identity functors are placed because of appearance of the diagram and double arrows ($\Rightarrow$) represent isomorphisms of module functors.
By Remark~\ref{rem:equivariant-EW-adjoint} and the discussion in Subsection~\ref{subsec:internal-Yoneda}, the above diagram commutes up to isomorphisms.

Now we note that there are natural isomorphisms
\begin{subequations}
  \begin{align*}
    \Hom_H(\Yone(M)(N), X)
    & = \Hom_H(\Hom_L(H \catactl M, N), X) \\
    & \cong \Hom_H(\Yone_1(M) \otimes_L N, X)
      \quad (\text{by Lemma~\ref{lem:Yoneda-module}}) \\
    & \cong \Hom_H(\Yone_2(M) \otimes_L N, X)
      \quad (\text{by Lemma~\ref{lem:Yoneda-module-2}}) \\
    & \cong \Hom_L(N, \Hom_H(\Yone_2(M), X))
  \end{align*}
\end{subequations}
for $N \in {}_L \Mod$ and $X \in \mathcal{C}$.
Thus we may, and do, elect the functor
\begin{equation}
  \label{eq:YM-right-adj}
  \Yone(M)^{\radj} := \Hom_{H}(\Yone_2(M), -)
\end{equation}
to a right adjoint of $\Yone(M)$. We denote by
\begin{equation}
  \label{eq:Serre-1-functor}
  \Ser_2: {}_L \Mod \to {}_{\DD}({}_L \Mod),
  \quad \Ser_2(M) = \Hom_{H}(\Yone_2(M), \bfk)
  \quad (M \in {}_L \Mod)
\end{equation}
the twisted module functor obtained by the composition along the bottom row of the diagram \eqref{eq:main-theorem-proof-diagram-1}. Then we have:

\begin{lemma}
  \label{lem:main-theorem-1} 
  The functor $\Ser_2$ is a relative Serre functor for ${}_L\Mod$ whose structure morphism as a twisted module functor is given by the composition
  \begin{gather*}
    \Ser_2(X \catactl M)
    = \Hom_{H}(\Yone_2(X \catactl M), \bfk)
    \xrightarrow{\ \eqref{eq:Y2-M-C-mod-str} \ }
    \Hom_{H}(\Yone_2(M) \catactr X^*, \bfk) \\
    \xrightarrow{\ \eqref{eq:equivariant-EW-bicomod-alg-2} \ }
    \Hom_{H}(\Yone_2(M), X^{**})
    \xrightarrow{\ \eqref{eq:Hom-mod-str} \ }
    X^{**} \catactl \Hom_{H}(\Yone_2(M), \bfk)
    = X^{**} \catactl \Ser_2(M)
  \end{gather*}
  for $X \in \mathcal{C}$ and $M \in {}_L\Mod$. The isomorphism
  \begin{equation*}
    \iHom(M, N)^* \cong \iHom(N, \Ser_2(M))
    \quad (M, N \in {}_L \Mod)
  \end{equation*}
  of twisted $\mathcal{C}$-bimodule functors is induced by the pairing
  \begin{gather*}
    \iHom(N, \Ser_2(M)) \otimes \iHom(M, N)
    \xrightarrow{\quad \icomp \quad}
    \iHom(M, \Ser_2(M))
    \xrightarrow{\quad \varepsilon_M(\bfk) \quad} \bfk,
  \end{gather*}
  where $\varepsilon_M$ is the counit of the adjunction $\Yone(M) \dashv \Yone(M)^{\radj}$.
\end{lemma}
\begin{proof}
  The description of the structure of $\Ser_2$ as a twisted module functor follows from the definition of $\Ser_2$.
  The composition along the top row of the diagram \eqref{eq:main-theorem-proof-diagram-1} is the standard relative Serre functor discussed in Subsection~\ref{subsec:std-realization}.
  Under the choice \eqref{eq:YM-right-adj} of a right adjoint of $\Yone(M)$, the composition of two double arrows in \eqref{eq:main-theorem-proof-diagram-1} is actually equal to the identity natural transformation.   
  Thus $\Ser_2$ is equal to the standard relative Serre functor. The latter half part of this lemma now follows from Lemma~\ref{lem:rel-Ser-pairing}.
\end{proof}

Hence we have obtained a relative Serre functor $\Ser_2$ of ${}_L \Mod$ with structure morphisms written explicitly. To establish Theorem~\ref{thm:main-theorem}, we will give a natural isomorphism $M \cong \Ser_2(M)$ of vector spaces with the use of integrals for $H$. The theorem will be proved by transporting the structure morphisms from $\Ser_2(M)$ to $M$ through this isomorphism.

\subsection{Proof of Theorem~\ref{thm:main-theorem} (1)}

We use the same notation as in the previous subsection.
We recall that $\Lambda := \Lambda_H$ is a non-zero right integral in $H$ and $\lambda_H$ is the right cointegral on $H$ such that $\langle \lambda_H, \Lambda \rangle = 1$. By definition, $\Ser_2(M)$ is a subspace of $(H \otimes M)^{**}$.

\begin{lemma}
  \label{lem:Y2-M-varphi}
  For $M \in {}_L \Mod$, we define the linear map
  \begin{equation}
    \label{eq:Y2-M-varphi-1}
    \varphi_M: \Ser_2(M) \to \Ser(M),
    \quad \varphi_M(\xi) = (\lambda_H \otimes \id_M) \Phi_{H \otimes M}^{-1}(\xi).
  \end{equation}
  This map is an isomorphism of left $L$-modules with the inverse given by
  \begin{equation}
    \label{eq:Y2-M-varphi-2}
    \varphi_M^{-1}(m) = \Phi_{H \otimes M}(\Lambda \otimes m)
    \quad (m \in \Ser(M)).
  \end{equation}
\end{lemma}
\begin{proof}
  We identify the vector space $\Yone_2(M)$ with $H^* \otimes M^*$ via the isomorphism
  \begin{equation*}
    H^* \otimes M^* \to \Yone_2(M) = (H \otimes M)^*,
    \quad h^* \otimes m^* \mapsto [h \otimes m \mapsto \langle h^*, h \rangle \langle m^*, m \rangle].
  \end{equation*}
  Then the action of $H$ on $\Yone_2(M)$, given by \eqref{eq:Y2-M-action}, is expressed as follows:
  \begin{equation*}
    h \cdot (h^* \otimes m^*) = (h \rightharpoonup h^*) \otimes m^*
    \quad (h \in H, h^* \in H^*, m^* \in M^*).
  \end{equation*}
  As is well-known, $\lambda_H$ is a Frobenius form on $H$. Hence the map $\vartheta: H \to H^*$ given by $h \mapsto h \rightharpoonup \lambda_H$ is an isomorphism of left $H$-modules. Now we have
  \begin{gather*}
    \Ser_2(M) = \Hom_H(\Yone_2(M), \bfk)
    \cong \Hom_H(H^* \otimes M^*, \bfk) \\
    \cong \Hom_H(H \otimes M^*, \bfk)
    \cong \Hom_{\bfk}(M^*, \bfk)
    \cong M,
  \end{gather*}
  where the second isomorphism is induced by $\vartheta$. The map $\varphi_M$ is obtained as the composition, and thus it is an isomorphism. It is easy to check that the inverse of $\varphi_M$ is given as stated.

  To complete the proof, we need to check that $\varphi_M$ is $L$-linear. Instead of doing so, we show that $\varphi_M^{-1}$ is $L$-linear. We recall that the action of $L$ on $\Ser(M)$ is given by \eqref{eq:Ser-action}. For $a \in L$, $m \in M$ and $\xi \in \Yone_2(M)$, we have
  \begin{align*}
    \langle \varphi_M^{-1}(a \bullet m), \xi \rangle
    & = \langle \alpha_H, \nu_L(a)_{(-1)} \rangle \langle \varphi_M^{-1}(\nu_L(a)_{(0)} m), \xi \rangle \\
    & = \langle \alpha_H, \nu_L(a)_{(-1)} \rangle \langle \xi, \Lambda \otimes \nu_L(a)_{(0)} m \rangle \\
    {}^{\eqref{eq:right-modular}}
    & = \langle \xi, \nu_L(a)_{(-1)} \Lambda \otimes \nu_L(a)_{(0)} m) \rangle \\
    {}^{\eqref{eq:Y2-M-action}}
    & = \langle \xi \cdot a, \Lambda \otimes m \rangle
      = \langle \varphi_M^{-1}(m), \xi \cdot a \rangle
      = \langle a \cdot \varphi_M^{-1}(m), \xi \rangle.
  \end{align*}
  The proof is done.
\end{proof}

We export the twisted left $\mathcal{C}$-module structure of $\Ser_2$ to $\Ser$ through the isomorphism given by the above lemma. For $X \in {}_H \Mod$ and $M \in {}_L \Mod$, the resulting twisted left $\mathcal{C}$-module structure of $\Ser$ is the following composition:
\begin{subequations}
  \newcommand{\xarr}[1]{\xrightarrow{\makebox[5em]{$\scriptstyle #1$}}}
  \begin{align}
    \label{eq:pf-main-thm-1}
    \Ser(X \catactl M)
    & \xarr{\varphi_{X \catactl M}^{-1}}
      \Hom_H(\Yone_2(X \catactl M), \bfk) \\
    \label{eq:pf-main-thm-2}
    & \xarr{\eqref{eq:Y2-M-C-mod-str}}
      \Hom_H(\Yone_2(M) \otimes X^*, \bfk) \\
    \label{eq:pf-main-thm-3}
    & \xarr{\eqref{eq:equivariant-EW-bicomod-alg-2}}
      \Hom_H(\Yone_2(M), X^{**}) \\
    \label{eq:pf-main-thm-4}
    & \xarr{\eqref{eq:Hom-mod-str-inverse}}
      X^{**} \catactl \Hom_H(\Yone_2(M), \bfk) \\
    \label{eq:pf-main-thm-5}
    & \xarr{\id_{X^{**}} \catactl \varphi_M}
      X^{**} \catactl \Ser(M).
  \end{align}
\end{subequations}  

\begin{proof}[Proof of Theorem~\ref{thm:main-theorem} (1)]
  We compute the composition of \eqref{eq:pf-main-thm-1}--\eqref{eq:pf-main-thm-5}. Let $\{ x_i \}$ be a basis of $X$, and let $\{ x^i \}$ be the dual basis of $X^*$. For $x \in X$ and $m \in M$, we compute
  \begin{align*}
    x \otimes m
    \xmapsto{\makebox[3em]{\scriptsize \eqref{eq:pf-main-thm-1}}}
    & \, [f \mapsto \langle f, \Lambda \otimes x \otimes m \rangle ] \\
    \xmapsto{\makebox[3em]{\scriptsize \eqref{eq:pf-main-thm-2}}}
    & \, [ f \otimes x^* \mapsto \langle f, \Lambda_{(1)} \otimes m \rangle
      \langle x^*, S(\Lambda_{(2)}) x\rangle ] \\
    \xmapsto{\makebox[3em]{\scriptsize \eqref{eq:pf-main-thm-3}}}
    & \, [ f \mapsto \langle f, \Lambda_{(1)} \otimes m \rangle
      \langle x^i, S(\Lambda_{(2)}) x \rangle \Phi_X(x_i) ] \\
    = & \, [ f \mapsto \langle f, \Lambda_{(1)} \otimes m \rangle \Phi_X(S(\Lambda_{(2)}) x) ]
        \qquad =: (\ast 1).
  \end{align*}
  Set $Y = \Yone_2(M) \in {}_H^H\Mod_L$ and consider the left $\mathcal{C}$-module structure
  \begin{equation*}
    \theta_{V,W}: V \catactl \Hom_H(Y, W) \to \Hom_{H}(Y, V \otimes W)
    \quad (V, W \in \mathcal{C})
  \end{equation*}
  of $\Hom_{H}(Y, -)$ given in Lemma~\ref{lem:Hom-mod-str}.
  Let $\{ w_k \}$ be a basis of $H \otimes M$, and let $\{ w^k \}$ be the dual basis of $Y = (H \otimes M)^*$. For simplicity of notation, we write
  \begin{equation*}
    T_g(y) = S(y_{(-1)}) g(y_{(0)})
    \quad (g \in \Hom_{\bfk}(Y, X^{**}), y \in Y).
  \end{equation*}
  Then, for $g \in \Hom_{H}(Y, X^{**})$, we have
  \begin{equation}
    \label{eq:main-thm-proof-1}
    \theta_{X^{**},\bfk}^{-1}(g)
    \mathop{=}^{\eqref{eq:Hom-mod-str-inverse}}
    \Phi_{X}(x_i) \otimes (\Phi_{X^*}(x^i) \circ T_g)
    = T_g(w^k) \otimes \Phi_{H \otimes M}(w_k)
  \end{equation}
  in $X^{**} \otimes Y^*$.
  We should be careful to deal with this formula.
  By our convention, the sum over $k$ is implicit in \eqref{eq:main-thm-proof-1}.
  Although the sum $T_g(w^k) \otimes \Phi_{H \otimes M}(w_k)$ certainly belongs to the vector space $X^{**} \otimes \Hom_H(H \catactl M, \bfk)$, each term is only an element of the vector space $X^{**} \otimes Y^*$.

  By \eqref{eq:main-thm-proof-1}, we have
  \begin{align*}
    (\ast 1) \xmapsto{\makebox[3em]{\scriptsize \eqref{eq:pf-main-thm-4}}}
    & \, \text{$T_{g}(w^k) \otimes \Phi_{H \otimes M}(w_k)$ with $g = (\ast 1)$} \\
    = & \, S(w^k_{(-1)}) \langle w^k_{(0)}, \Lambda_{(1)} \otimes m \rangle
      \Phi_X( S(\Lambda_{(2)}) x )
      \otimes \Phi_{H \otimes M}(w_k) \\[2pt]
    {}^{\eqref{eq:claim-Y2-M-coact-1}} =
    & \, S(S(\Lambda_{(1)}) b_{i(-1)}) \Phi_X( S(\Lambda_{(3)}) x ) \\[-2pt]
    & \qquad \qquad \otimes \langle \lambda_L, b_{i(0)} \rangle \langle w^k, a^i_{(-1)} \Lambda_{(2)}
      \otimes a^i_{(0)} m \rangle \Phi_{H \otimes M}(w_k) \\[2pt]
    = \,
    & \Phi_X( S^3(b_{i(-1)}) S^4(\Lambda_{(1)}) S(\Lambda_{(3)}) x ) \\[-2pt]
    & \qquad \qquad \otimes \langle \lambda_L, b_{i(0)} \rangle
      \Phi_{H \otimes M}(a^i_{(-1)} \Lambda_{(2)} \otimes a^i_{(0)} m)
      \qquad =: (\ast 2)
  \end{align*}
  in the vector space $X^{**} \otimes Y^*$.

  We shall apply the map \eqref{eq:pf-main-thm-5} to $(\ast 2)$.
  By the reason noted after \eqref{eq:main-thm-proof-1}, the term `$\Phi_{H \otimes M}(a^i_{(-1)} \Lambda_{(2)} \otimes a^i_{(0)} m)$' may not belong to the source of $\varphi_M$. To resolve this technical problem, we introduce the linear map
  \begin{align*}
    \widetilde{\varphi}_{X,M}: X^{**} \otimes Y^*
    & \to X^{**} \otimes M, \\
    \beta \otimes \gamma
    & \mapsto \beta \otimes ((\lambda_H \otimes \id_M)\Phi_{H \otimes M}^{-1}(\gamma)).
  \end{align*}
  It is obvious that the equation $\widetilde{\varphi}_{X,M}(T) = (\id_{X^{**}} \otimes \varphi_M)(T)$ holds whenever $T$ belongs to the source of $\id_{X^{**}} \otimes \varphi_M$. Thanks to this observation, we continue the computation as follows:
  \begin{align*}
    (\ast 2) \xmapsto{\makebox[3em]{\scriptsize \eqref{eq:pf-main-thm-5}}}    
    & \, \Phi_X( S^3(b_{i(-1)}) S^4(\Lambda_{(1)}) S(\Lambda_{(3)}) x ) \\[-2pt]
    & \qquad \qquad \otimes \langle \lambda_L, b_{i(0)} \rangle \langle \lambda_H, a^i_{(-1)} \Lambda_{(2)} \rangle a^i_{(0)} m
      \qquad =: (\ast 3).
  \end{align*}
  For all elements $a, h \in H$, we have
  \begin{gather}
    \label{eq:pf-main-thm-7}
    h_{(1)} \langle \lambda_H, a h_{(2)} \rangle
    = S(a_{(1)}) a_{(2)} h_{(1)} \langle \lambda_H, a_{(3)} h_{(2)} \rangle
    \mathop{=}^{\eqref{eq:distinguished-grouplike}}
    S(a_{(1)}) g_H \langle \lambda_H, a_{(2)} h \rangle, \\
    \label{eq:pf-main-thm-8}
    \langle \lambda_H, a h_{(1)} \rangle S(h_{(2)})
    = \langle \lambda_H, a_{(1)} h_{(1)} \rangle S(S^{-1}(a_{(3)}) a_{(2)}h_{(2)})
    = \langle \lambda_H, a_{(1)} h \rangle a_{(2)}.
  \end{gather}
  For simplicity, we set $\overline{\alpha}_H = \alpha_H \circ S$. 
  Radford's $S^4$-formula \cite[Proposition 6]{MR0407069} states that the following equation holds:
  \begin{equation}
    \label{eq:Radford-S4}
    S^4(h) = g_H (\overline{\alpha}_H \rightharpoonup h \leftharpoonup \alpha_H)  g_H^{-1}
    \quad (h \in H).
  \end{equation}
  Now, for $a \in H$, we have
  \begin{align*}
    & S^4(\Lambda_{(1)}) S(\Lambda_{(3)}) \langle \lambda_H, a \Lambda_{(2)} \rangle \\
    {}^{\eqref{eq:pf-main-thm-7}} =
    & \, S^4(S(a_{(1)}) g_H)) \langle \lambda_H, a_{(2)} \Lambda_{(1)} \rangle S(\Lambda_{(2)}) \\
    {}^{\eqref{eq:pf-main-thm-8}} =
    & \, S^4(S(a_{(1)}) g_H) \langle \lambda_H, a_{(2)} \Lambda \rangle a_{(3)} \\
    {}^{\eqref{eq:right-modular}} =
    & \, S^5(a_{(1)}) g_H \langle \alpha_H, a_{(2)} \rangle \langle \lambda_H, \Lambda \rangle a_{(3)} \\
    {}^{\eqref{eq:Radford-S4}} =
    & \, S(g_H \langle \alpha_H, a_{(1)} \rangle a_{(2)} \langle \overline{\alpha}_H, a_{(3)} \rangle g_H^{-1}) g_H \langle \alpha_H, a_{(4)} \rangle a_{(5)} \\
    = & \, g_H \langle \alpha_H, a_{(1)} \rangle S(a_{(2)}) a_{(3)}
        = g_H \langle \alpha_H, a \rangle.
  \end{align*}
  This reduces the expression $(\ast 3)$ to the form as stated. The proof is done.
\end{proof}

\subsection{Proof of Theorem~\ref{thm:main-theorem} (2)}

We keep the notation of the previous subsection.
We have chosen \eqref{eq:YM-right-adj} as a right adjoint of $\Yone(M)$ for $M \in {}_L \Mod$. Under this choice of right adjoints, the counit of the adjunction $\Yone(M) \dashv \Yone(M)^{\radj}$ is given by the following composition:
\begin{subequations}
  \newcommand{\xarr}[1]{\xrightarrow{\makebox[6em]{$\scriptstyle #1$}}}
  \begin{align*}
    \varepsilon_M(N) : \Yone(M) \Yone(M)^{\radj}(N) = \,
    & \Hom_L(H \catactl M, \Hom_H(\Yone_2(M), N)) \\
    \xarr{\text{Lemma~\ref{lem:Yoneda-module-2}}} \,
    & \Hom_L(H \catactl M, \Hom_H(\Yone_1(M), N)) \\
    \xarr{\text{\eqref{eq:dagger-eval-inv}}} \,
    & \Yone_1(M) \otimes_L \Hom_H(\Yone_1(M), N) \\
    \xarr{\mathtt{eval}_{N}} \, & N,
  \end{align*}
\end{subequations}
where $\mathtt{eval}_N(y \otimes_L f) = f(y)$.
By Lemma~\ref{lem:rel-Ser-pairing} (b), the isomorphism
\begin{equation*}
  \phi'_{M,N}: \iHom(N, \Ser_2(M)) \to \iHom(M, N)^*
  \quad (M, N \in {}_L \Mod)
\end{equation*}
of $\mathcal{C}$-bimodule functors is induced by the pairing given by the composition
\begin{gather*}
  \iHom(N, \Ser_2(M)) \otimes \iHom(M, N)
  \xrightarrow{\quad \icomp \quad}
  \iHom(M, \Ser_2(M))
  \xrightarrow{\quad \varepsilon_M(\bfk) \quad} \bfk.
\end{gather*}
In actuality, we would like to use the functor $\Ser$ of  Theorem~\ref{thm:main-theorem} as a relative Serre functor rather than $\Ser_2$. Thus we define
\begin{equation*}
  \phi^{}_{M,N} := \phi'_{M,N} \circ \iHom(N, \varphi^{-1}_M):
  \iHom(N, \Ser(M)) \to \iHom(M, N)^*,
\end{equation*}
where $\varphi_M: \Ser_2(M) \to \Ser(M)$ is the isomorphism given by Lemma~\ref{lem:Y2-M-varphi}.
Then $(\Ser, \phi)$ is a relative Serre functor of ${}_L \Mod$.

\begin{proof}[Proof of Theorem~\ref{thm:main-theorem} (2)]
  By the naturality of $\icomp$, we have
  \begin{align*}
    & \eval_{\iHom(M, N)} \circ (\phi_{M,N}^{} \otimes \id_{\iHom(M, N)}) \\
    & = \varepsilon_{M}(\bfk) \circ \icomp_{M, N, \Ser_2(M)} \circ (\iHom(N, \varphi^{-1}_M) \otimes \id_{\iHom(M, N)}) \\
    & = \varepsilon_{M}(\bfk) \circ \iHom(M, \varphi^{-1}_M) \circ \icomp_{M,N,\Ser(M)}.
  \end{align*}
  Thus, to prove Theorem~\ref{thm:main-theorem} (2), it is sufficient to prove
  \begin{equation*}
    \varepsilon_{M}(\bfk) \circ \iHom(M, \varphi_M^{-1}) = \itrace_M
    \quad (M \in {}_L \Mod),
  \end{equation*}
  where $\itrace_M$ is the map defined by \eqref{eq:iHom-trace}.
  Let $t_1, \dotsc, t_k$ and $t^1, \dotsc, t^k$ as in the statement of Theorem~\ref{thm:main-theorem} (2). For $\xi \in \iHom(M, \Ser(M))$, we have
  \begin{gather*}
    (\varepsilon_M(\bfk) \iHom(M, \varphi_M^{-1}))(\xi)
    = \mathtt{eval}_{\bfk}(t^i \otimes_L (\psi_M^{*}\varphi_M^{-1}\xi)(t_i))
    = \langle \psi_M^{*}\varphi_M^{-1}\xi(t_i), t^i \rangle \\
    \mathop{=}^{\eqref{eq:Y2-M-varphi-2}}
    \langle \Phi_{H \otimes M}(\Lambda \otimes \xi(t_i)), \psi_M(t^i) \rangle
    \mathop{=}^{\eqref{eq:Y2-M-psi-1}}
    \langle \lambda_L t^i, \Lambda \otimes \xi(t_i) \rangle
    \mathop{=}^{\eqref{eq:iHom-trace}}
    \underline{\mathrm{tr}}_M(\xi). \qedhere
  \end{gather*}
\end{proof}

\subsection{Pivotal structures of comodule algebras}

A {\em pivotal element} of $H$ is a grouplike element $g \in H$ such that the equation $g h g^{-1} = S^2(h)$ holds for all $h \in H$. Given a pivotal element $g \in H$, we define a natural transformation
\begin{equation*}
  \mathfrak{p}_X: X \to X^{**},
  \quad x \mapsto \Phi_X(g x)
\end{equation*}
for $X \in {}_{H}\Mod$. Then $\mathfrak{p} = \{ \mathfrak{p}_X \}$ is a pivotal structure of ${}_H \Mod$. It is well-known that this construction establishes a bijection between the set of pivotal elements of $H$ and the set of pivotal structures of ${}_H \Mod$.

Now we fix a pivotal element $g_{\piv} \in H$. Let $L$ be an exact left $H$-comodule algebra. We fix a Frobenius form $\lambda_L$ on $L$ and define the algebra automorphism $\nu_L'$ of $L$ and the element $\mathfrak{T} \in H \otimes L$ by~\eqref{eq:twisted-Nakayama} and \eqref{eq:Serre-mod-str-element}, respectively.

\begin{definition}
  A {\em pivotal element} of $L$ respecting $g_{\piv}$ is an invertible element $\widetilde{g} \in L$ satisfying the following two equations:
  \begin{gather}
    \label{eq:comod-alg-pivot-1}
    \mathfrak{T} \cdot \delta_L(\widetilde{g}) = g_{\piv}^{} \otimes \widetilde{g}, \\
    \label{eq:comod-alg-pivot-2}
    \widetilde{g} a \widetilde{g}{}^{-1} = \nu_L'(a) \quad (a \in L).
  \end{gather}
\end{definition}

By Theorem~\ref{thm:main-theorem}, it is routine to check:

\begin{theorem}
  \label{thm:comod-alg-pivotal}
  Given a pivotal element $\widetilde{g}$ of $L$, we define
  \begin{equation*}
    \widetilde{\mathfrak{p}}_M: M \to \Ser(M),
    \quad m \mapsto \widetilde{g} \cdot m
    \quad (M \in {}_L \Mod),
  \end{equation*}
  where $\Ser$ is the relative Serre functor given in Theorem~\ref{thm:main-theorem}. Then $\widetilde{\mathfrak{p}} = \{ \widetilde{\mathfrak{p}}_M \}$ is a pivotal structure of ${}_L \Mod$ respecting the pivotal structure of $\mathcal{C}$ associated to the pivotal element $g_{\piv}$. This establishes a one-to-one correspondence between the set of pivotal elements of $L$ and the set of pivotal structures on ${}_L \Mod$.
\end{theorem}

\subsection{Grouplike cointegrals on comodule algebras}
\label{subsec:gl-coint}

Our formula~\eqref{eq:Ser-mod-structure} of the twisted module structure of a relative Serre functor may not be convenient for practical use. We now remark that the formula reduces to a simpler form if the Frobenius form on $L$ is a `grouplike cointegral' in the following sense:

\begin{definition}[Kasprzak \cite{2018arXiv181007114K}]
  \label{def:gl-cointegral}
  Let $H$ be a Hopf algebra, and let $L$ be a left $H$-comodule algebra. A {\em grouplike cointegral} on $L$ is a pair $(g, \lambda)$ consisting of a grouplike element $g \in H$ and a linear form $\lambda: L \to \bfk$ such that the equation
  \begin{equation}
    \label{eq:g-cointegral}
    a_{(-1)} \langle \lambda, a_{(0)} \rangle = \langle \lambda, a \rangle g
  \end{equation}
  holds for all elements $a \in L$. If $(g, \lambda)$ is a grouplike cointegral on $L$ in this sense, then $\lambda$ is said to be a {\em $g$-cointegral} on $L$. A $g$-cointegral on $L$ is said to be {\em non-degenerate} if it is a Frobenius form on $L$.
\end{definition}

For example, a non-zero right cointegral on $H$ is a non-degenerate $g_H$-cointegral, where $g_H$ is the distinguished grouplike element of $H$ defined by \eqref{eq:distinguished-grouplike}.

\begin{theorem}
  \label{thm:main-theorem-2}
  Let $H$, $L$, $\lambda_L$ and $\mathfrak{T}$ be as in Theorem~\ref{thm:main-theorem}. If the Frobenius form $\lambda_L$ is a $g_L$-cointegral for some grouplike element $g_L \in H$, then we have
  \begin{equation}
    \label{eq:thm-main-theorem-2}
    \mathfrak{T} = g_L^{-1} g_H^{} \otimes 1_L.
  \end{equation}
  Thus, for $X \in \mathcal{C}$ and $M \in {}_{L} \Mod$, the formula \eqref{eq:Ser-mod-structure} of the twisted module structure of the relative Serre functor reduces to the following form:
  \begin{equation*}
    \Ser(X \catactl M) \to X^{**} \catactl \Ser(M),
    \quad x \otimes m \mapsto \Phi_X(g_L^{-1} g_H^{} x) \otimes m.
  \end{equation*}
\end{theorem}
\begin{proof}
  Using the same notation as in  Theorem~\ref{thm:main-theorem}, we compute:
  \begin{gather*}
    \mathfrak{T}
    \mathop{=}^{\eqref{eq:g-cointegral}}
    \langle \lambda_L, b_{i} \rangle S^3(g_L^{}) g_H^{}
    \otimes \langle \alpha_H, a^i_{(-1)} \rangle a^i_{(0)}
    \mathop{=}^{\eqref{eq:Frob-str-L-1}}
    g_L^{-1} g_H^{} \otimes 1_L^{}. \qedhere
  \end{gather*}
\end{proof}
Suppose that a pivotal element $g_{\piv} \in H$ is given.
By the above theorem, a pivotal element of $L$ respecting $g_{\piv}$ is the same thing as an invertible element $\widetilde{g} \in L$ such that the equations
\begin{equation}
  \label{eq:comod-alg-pivot-3}
  \delta_L(\widetilde{g}) = g_H^{-1} g_L^{} g_{\piv}^{} \otimes \widetilde{g}
  \quad \text{and} \quad
  \widetilde{g} a \widetilde{g}{}^{-1} = \nu_L'(a)
\end{equation}
hold for all $a \in L$. We note that $g_{\piv}^{}$ belongs to the center of the group $G(H)$ of grouplike elements of $H$ since $S^2(g) = g$ for all $g \in G(H)$. Radford's $S^4$-formula~\eqref{eq:Radford-S4} implies that $g_H^{}$ is also central in $G(H)$. Thus the appearance order of $g_H^{-1}$, $g_L^{}$ and $g_{\piv}^{}$ in \eqref{eq:comod-alg-pivot-3} does not matter.

In general, a left $H$-comodule algebra does not admit a non-degenerate grouplike cointegral. However, as the following two propositions demonstrate, there is a large class of $H$ and $L$ for which the assumption of the above theorem is satisfied.

\begin{proposition}
  Suppose that the Hopf algebra $H$ is either cosemisimple or pointed. Then every finite-dimensional left $H$-comodule algebra admits a non-zero grouplike cointegral.
\end{proposition}
\begin{proof}
  Let $L$ be a finite-dimensional left $H$-comodule algebra, and let $g \in H$ be a grouplike element. We note that the space of $g$-cointegrals on $L$ is identified with the space of left $H$-comodule maps from $L$ to $\bfk g$.

  Suppose that $H$ is cosemisimple. The unit map $u: \bfk \to H$ is $H$-colinear. By the cosemisimplicity, $u$ splits in the category of left $H$-comodules. Namely, there is a left $H$-comodule map $\lambda : H \to \bfk$ such that $\lambda u = \id_{\bfk}$. By the above remark, $\lambda$ is a non-zero $1$-cointegral on $L$.

  We now consider the case where $H$ is pointed. Let $\lambda : L \to V$ be a surjective left $H$-comodule map from $L$ to a simple left $H$-comodule $V$ (which exists since $L$ is finite-dimensional). Since $H$ is pointed, $V \cong \bfk g$ for some grouplike element $g \in H$. Thus $\lambda$ may be viewed as a non-zero $g$-cointegral.
\end{proof}

\begin{proposition}[{\it cf}. {\cite[Theorem 4.11]{2018arXiv181007114K}}]
  We assume that the base field $\bfk$ is algebraically closed.
  Let $L$ be a finite-dimensional $H$-simple left $H$-comodule algebra, and let $g \in H$ be a grouplike element. Then the dimension of the space of $g$-cointegrals on $L$ is 0 or 1. If a non-zero $g$-cointegral on $L$ exists, then it is non-degenerate.
\end{proposition}
\begin{proof}
  We consider the category ${}^H\Mod_L$ of right $L$-modules in ${}^H\Mod$.
  Since $L$ is a left $L$-module in ${}^H\Mod$, its right dual ${}^*L$ is an object of ${}^H\Mod_L$ in a canonical way. Let $I^g$ denote the space of $g$-cointegrals on $L$. Then there are isomorphisms
  \begin{equation}
    \label{eq:relative-Hopf-and-cointegrals}
    I^g
    \cong {}^H\Mod(L, \bfk g)
    \cong {}^H\Mod_L(L, M)
    \quad (M := \bfk g \otimes {}^* \! L)
  \end{equation}
  of vector spaces. Since every object of ${}^H\Mod_L$ is free as a right $L$-module \cite[Theorem 4.2]{MR2286047}, both $L$ and $M$ are simple objects in ${}^H\Mod_L$. Thus, by Schur's lemma, the dimension of $I^g$ is 0 or 1.
  The remaining part of this lemma follows from Schur's lemma and that \eqref{eq:relative-Hopf-and-cointegrals} restricts to a bijection between the set of non-degenerate $g$-cointegrals and the set of isomorphisms from $L$ to $M$ in ${}^H\Mod_L$.
\end{proof}

\subsection{Hopf subalgebras}

In this subsection, we consider the case where $L$ is a Hopf subalgebra of $H$.
We first introduce the following notation:
Let $X$ be a Hopf algebra with antipode $S_X$.
Given a grouplike element $g \in X$ and an algebra map $\beta: X \to \bfk$, we define
\begin{equation*}
  \inner(g)(x) = g x g^{-1}
  \quad \text{and} \quad
  \inner(\beta)(x) = \beta \rightharpoonup x \leftharpoonup \overline{\beta}
\end{equation*}
for $x \in X$, where $\overline{\beta} = \beta \circ S_X$ is the convolution inverse of $\beta$.

Let $L$ be a Hopf subalgebra of $H$ and view it as a left $H$-comodule algebra by the comultiplication of $H$. By the Nichols-Zoeller theorem \cite{MR1243637}, we see that $L$ is in fact an exact left $H$-comodule algebra.
As a Frobenius form on $L$, we take a non-zero right cointegral $\lambda_L$ on $L$. Then we have
\begin{equation*}
  \lambda_L(y x)
  = \lambda_L(S^2(x \leftharpoonup \overline{\alpha}_L) y)
  \quad (x, y \in L)
\end{equation*}
by \cite[Theorem 3]{MR1265853}.
In other words, the Nakayama automorphism of $L$ is given by $\nu_L(x) = S^2(x \leftharpoonup \overline{\alpha}_L)$ for $x \in L$. Thus the automorphism $\nu_L'$, defined by \eqref{eq:twisted-Nakayama}, is given by the following formula:
\begin{equation}
  \label{eq:twisted-Nakayama-Hopf-sub}
  \nu_L'(x)
  = \langle \overline{\alpha}_L, x_{(1)} \rangle \langle \alpha_H, S^2(x_{(2)}) \rangle S^2(x_{(3)})
  = S^2(x \leftharpoonup \overline{\alpha}_L \alpha_H)
  \quad (x \in L).
\end{equation}

Suppose that a pivotal element $g_{\piv}^{}$ of $H$ is given. We investigate when the left $\mathcal{C}$-module category ${}_L \Mod$ admits a pivotal structure (which is a notion different from a pivotal structure of the finite tensor category ${}_L\Mod$). The defining formula \eqref{eq:distinguished-grouplike} of $g_L$ says that $\lambda_L$ is a $g_L$-cointegral of $L$ in the sense of Definition~\ref{def:gl-cointegral}.
Hence the formula \eqref{eq:thm-main-theorem-2} for the element $\mathfrak{T} \in H \otimes L$ is available.

\begin{theorem}
  \label{thm:Hopf-subalg-pivotal}
  For a Hopf subalgebra $L$ of $H$, the following are equivalent:
  \begin{enumerate}
  \item The left $\mathcal{C}$-module category ${}_L \Mod$ admits a pivotal structure respecting the pivotal structure of $\mathcal{C}$ corresponding to $g_{\piv}$.
  \item $g_{\piv}^{} g_H^{-1} \in L$ and $\alpha_H |_L = \alpha_L$.
  \end{enumerate}
  If these equivalent conditions hold, then $g_H^{-1} g_L^{} g_{\piv}^{} \in L$ is a pivotal element of the left $H$-comodule algebra $L$.
\end{theorem}
\begin{proof}
  Suppose that ${}_L \Mod$ admits a pivotal structure. Then, by Theorem~\ref{thm:comod-alg-pivotal}, the left $H$-comodule algebra $L$ has a pivotal element $\widetilde{g} \in L$. Since $\widetilde{g}$ is invertible, we have $\varepsilon(\widetilde{g}) \ne 0$. Thus, by renormalizing $\widetilde{g}$, we may assume $\varepsilon(\widetilde{g}) = 1$. By \eqref{eq:comod-alg-pivot-1} and \eqref{eq:thm-main-theorem-2}, we have
  $\widetilde{g} = (\id_H \otimes \varepsilon) \Delta(\widetilde{g}) = g_H^{-1} g_L^{} g_{\piv}^{}$. Hence,
  \begin{equation*}
    g_H^{-1} g_{\piv}^{} = \widetilde{g} g_L^{-1} \in L.
  \end{equation*}
  By Radford's $S^4$-formula \eqref{eq:Radford-S4}, the inner automorphism induced by $\widetilde{g} = g_H^{-1} g_L^{} g_{\piv}^{}$ is computed as follows:
  \begin{align*}
    \inner(\widetilde{g})
    = \inner(\overline{\alpha}_H) \circ S^{-4} \circ S^4 \circ \inner(\alpha_L) \circ S^2
    = S^2 \circ \inner(\overline{\alpha}_H \alpha_L).
  \end{align*}
  By \eqref{eq:comod-alg-pivot-2} and \eqref{eq:twisted-Nakayama-Hopf-sub}, we have $\overline{\alpha}_H \alpha_L \rightharpoonup x = x$ for all $x \in L$. This is equivalent to the equation $\alpha_H|_L = \alpha_L$. Thus (2) holds.

  Suppose, conversely, that (2) holds. Then $\widetilde{g} := g_H^{-1} g_L^{} g_{\piv}^{}$ is an invertible element of $L$. Since $g_H^{}$, $g_L^{}$ and $g_{\piv}$ are grouplike elements, $\widetilde{g}$ satisfies \eqref{eq:comod-alg-pivot-1}. By Radford's $S^4$-formula \eqref{eq:Radford-S4}, we verify \eqref{eq:comod-alg-pivot-2} in the same way as above. Thus $\widetilde{g}$ is a pivotal element. Hence ${}_L \Mod$ admits a pivotal structure. The proof is done.
\end{proof}

By applying Theorem~\ref{thm:Hopf-subalg-pivotal} to $L = \bfk$, we obtain:

\begin{corollary}
  \label{cor:Hopf-subalg-pivotal}
  The left $\mathcal{C}$-module category ${}_{\bfk}\Mod$ admits a pivotal structure if and only if $g_{\piv}^{} = g_{H}^{}$. 
\end{corollary}

We close this section by giving some remarks and (counter)examples.

\begin{remark}
  The condition $\alpha_H|_L = \alpha_L$ appearing in Theorem~\ref{thm:Hopf-subalg-pivotal} is equivalent to that $L \subset H$ is a Frobenius extension of algebras \cite{MR1401518}.
\end{remark}

\begin{remark}
  Let $\mathcal{M}$ be an exact left $\mathcal{C}$-module category. By Theorem~\ref{thm:Hopf-subalg-pivotal}, we see that whether $\mathcal{M}$ admits a pivotal structure depends on the choice of a pivotal structure of $\mathcal{C}$. For example, let $G$ be a finite group, and let $z$ be a central element of $G$. Then $z$ is a pivotal element of $\bfk G$. For a subgroup $F$ of $G$, the left ${}_{\bfk G}\Mod$-module category ${}_{\bfk F}\Mod$ admits a pivotal structure if and only if $z \in F$.
\end{remark}

\begin{remark}
  If the condition $g_{\piv} = g_H$ of Corollary \ref{cor:Hopf-subalg-pivotal} is satisfied, then we have
  \begin{equation*}
    S^2 = S^{-2} \circ S^4 = \inner(g_{\piv}^{-1}) \circ \inner(g_H^{}) \circ \inner(\overline{\alpha}_H)
    = \inner(\overline{\alpha}_H)
  \end{equation*}
  by Radford's $S^4$-formula \eqref{eq:Radford-S4}. Thus, when $H$ is unimodular ({\it i.e.}, $\alpha_H$ is equal to the counit of $H$), the condition $g_{\piv} = g_H^{}$ implies $S^2 = \id_H$. In characteristic zero, it also follows from the Larson-Radford theorem that $H$ is semisimple.
\end{remark}

\begin{remark}
  In \cite{MR3943747}, the notion of a {\em matched pivotal structure} for a module category is introduced and studied.
  According to \cite[Example 2.7]{MR3943747}, there is a pivotal structure of $\mathcal{C}$ matched to the left $\mathcal{C}$-module category ${}_{\bfk}\Mod$ if and only if $S^2 = \id_H$. Thus there are finite-dimensional pivotal Hopf algebras $H$ such that the left $\mathcal{C}$-module category ${}_{\bfk}\Mod$ has a pivotal structure in our sense but no matched pivotal structures in the sense of \cite{MR3943747}. The Taft algebra $T(\omega)$ is such an example (see Subsection \ref{subsec:taft-algebra}).
  We note that the relationship between a matched pivotal structure and a pivotal structure in our sense is discussed in \cite[Appendix A]{MR3943747}.
  Further examples of where two types of structures do not agree are also given in there.
\end{remark}

\begin{remark}
  \label{rem:non-pivotal-module-but}
  Theorem~\ref{thm:Hopf-subalg-pivotal} yields examples of non-pivotal exact left $\mathcal{C}$-module categories $\mathcal{M}$ such that $\mathcal{C}_{\mathcal{M}}^*$ is a pivotal finite tensor category: Let $L$ be a finite-dimensional non-unimodular Hopf algebra admitting a pivotal element $g_{\piv}^{}$ ({\it e.g.}, the Taft algebra). We regard $L$ as a Hopf subalgebra of the Drinfeld double $H := D(L)$. Then $g_{\piv}^{} \in L$ is also a pivotal element of $H$. Since the Drinfeld double is unimodular \cite{MR1220770}, we have $\alpha_H|_{L} = \varepsilon \ne \alpha_L$. Thus, by Theorem~\ref{thm:Hopf-subalg-pivotal}, the left $\mathcal{C}$-module category $\mathcal{M} := {}_{L}\Mod$ is not pivotal (no matter how we change the pivotal element of $H$). Nevertheless, since there are equivalences
  $\mathcal{C}^*_{\mathcal{M}} \approx \mathcal{Z}(\mathcal{M})^*_{\mathcal{M}} \approx \mathcal{M} \boxtimes \mathcal{M}^{\rev}$
  of tensor categories \cite{MR3242743}, $\mathcal{C}^*_{\mathcal{M}}$ admits a pivotal structure.
\end{remark}

\begin{remark}
  \label{rem:non-symmetric-iEnd}
  Let $\mathcal{D}$ be a finite tensor category. A Frobenius algebra $(A, \lambda)$ in $\mathcal{D}$ has a canonical coalgebra structure with counit $\lambda$ \cite{MR2500035}. A (normalized) {\em special} Frobenius algebra in $\mathcal{D}$ is a Frobenius algebra $(A, \lambda)$ in $\mathcal{C}$ such that $\mu \Delta = \id_A$ and $\lambda u = \beta \id_{\unitobj}$ for some $\beta \in \bfk^{\times}$, where $\mu$, $\Delta$ and $u$ are the multiplication, the comultiplication and the unit of $A$, respectively.

  By Theorem~\ref{thm:sym-Frobenius-alg}, the algebra $\iEnd(M)$ in $\mathcal{C}$ is a symmetric Frobenius algebra if $M$ is an object of a pivotal left $\mathcal{C}$-module category. We remark that $\iEnd(M)$ is not a special Frobenius algebra in general. To give such an example, we consider the left $\mathcal{C}$-module category $\mathcal{M} := {}_{\bfk}\Mod$. Then $A := \iEnd(\bfk)$ is an algebra in $\mathcal{C}$ such that the functor $\iHom(\bfk, -): \mathcal{M} \to \mathcal{C}_A$ is an equivalence. By Theorem~\ref{thm:Frobenius-alg} and $\Ser(\bfk) = \bfk$, the algebra $A$ is Frobenius. Thus we have isomorphisms
  \begin{equation*}
    \Hom_{\mathcal{C}}(A, \unitobj)
    \cong \Hom_{\mathcal{C}_A}(A, A^*)
    \cong \Hom_{\mathcal{C}_A}(A, A)
    \cong \Hom_{\mathcal{M}}(\bfk, \bfk) \cong \bfk.
  \end{equation*}
  This means that a Frobenius form on $A$ is unique up to scalar multiple. Now we fix a non-zero right integral $\Lambda \in H$. By Theorem~\ref{thm:main-theorem} (2) and the above argument, every Frobenius form on $A$ is a scalar multiple of the map
  \begin{equation*}
    t: A \to \bfk, \quad \xi \mapsto \xi(\Lambda)
    \quad (\xi \in A),
  \end{equation*}
  where we have identified $A = \Hom_{\bfk}(H \otimes \bfk, \bfk)$ with $H^*$ as a vector space. Thus, if $A$ is a special Frobenius algebra in $\mathcal{C}$, then we have
  \begin{equation*}
    \varepsilon(\Lambda) \id_{\bfk} = t \circ \icoev_{\bfk, \bfk} \ne 0
  \end{equation*}
  and therefore $H$ is semisimple by Maschke's theorem \cite{MR1243637}.

  By the above discussion, we now have the following example:
  Let $G$ be a finite group.
  We choose $g_{\piv} := 1$ as a pivotal element of $H = \bfk G$. Then ${}_{\bfk}\Mod$ is a pivotal left $\mathcal{C}$-module category by Corollary \ref{cor:Hopf-subalg-pivotal}. If $p := \mathrm{char}(\bfk) > 0$ and $p$ divides $|G|$, then the symmetric Frobenius algebra $A = \iEnd(\bfk)$ in $\mathcal{C}$ considered in the above is not a special Frobenius algebra in $\mathcal{C}$ since $\bfk G$ is not semisimple.
\end{remark}

\section{Examples}
\label{sec:examples}

In this section, for some examples of exact comodule algebras given in \cite{MR2678630}, we give non-degenerate grouplike cointegrals, compute the Nakayama automorphism and determine whether it admits a pivotal element.

Throughout this section, $\bfk$ denotes an algebraically closed field of characteristic zero.
For a finite-dimensional Hopf algebra $H$ (over $\bfk$), we use the symbols $\alpha_H \in H^*$ and $g_H \in H$ to mean the elements defined by~\eqref{eq:right-modular} and \eqref{eq:distinguished-grouplike}, respectively. Given $q \in \bfk^{\times}$ and $n \in \mathbb{Z}_{\ge 0}$, we set $(0)_q = 0$ and $(n)_q := 1 + q + \dotsb + q^{n-1}$ for $n > 0$. The following $q$-binomial formula will be used extensively: If $X$ and $Y$ are elements of an algebra satisfying $Y X = q X Y$, then the equation
\begin{equation*}
  (X + Y)^n = \sum_{i = 0}^n
  \binom{n}{i}_{\!\! q} X^i Y^{n - i}
\end{equation*}
holds for all $n \in \mathbb{Z}_{\ge 0}$, where
\begin{equation*}
  \binom{n}{i}_{\!\! q} = \frac{(n)!_q}{(i)!_q (n-i)!_q},
  \quad (0)!_q = 1
  \quad \text{and}
  \quad (n)!_q = (n)_{q} \cdot (n-1)!_q
  \quad (n > 0).
\end{equation*}

\subsection{Taft algebra}
\label{subsec:taft-algebra}

We fix an integer $N > 1$ and let $\omega$ be a primitive $N$-th root of unity. The Taft algebra $T(\omega)$ is the Hopf algebra over $\bfk$ generated, as an algebra, by $g$ and $x$ subject to the relations
\begin{equation*}
  x^N = 0,
  \quad g^N = 1
  \quad \text{and}
  \quad g x = \omega x g.
\end{equation*}
The Hopf algebra structure of $T(\omega)$ is determined by
\begin{equation*}
  \Delta(g) = g \otimes g,
  \quad \Delta(x) = x \otimes 1 + g \otimes x.
\end{equation*}
The antipode of $T(\omega)$ is given by $S(g) = g^{-1}$ and $S(x) = -g^{-1} x$. Thus the element $g_{\piv} := g^{-1}$ is a pivotal element of $T(\omega)$. If $h$ is another pivotal element of $T(\omega)$, then $h^{-1} g_{\piv}$ is a central grouplike element of $T(\omega)$. Since $T(\omega)$ has no non-trivial central grouplike element, $g_{\piv}$ is in fact a unique pivotal element of $T(\omega)$.

The set $\{ x^i g^j \mid i, j = 0, 1, \dotsc, N - 1 \}$ is a basis of $T(\omega)$.
By the $q$-binomial formula, the comultiplication is given by
\begin{equation*}
  \Delta(x^{r} g^s)
  = (x \otimes 1 + g \otimes x)^r (g \otimes g)^s
  = \sum_{i = 0}^r \binom{r}{i}_{\!\!\omega} x^{i} g^{r - i + s} \otimes x^{r - i} g^{s}
\end{equation*}
for $r \in \mathbb{Z}_{\ge 0}$ and $s \in \mathbb{Z}$. By this equation, we see that the linear map
\begin{equation*}
  \lambda_{T(\omega)}: T(\omega) \to \bfk,
  \quad
  \lambda_{T(\omega)}(x^r g^s) = \delta_{r, N-1} \delta_{s, 0}
  \quad (r, s = 0, \dotsc, N - 1)
\end{equation*}
is a right cointegral on $T(\omega)$. The element $\Lambda := \sum_{i = 1}^{N - 1} x^{N-1} g^{i}$ is a non-zero right integral in $T(\omega)$ such that $\langle \lambda_{T(\omega)}, \Lambda \rangle = 1$. Thus we have
\begin{equation*}
  g_{T(\omega)} = g^{-1},
  \quad
  \alpha_{T(\omega)}(g) = \omega^{-1}
  \quad \text{and} \quad \alpha_{T(\omega)}(x) = 0.
\end{equation*}

For a divisor $d \mid N$ and an element $\xi \in \bfk$ of the base field, we set $m = N/d$ and introduce the following algebras:
\begin{itemize}
\item $L_0(d) = \bfk \langle G \mid G^d = 1 \rangle$.
\item $L_1(d; \xi) = \bfk \langle G, X \mid G^d = 1, X^N = \xi 1, G X = \omega^m X G \rangle$.
\end{itemize}
They are left $T(\omega)$-comodule algebras with the coaction determined by
\begin{equation*}
  \delta(G) = g^{m} \otimes G,
  \quad \delta(X) = x \otimes 1 + g \otimes X.
\end{equation*}
In this subsection, for each comodule algebra $L$ in the above list, we classify non-degenerate grouplike cointegrals, compute the associated Nakayama automorphism and check whether $L$ has a pivotal element.

\begin{remark}
  According to \cite[Proposition 8.3]{MR2678630}, every indecomposable exact left module category over ${}_{T(\omega)}\Mod$ is equivalent to ${}_L \Mod$, where $L$ is one of the comodule algebras $L_0(d)$ or $L_1(d; \xi)$ introduced in the above. 
  The comodule algebras listed in \cite[Proposition 8.3]{MR2678630} are expressed as $\bfk C_{d} = L_0(d)$ and $\mathcal{A}(d, \xi) = L_1(d; \xi)$ in our notation.
\end{remark}

\begin{remark}
  The comodule algebras $L_0(d)$ and $L_1(d; 0)$ can be regarded as left coideal subalgebras of $T(\omega)$ by the algebra map given by $G \mapsto g^{m}$ and $X \mapsto x$.
  The comodule algebra $L_1(0; \xi)$ can also be regarded as a left coideal subalgebra of $T(\omega)$ by $X \mapsto \zeta g + x$, where $\zeta$ is an $N$-th root of $\xi$.
  We note that grouplike cointegrals on coideal subalgebras of $T(\omega)$ were classified in \cite{2018arXiv181007114K}.
\end{remark}

\subsubsection{The comodule algebra $L_0(d)$}

We fix a divisor $d$ of $N$ and set $L = L_0(d)$. Let $\lambda: L \to \bfk$ be a non-zero grouplike cointegral. Then the image of $(\id_{T(\omega)} \otimes \lambda) \delta$ is spanned by a single grouplike element of $T(\omega)$. Since
\begin{equation*}
  (\id_{T(\omega)} \otimes \lambda) \delta(G^r) = \langle \lambda, G^r \rangle g^{m r}
  \quad (r \in \mathbb{Z}/d\mathbb{Z}),
\end{equation*}
we have $\langle \lambda, G^r \rangle = 0$ for all but one element $r \in \mathbb{Z}/d\mathbb{Z}$.
Taking this observation into account, for $s \in \mathbb{Z}/d\mathbb{Z}$, we define the linear map $\lambda_s: L \to \bfk$ by
\begin{equation}
  \label{eq:L0d-lambda-s}
  \lambda_s(G^{r}) = \Kdelta_{r s}
  \quad (r \in \mathbb{Z}/d\mathbb{Z}),
\end{equation}
where $\Kdelta$ denotes the Kronecker delta. Then $\lambda_s$ is a $g^{m s}$-cointegral on $L$ and every non-zero grouplike cointegral on $L$ is a scalar multiple of $\lambda_s$ for some $s$.

Now we fix $s \in \mathbb{Z}/d\mathbb{Z}$. It is easy to see that $\lambda_s$ is a Frobenius form on $L$. Since $L$ is commutative, the Nakayama automorphism of $L$ is the identity. Let $\nu'_s$ be the automorphism on $L$ given by~\eqref{eq:twisted-Nakayama} with $\lambda_L = \lambda_s$. Explicitly, we  have
\begin{equation*}
  \nu'_s(G^r) = \langle \alpha_{T(\omega)}, g^{m r} \rangle G^r = \omega^{-m r} G^r
  \quad (r \in \mathbb{Z}/d\mathbb{Z}).
\end{equation*}

It is easy to see that $\nu'_s$ is inner if and only if $d = 1$.
Thus, if $d > 1$, a relative Serre functor of ${}_{L}\Mod$ is not isomorphic to the identity functor.
On the other hand, if $d = 1$, then $1 \in L$ is a pivotal element.

\subsubsection{The comodule algebra $L_1(d; \xi)$}
\label{subsubsec:Taft-comod-alg-L1dxi}

We classify grouplike cointegrals on the comodule algebra $L := L_1(d; \xi)$. We first note that the set
\begin{equation*}
  \{ X^r G^s \mid r = 0, \dotsc, N - 1; s = 0, \dotsc, d - 1 \}
\end{equation*}
is a basis of $L$.
Let $\lambda$ be a non-zero grouplike cointegral on $L$ and set
\begin{equation*}
  w_{r,s} := (\id_{T(\omega)} \otimes \lambda) \delta(X^r G^s)
  = \sum_{i = 0}^r \binom{r}{i}_{\!\! \omega} \langle \lambda, X^{r - i} G^{s} \rangle x^{i} g^{r - i + m s}.
\end{equation*}
Since the image of the map $(\id_{T(\omega)} \otimes \lambda) \delta$ is spanned by a single grouplike element, the coefficient of $x^i g^k$ for $i > 0$ in the above sum must be zero. Thus $\langle \lambda, X^{r} G^s \rangle = 0$ for all integers $r < N - 1$. Since then $w_{N - 1, s} = \langle \lambda, X^{N-1} G^s \rangle g^{m s - 1}$, we have $\langle \lambda, X^{N-1} G^s \rangle = 0$ for all but one element $s \in \mathbb{Z}/d\mathbb{Z}$.
Following the above discussion, for $t \in \mathbb{Z}/d\mathbb{Z}$, we define the linear map $\lambda_t: L \to \bfk$ by
\begin{equation}
  \label{eq:Taft-comod-alg-L1dxi-coint}
  \lambda_t(X^r G^{s}) = \Kdelta_{r, N - 1} \Kdelta_{s, t}
  \quad (r \in \{ 0, \dotsc, N - 1 \}, s \in \mathbb{Z}/d\mathbb{Z}).
\end{equation}
Then $\lambda_t$ is a $g^{m t - 1}$-cointegral on $L$ and every non-zero grouplike cointegral on $L$ is a scalar multiple of $\lambda_t$ for some $t$ by the above argument.

Now we fix $t \in \mathbb{Z}/d\mathbb{Z}$. For $r, r' \in \{ 0, \dotsc, N - 1 \}$ and $s, s' \in \mathbb{Z}/d\mathbb{Z}$, we have
\begin{equation*}
  \langle \lambda_t, X^r G^s \cdot X^{r'} G^{s'} \rangle
  = \omega^{m s r'} \langle \lambda_t, X^r X^{r'} G^{s} G^{s'} \rangle
  = \omega^{m s r'} \Kdelta_{r + r', N - 1} \Kdelta_{s + s', t}.
\end{equation*}
Thus $\lambda_t$ is non-degenerate. Let $\nu_t$ be the Nakayama automorphism of $L$ with respect to $\lambda_t$, and let $\nu'_t$ be the automorphism on $L$ given by~\eqref{eq:twisted-Nakayama}. By the above computation, we have
\begin{equation}
  \label{eq:Taft-comod-alg-L1dxi-Nakayama}
  \nu_t(X) = \omega^{m t} X, \quad
  \nu_t(G) = \omega^m G, \quad
  \nu'_t(X) = \omega^{m t - 1} X, \quad
  \nu'_t(G) = G.
\end{equation}
Now we check whether $\nu_t$ and $\nu_t'$ are inner and, when $\nu_t'$ is inner, examine whether there is a pivotal element by case-by-case analysis as follows:
\begin{enumerate}
\item Suppose that $\xi = 0$ and $1 < d < N$. There is an algebra map $\varepsilon: L \to \bfk$ such that $\varepsilon(X) = 0$ and $\varepsilon(G) = 1$. Since $\varepsilon \circ \nu_t \ne \varepsilon$, the automorphism $\nu_t$ is not inner. The automorphism $\nu_t'$ is not inner as well. To see this, we assume that $\nu_t'$ is implemented by $a \in L^{\times}$. Then we have $\nu_t'(a) = a a a^{-1} = a$. By considering the eigenspaces of $\nu_t'$, we see that the element $a$ is of the form $a = \sum_{i = 0}^{d - 1} c_i G^i$ for some scalars $c_i \in \bfk$. Since $a X = \nu_t'(X) a = \omega^{m t -1} X a$, we have $c_i \omega^{m i} = c_i \omega^{m t -1}$ for all $i$. Thus $c_i = 0$ for all $i$. This contradicts to the assumption that $a$ is invertible. Hence $\nu_t'$ is not inner.
\item Suppose that $\xi = 0$ and $d = 1$. Then $\nu_t = \id_L$ and, in particular, it is an inner automorphism.
  On the other hand, $\nu_t'$ is not an inner automorphism by the same argument as (1).
\item Suppose that $\xi = 0$ and $d = N$. Then $\nu_t$ is not an inner automorphism by the same reason as (1).
  On the other hand, $\nu'_t$ is an inner automorphism implemented by $G^{t - 1} \in L$, which is actually a pivotal element.
\item Suppose that $\xi \ne 0$ and $d < N$. We note that $X \in L$ is invertible if this is the case and $\nu_t$ is an inner automorphism implemented by $G^m X^{-1}$. On the other hand, $\nu'_t$ is not inner. To see this, we fix an $N$-th root $\zeta$ of $\xi$ and define a left $L$-module $V$ as follows: As a vector space, $V$ has a basis $\{ v_i \}_{i \in \mathbb{Z}/d\mathbb{Z}}$. The action of $L$ on $V$ is determined by
  \begin{equation*}
    X \cdot v_i = \zeta v_{i + 1}
    \quad \text{and} \quad
    G \cdot v_i = \omega^{m i} v_i
    \quad (i \in \mathbb{Z}/d\mathbb{Z}).
  \end{equation*}
  Let $V'$ be the left $L$-module obtained from $V$ by twisting the action of $L$ by $\nu'_t$. Then $X^{d}$ acts on $V$ and $V'$ as scalars $\zeta^d$ and $\omega^{-d} \zeta^d$, respectively. Thus $V \not \cong V'$ as left $L$-modules. Therefore $\nu_t'$ is not inner.
\item Suppose that $\xi \ne 0$ and $d = N$. Then $\nu_t$ is an inner automorphism by the same reason as the case (4). Unlike that case, $\nu_t'$ is the inner automorphism implemented by a pivotal element $G^{t-1} \in L$.
\end{enumerate}

\subsubsection{Summary of the results}

Table~\ref{tab:Taft-comod-alg-rel-Serre} summarizes our results. The first column of the table shows whether the Nakayama functor $\Nak := L^* \otimes_L (-)$ on ${}_L \Mod$ is isomorphic to the identity functor (or, equivalently, whether $\nu_L$ is inner). The second column shows whether a relative Serre functor for ${}_L \Mod$ is isomorphic to the identity functor  (or, equivalently, whether $\nu'_L$ is inner). The third column shows whether the left ${}_{T(\omega)}\Mod$-module category ${}_L \Mod$ admits a pivotal structure.

\begin{table}
  \centering
  \def\arraystretch{1.25}
  \begin{tabular}{ll|ccc}
    \hline
    $L$
    & & $\Nak \cong \id?$ & $\Ser \cong \id?$ & Pivotal? \\ \hline
    $L_0(d)$
    & $d = 1$ & Yes & Yes & Yes \\ \cline{2-5}
    & $d > 1$ & Yes & No & No \\ \hline
    $L_1(d; \xi)$
    & $\xi = 0$ and $d = 1$ & Yes & No & No \\ \cline{2-5}
    & $\xi = 0$ and $1 < d < N$ & No & No & No \\ \cline{2-5}
    & $\xi = 0$ and $d = N$ & No & Yes & Yes \\ \cline{2-5}
    & $\xi \ne 0$ and $d < N$ & Yes & No & No \\ \cline{2-5}
    & $\xi \ne 0$ and $d = N$ & Yes & Yes & Yes \\ \hline
  \end{tabular}
  \medskip
  \caption{Results for exact $T(\omega)$-comodule algebras}
  \label{tab:Taft-comod-alg-rel-Serre}
\end{table}

\subsection{Book Hopf algebra}

We fix an integer $N > 1$ and let $\omega$ be a primitive $N$-th root of unity. The Hopf algebra $\mathcal{H}(1, \omega)$ discussed in \cite[Section 8]{MR2678630} is generated, as an algebra, by $x$, $y$ and $g$ subject to the relations
\begin{equation*}
  g^N = 1,
  \quad g x = \omega x g,
  \quad g y = \omega^{-1} y g,
  \quad x y = \omega y x
  \quad \text{and}
  \quad x^N = y^N = 0.
\end{equation*}
The Hopf algebra structure of $H$ is determined by
\begin{equation*}
  \Delta(g) = g \otimes g,
  \quad \Delta(x) = x \otimes 1 + g^{-1} \otimes x,
  \quad \Delta(y) = y \otimes 1 + g^{-1} \otimes y.
\end{equation*}
The antipode is given by $S(g) = g^{-1}$, $S(x) = -g x$ and $S(y) = -gy$. It is easy to see that $g \in H$ is a pivotal element. As in the case of the Taft algebra, the Hopf algebra $H$ has no non-trivial central grouplike elements and thus $g_{\piv} := g$ is a unique pivotal element.

\begin{remark}
  Further information on $\mathcal{H}(1, \omega)$ is found in \cite{MR1489920}.
  We note that $\mathcal{H}(1, \omega)$ is isomorphic to the Hopf algebra $\mathbf{h}(\omega, -1)$ called the {\em book Hopf algebra} in \cite{MR1489920}.
\end{remark}

From now on, we write $H = \mathcal{H}(1, \omega)$ for brevity.
The set $\{ x^r y^s g^t \}_{r, s, t = 0, \dotsc, N - 1}$ is a basis of $H$. By the $q$-binomial formula, the comultiplication of the elements of the basis are given by
\begin{equation*}
  \Delta(x^r y^s g^t)
  = \sum_{i = 0}^r \sum_{j = 0}^s
  \binom{r}{i}_{\!\! \omega^{-1}} \binom{s}{j}_{\!\! \omega} \omega^{j(i-r)}
  x^i y^j g^{i + j - r - s + t} \otimes x^{r - i} y^{s - j} g^{t}
\end{equation*}
for $r, s, t \in \{ 0, \dotsc, N - 1 \}$. Thus the linear map
\begin{equation*}
  \lambda_{H}: H \to \bfk, \quad
  \lambda_{H}(x^r y^s g^t) = \Kdelta_{r,N-1} \Kdelta_{s,N-1} \Kdelta_{t,0}
\end{equation*}
is a right cointegral on $H$. The element $\Lambda = \sum_{i = 0}^{N-1} x^{N-1} y^{N-1} g^{i}$ is a two-sided integral in $H$ such that $\lambda_{H}(\Lambda) = 1$. Thus we have
\begin{equation*}
  \alpha_{H} = \varepsilon
  \quad \text{and} \quad
  g_{H} = g^2.
\end{equation*}

For a divisor $d$ of $N$, elements $\xi$ and $\mu$ of $\bfk$, and a pair $(a, b)$ of elements of $\bfk$ such that $(a, b) \ne (0, 0)$, we set $m = N/d$ and introduce the following algebras:
\begin{itemize}
\item $L_0(d) = \bfk \langle G \mid G^d = 1 \rangle$.
\item $L_1(d; \xi) = \bfk \langle G, X \mid G^d = 1, G X = \omega^m X G, X^N = \xi 1 \rangle$.
\item $L_2(d; \xi) = \bfk \langle G, X \mid G^d = 1, G Y = \omega^{-m} Y G, Y^N = \xi 1 \rangle$.
\item $L_3(a, b; \xi) = \bfk \langle W \mid W^N = \xi 1 \rangle$.
\item $L_4(d; \xi, \mu)$ is the algebra generated by $G$, $X$ and $Y$ subject to
  \begin{gather*}
    G^{d} = 1, \quad X^N = \xi 1, \quad Y^N = \mu 1, \\
    G X = \omega^{m} X G, \quad G Y = \omega^{-m} Y G, \quad X Y = \omega Y X.
  \end{gather*}
\item $L_4(N; \xi, \mu, \eta)$ is the algebra defined by the same generators and the same relations as $L_4(N; \xi, \mu)$ but with the relation $X Y = \omega Y X$ replaced with
  \begin{equation*}
    X Y = \omega Y X + \eta G^{N - 2}.
  \end{equation*}
\end{itemize}
These are left $H$-comodule algebras by the coaction determined by
\begin{align*}
  \delta(X) & = x \otimes 1 + g^{-1} \otimes X,
  & \delta(Y) & = y \otimes 1 + g^{-1} \otimes Y, \\
  \delta(G) & = g^{m} \otimes G,
  & \delta(W) & = (a x + b y) \otimes 1 + g^{-1} \otimes W.
\end{align*}
Let $L$ be one of exact left $H$-comodule algebras listed in the above, and let $\nu_L$ be the Nakayama automorphism of $L$. Since $\alpha_{H}$ is identical to the counit, the automorphism $\nu_L'$ defined by~\eqref{eq:twisted-Nakayama} coincides with $\nu_L$. Thus a relative Serre functor of ${}_L \Mod$ is isomorphic to the identity functor if and only if $\nu_L$ is inner, or, equivalently, $L$ admits a symmetric Frobenius form.

\begin{remark}
  In our notation, the comodule algebras given by Mombelli \cite[Subsection 8.3]{MR2678630} are expressed as follows:
  \begin{gather*}
    \bfk C_d = L_0(d),
    \ \mathcal{A}_0(d, \xi) = L_1(d; \xi),
    \ \mathcal{A}_1(d, \xi) = L_2(d; \xi),
    \ \mathcal{A}(\xi, a) = L_3(a, 1; \xi), \\
    \mathcal{D}(d, \xi, \mu) = L_4(d; \xi, \mu),
    \quad \mathcal{D}_1(\xi, \mu, \eta) = L_4(N; \xi, \mu, \eta).
  \end{gather*}
  According to \cite[Proposition 8.9]{MR2678630}, the above comodule algebras are exact. Furthermore, if $N$ is odd, then every indecomposable exact left module category over ${}_{H}\Mod$ is equivalent to ${}_L \Mod$ for some $L$ in the above list (if $N$ is even, two more families of comodule algebras arise in addition to the above list).
  When the parameters $\xi$, $\mu$ and $\eta$ are zero, the above comodule algebras are identified with coideal subalgebras of $H$ via $G \mapsto g^{m}$, $X \mapsto x$, $Y \mapsto y$, $W \mapsto a x + b y$. We arranged Mombelli's list so that the corresponding coideal subalgebra is easy to recognized. Because of this rearrangement, our list has some duplicates, such as
  \begin{gather*}
    L_1(1; \xi) = L_3(1, 0; \xi),
    \quad L_2(1; \xi) = L_3(0, 1; \xi),
    \quad L_4(N; \xi, \mu, 0) = L_4(N; \xi, \mu),
  \end{gather*}
  unlike the original list of Mombelli. Furthermore, $L_3(a, b; \xi)$ and $L_3(a', b'; \xi)$ are isomorphic as $H$-comodule algebras if $a b' = b a'$.
\end{remark}

\subsubsection{The comodule algebra $L_0(d)$}

By the same way as the case of the Taft algebra, we see that $L := L_0(d)$ has a $g^{m s}$-cointegral $\lambda_s: L_0(d) \to \bfk$ ($s \in \mathbb{Z}/d\mathbb{Z})$ given by the same formula as \eqref{eq:L0d-lambda-s} and every non-zero group-like cointegral on $L$ is a scalar multiple of $\lambda_s$ for some $s$.

Now we fix $s \in \mathbb{Z}/d\mathbb{Z}$. Since $L$ is commutative, the Nakayama automorphism associated to $\lambda_s$ is the identity map. Thus a relative Serre functor of ${}_L\Mod$ is isomorphic to the identity functor.
The comodule algebra $L$ can be regarded as a Hopf subalgebra of $H$ by the algebra map $G \mapsto g^m$. By Theorem~\ref{thm:Hopf-subalg-pivotal}, $L$ admits a pivotal element if and only if $g \in L$ or, equivalently, $d = N$.

\subsubsection{The comodule algebras $L_1(d; \xi)$ and $L_2(d; \xi)$}

We first consider the comodule algebra $L = L_1(d; \xi)$. The set $\{ X^r G^s \mid r = 0, \dotsc, N - 1; s = 0, \dotsc, d - 1 \}$ is a basis of $L$. By the $q$-binomial formula, we have
\begin{equation*}
  \delta_L(X^r G^s) = \sum_{i = 0}^r \binom{r}{i}_{\!\! \omega^{-1}} x^{i} g^{- r + i + m s} \otimes X^{r - i} G^{s}.
\end{equation*}
for $0 \le r < N$ and $s \in \mathbb{Z}/d\mathbb{Z}$.
For $t \in \mathbb{Z}/d\mathbb{Z}$, we define $\lambda_t: L \to \bfk$ by \eqref{eq:Taft-comod-alg-L1dxi-coint}. By the same argument as in \S\ref{subsubsec:Taft-comod-alg-L1dxi}, we see that $\lambda_t$ is a $g^{m t + 1}$-cointegral and every non-zero grouplike cointegral on $L$ is a scalar multiple of $\lambda_t$ for some $t$.

Now we fix $t \in \mathbb{Z}/d\mathbb{Z}$. Again by the same argument as in \S\ref{subsubsec:Taft-comod-alg-L1dxi}, we see that $\lambda_t$ is non-degenerate, the associated Nakayama automorphism $\nu_t$ is given by the same formula as \eqref{eq:Taft-comod-alg-L1dxi-Nakayama}, and $\nu_t$ is inner if and only if $\xi \ne 0$ or $(\xi, d) = (0, 1)$.

\begin{claim}
  \label{claim:book-comod-alg-L1dxi-pivot}
  $L$ has a pivotal element if and only if $d = 1$.
\end{claim}
\begin{proof}
  If $d = 1$, then $1 \in L$ is a pivotal element of $L$. If, conversely, $L$ has a pivotal element $\widetilde{g}$, then $\bfk \widetilde{g}$ is a one-dimensional subcomodule of $L$ by \eqref{eq:comod-alg-pivot-1}. Thus $\widetilde{g}$ is a non-zero scalar multiple of $G^s$ for some $s$. By~\eqref{eq:comod-alg-pivot-2}, we have $\omega^{m} G = \nu_t(G) = \widetilde{g} \, G \, \widetilde{g}^{-1} = G$. This implies $d = 1$. The proof is done.
\end{proof}

The result is same for $L_2(d; \xi)$. Namely, the Nakayama automorphism of $L_2(d; \xi)$ is inner if and only if $\xi \ne 0$ or $(\xi, d) = (0, 1)$, and $L_2(d; \xi)$ has a pivotal element if and only if $d = 1$.

\subsubsection{The comodule algebra $L_3(a, b; \xi)$}

The set $\{ W^r \mid 0 \le r < N \}$ is a basis of the comodule algebra $L := L_3(a, b; \xi)$. We define $A, B, C \in H \otimes L$ by $A = b y \otimes 1$, $B = g^{-1} \otimes W$ and $C = a x \otimes 1$, respectively. Then we have $(B + C) A = \omega A(B + C)$ and $C B = \omega B C$. Thus, by two-fold use of the $q$-binomial formula, we have
\begin{gather*}
  \delta(W^r)
  = \sum_{j = 0}^{r} \binom{r}{j}_{\!\!\omega} A^j (B + C)^{r - j}
  = \sum_{i + j \le r} \omega^{i^2 - i r}
  \binom{r}{i,j}_{\!\!\omega} a^i b^j x^i y^j g^{i + j - r} \otimes W^{r - i - j}
\end{gather*}
for $0 \le r < N$, where
\begin{equation*}
  \binom{r}{i,j}_{\!\!\omega} = \frac{(r)!_{\omega}}{(i)!_{\omega} (j)!_{\omega} (r - i - j)!_{\omega}}.
\end{equation*}
By this formula, we see that the linear map $\lambda: L \to \bfk$ defined by $\lambda(W^r) = \Kdelta_{r, N-1}$ ($0 \le r < N$) is a $g$-cointegral on $L$ and every grouplike cointegral on $L$ is a scalar multiple of $\lambda$. Since $L$ is commutative, the Nakayama automorphism of $L$ is the identity. The element $1 \in L$ is a pivotal element of $L$.

\subsubsection{The comodule algebra $L_4(d; \xi, \mu)$}
\label{subsubsec:book-comod-alg-L4d-xi-mu}

The set
\begin{equation*}
  \{ X^r Y^s G^t \mid r, s = 0, \dotsc, N - 1; t = 0, \dotsc, d - 1 \}
\end{equation*}
is a basis $L := L_4(d; \xi, \mu)$.
By the $q$-binomial formula, the coaction is given by
\begin{equation}
  \label{eq:book-comod-alg-L4d-coact}
  \delta(X^r Y^s G^t) = \sum_{i = 0}^r \sum_{j = 1}^s \omega^{j(i - r)}
  \binom{r}{i}_{\!\! \omega^{-1}} \! \binom{s}{j}_{\!\! \omega}
  x^{i} y^j g^{i + j - r - s + m t} \otimes X^{r - i} Y^{s - j} G^{t}
\end{equation}
for $0 \le r, s < N$ and $0 \le t < d$. For $u \in \mathbb{Z}/d\mathbb{Z}$, we define a linear map
\begin{equation}
  \label{eq:book-comod-alg-L4d-coint}
  \lambda_u: L \to \bfk,
  \quad
  \lambda_u(X^r Y^{s} G^t) = \Kdelta_{r, N - 1} \Kdelta_{s, N - 1} \Kdelta_{t, u}
  \quad (0 \le r, s < N, t \in \mathbb{Z}/d\mathbb{Z}).
\end{equation}
The map $\lambda_u$ is a $g^{m u + 2}$-cointegral on $L$ and every grouplike cointegral on $L$ is a scalar multiple of $\lambda_u$ for some $u$. It is easy to see that $\lambda_u$ is non-degenerate. The associated Nakayama automorphism $\nu_u$ is given by
\begin{equation}
  \label{eq:book-comod-alg-L4d-Nakayama}
  \nu_u(X) = \omega^{m u + 1} X,
  \quad \nu_u(Y) = \omega^{-m u - 1} Y
  \quad \text{and} \quad \nu_u(G) = G.
\end{equation}
Since $L_3(N; \xi, \mu) = L_4(N; \xi, \mu, 0)$, we only consider the case where $d < N$ (the case where $d = N$ will discussed in \S\ref{subsubsec:L4(N;xi,mu,eta)}).

\begin{claim}
  Suppose that $d < N$. Then the automorphism $\nu_u$ is inner if and only if both $\xi$ and $\mu$ are non-zero.
\end{claim}
\begin{proof}
  Since $\nu_u(a) = G^u \nu_0(a) G^{-u}$ for all $a \in L$, it is enough to determine when $\nu_0$ is inner. We verify the claim by the case-by-case analysis as follows:
  \begin{enumerate}
  \item We first consider the case where $\mu = 0$.
    Suppose that the automorphism $\nu_0$ is implemented by $a \in L^{\times}$. Then $\nu_0(a) = a$.
    We note that $L$ is a free right module over $\bfk \langle G \rangle$ with basis
    $\{ X^r Y^s \mid r, s = 0, \cdots, N - 1 \}$.
    By considering the eigenspace decomposition of $\nu_0$, we have
    \begin{equation*}
      a = a_0 + X Y a_1 + X^2 Y^2 a_2 + \dotsb + X^{N-1} Y^{N-1} a_{N-1}
    \end{equation*}
    for some $a_i \in \bfk \langle G \rangle$. Since $a X = \nu_0(X) a = \omega X a$, we have $a_0 = \omega a_0$. Thus $a_0 = 0$. Hence $a Y^{N-1} = 0$. This contradicts to our assumption that $a$ is invertible. Therefore $\nu_0$ is not an inner automorphism.
  \item If $\xi = 0$, then we find that $\nu_0$ is not an inner automorphism by the same way as the case (1).
  \item If $\xi \ne 0$ and $\mu \ne 0$, then $\nu_0$ is implemented by $(X Y)^{-1}$. \qedhere
  \end{enumerate}
\end{proof}

\begin{claim}
  $L$ has no pivotal elements if $d < N$.
\end{claim}
\begin{proof}
  If $L$ has a pivotal element $\widetilde{g}$, then it must be a scalar multiple of $G^s$ for some $s$ ({\it cf}. the proof of Claim~\ref{claim:book-comod-alg-L1dxi-pivot}), however, the inner automorphism implemented by $G^s$ is not equal to $\nu_u$. Thus $L$ has no pivotal element. The proof is done.
\end{proof}

\subsubsection{The comodule algebra $L_4(N; \xi, \mu, \eta)$}
\label{subsubsec:L4(N;xi,mu,eta)}

The algebra $L := L_4(N; \xi, \mu, \eta)$ is filtered by $\deg(X) = \deg(Y) = 1$ and $\deg(G) = 0$. The associated graded algebra is identified with $L_4(N; 0, 0, 0)$. From this, we easily see that the set
\begin{equation*}
  \{ X^r Y^s G^t \mid r, s, t = 0, \dotsc, N - 1 \}
\end{equation*}
is a basis of $L$. The coaction is given by the same formula as \eqref{eq:book-comod-alg-L4d-coact}.
For $u \in \mathbb{Z}/N\mathbb{Z}$, we define the linear map $\lambda_u : L \to \bfk$ by the same formula as  \eqref{eq:book-comod-alg-L4d-coint}.
The map $\lambda_u$ is a $g^{u + 2}$-cointegral on $L$ and every non-zero grouplike cointegral on $L$ is a scalar multiple of $\lambda_u$ for some $u$.

Now we fix $u \in \mathbb{Z}/N\mathbb{Z}$. We note that $\lambda_u$ vanishes on the subspace of elements of degree $<2N-2$.
Thanks to this, it is not difficult to check that $\lambda_u$ is a Frobenius form on $L$ and the associated Nakayama automorphism is given by the same formula as~\eqref{eq:book-comod-alg-L4d-Nakayama} with $d = N$. It is straightforward to see that the element $G^{u + 1}$ is a pivotal element of $L$.

\subsubsection{Summary of the results}

Table~\ref{tab:book-comod-alg-rel-Serre} summarizes our results for exact comodule algebras over $H$. The first column of the table shows whether a relative Serre functor for ${}_L \Mod$ is isomorphic to the identity functor (it is equivalent to whether the Nakayama automorphism of $L$ is inner, since $H$ is unimodular). The second column shows whether ${}_L \Mod$ is pivotal. Unlike the case of the Taft algebra, there are some examples of $L$ such that a relative Serre functor of ${}_L \Mod$ is isomorphic to the identity functor but ${}_L \Mod$ admits no pivotal structure.

\begin{table}
  \centering
  \def\arraystretch{1.25}
  \begin{tabular}{ll|c|c}
    \hline
    $L$ & & $\Ser \cong \id$? & Pivotal? \\ \hline
    $L_0(d)$
        & $d < N$ & Yes & No \\ \cline{2-4}
        & $d = N$ & Yes & Yes \\ \hline
    $L_1(d; \xi)$, $L_2(d; \xi)$
        & $d = 1$ & Yes & Yes \\ \cline{2-4}
        & $d > 1$ and $\xi \ne 0$ & Yes & No \\ \cline{2-4}
        & $d > 1$ and $\xi   = 0$ & No  & No \\ \hline
    $L_3(a, b; \xi)$ & & Yes & Yes \\ \hline
    $L_4(d; \xi, \mu)$
        & $d < N$ and $\xi \mu \ne 0$ & Yes & No \\ \cline{2-4}
        & $d < N$ and $\xi \mu   = 0$ & No  & No \\ \cline{2-4}
        & $d = N$ & Yes & Yes \\ \hline
    $L_4(N; \xi, \mu, \eta)$
        & & Yes & Yes \\ \hline
  \end{tabular}
  \medskip
  \caption{Results for exact $H$-comodule algebras}
  \label{tab:book-comod-alg-rel-Serre}
\end{table}

\def\cprime{$'$}

\end{document}